\documentclass[reqno]{amsart}
\usepackage[utf8]{inputenc} 
\usepackage{amssymb}
\usepackage[initials]{amsrefs} 
\usepackage{array}    
\usepackage{graphicx} 
\usepackage[counterclockwise]{rotating} 
\usepackage{verbatim} 
\usepackage{tikz}
\usepackage{tikz-cd}  

\usepackage[font=small]{caption}

\usepackage{hyperref} 

\newcommand{\silenceable}[1]
 {#1}  

\usetikzlibrary{trees,math}

\tikzset{
 every picture/.style={
 inner sep=1}
}

\graphicspath{{Figures_Gindikin/}}

\numberwithin{subsection}{section}
\numberwithin{equation}{section}
\numberwithin{table}   {section}
\numberwithin{figure}  {section}

\theoremstyle{plain}
 \newtheorem {theorem}    {Theorem}[section]
 \newtheorem {proposition}[theorem]{Proposition}
 \newtheorem {lemma}      [theorem]{Lemma}
 \newtheorem {corollary}  [theorem]{Corollary}
\theoremstyle{definition}
 \newtheorem {definition} [theorem]{Definition}

\theoremstyle{remark}
 \newtheorem {remark}     [theorem]{Remark}

\newcommand{\missarg}{\,\cdot\,}
\newcommand{\abs}    [1]{\left\lvert#1\right\rvert}
\newcommand{\ceiling}[1]     {\left\lceil#1\right\rceil}
\newcommand{\floor}  [1]    {\left\lfloor#1\right\rfloor}

\newcommand{\mathI}{\mathbb{I}}
\newcommand{\mathC}{\mathbb{C}}
\newcommand{\mathN}{\mathbb{N}}
\newcommand{\mathZ}{\mathbb{Z}}
\newcommand{\mathR}{\mathbb{R}}

\newcommand{\calG} {\mathcal{G}}    
\newcommand{\calW} {\mathcal{W}(\gamma)}	
\newcommand{\calR} {\mathcal{R}}	
\DeclareMathOperator{\Poisstr}{\mathcal K}    
\newcommand{\Poiss}{K}      				  

\newcommand{\qmean}{\mathfrak q}
\newcommand{\qFF}{\mathfrak B(\gamma)} 
\newcommand{\tildeqFF}{\widetilde{\mathfrak B}(\gamma)} 
\newcommand{\newGamma}{\mathcal Q(\gamma)} 
\newcommand{\graphGamma}{\mathcal G} 
\newcommand{\pcrit}{p_{\text{\rm crit}}}

\newcommand{\barra}{\,|\, }

\newcommand{\theroot}{\mathcal R(\gamma)}

\DeclareMathOperator{\Aut} {Aut}
\DeclareMathOperator{\dist}{dist}

\DeclareMathOperator{\sign}{sign}

\DeclareMathOperator{\Real} {{\rm Re}}
\DeclareMathOperator{\Imag }{{\rm Im}}

\DeclareMathOperator{\arccosh }{{\mathrm {arccosh}}}

\DeclareMathOperator{\spectrum}{sp}

\def\psiphi{\phi}

\begin{document}

\title
 [Spectrum of the Laplacian in semi-homogeneous trees]
{Spherical functions and spectrum of the Laplacian on semi-homogeneous trees}
\author
[E.     Casadio~Tarabusi]
{Enrico Casadio~Tarabusi}
\address
{Dipartimento di Matematica ``G. Castelnuovo''\\
 Sapienza Universit\`a di Roma\\
 Piazzale A. Moro 5\\00185~Roma\\Italy}
\email{enrico.casadio\_tarabusi@mat.uniroma1.it}
\author
[M.     ~A. Picardello]
{Massimo~A. Picardello}
\address
{Dipartimento di Matematica\\
 Universit\`a di Roma ``Tor Vergata''\\
 Via della Ricerca Scientifica\\00133~Roma\\Italy}
\email{picard@mat.uniroma2.it}



\begin{abstract}
On a semi-homogeneous tree, we study the $\ell^p$-spectrum of the Laplace operator $\mu_1$ (the isotropic nearest-neighbor transition operator); the known results in the much simpler setting of  homogeneous trees are obtained as particular cases.  The spectrum is given by the eigenvalues of spherical functions, i.e., eigenfunctions of $\mu_1$ that are radial with respect to a reference vertex $v_0$ and normalized there. We show that spherical functions are
 boundary integrals of generalized Poisson kernels that, unlike the homogeneous setting,  are not complex powers of the usual Poisson kernel. We compute these generalized Poisson kernels via Markov chains and their generating functions, whence we work out explicit expressions for  spherical functions.  On semi-homogeneous trees,  spherical functions turn out to have an 
$\ell^p$ behavior that does not occur on homogeneous trees: one of them, for an appropriate choice of $v_0$, belongs to $\ell^p$ for some $p<2$.

Up to normalization, the operator $\mu_1^2$ differs from the step-2 Laplacian $\mu_2$ only by a shift. On the other hand, 
the recurrence relation associated to the semi-homogeneous $\mu_2$ is is that of a polygonal graph, akin to that of the Laplacian in a homogeneous tree. By this token, we compute the spectra of $\mu_1^2$ on a semi-homogeneous tree, hence, by extracting square roots, the $\ell^p$-spectrum of $\mu_1$ for $1\leqslant p <\infty$, and show that it is disconnected for $p$ in an interval containing 2 but it is connected at all other values of $p$. 
\end{abstract}

\maketitle

\section{Introduction}
Let $T$ be a homogeneous tree, that is, a graph without loops and the same number $q+1$ of neighbors for each vertex; $q$ is called the homogeneity degree.
On its set $V$ of vertices, the study of the isotropic nearest-neighbor transition operator $\mu_1$ (often  called  Laplace operator) began long ago (see \cite{Figa-Talamanca&Picardello} for references), and was related to the Poisson boundary $\Omega$ of $T$ (a topological boundary that yields a compactification of $V$) and its Poisson kernel $\Poiss$. Far-reaching generalizations to general infinite trees have been considered \cite{Woess}, but the  setup of a  homogeneous tree has closer ties to harmonic analysis, because functions on $V$ are endowed with a convolution product induced by $\Aut T$ (or a simply transitive subgroup thereof), and  $\mu_1$ is a convolution operator. Appropriate estimates for the norms of convolution operators \cite{Haagerup} can then be used to determine the spectrum of $\mu_1$ on $\ell^p(V)$ for $1\leqslant p <\infty$ \cite{Figa-Talamanca&Picardello}*{Chapter 3, Theorem 3.3}. Once a root vertex $v_0$ is chosen, the $\gamma$-eigenfunction $\phi(\missarg,v_0\barra \gamma)$ of $\mu_1$ that is radial around $v_0$ and normalized to $1$ at $v_0$, called spherical function, can be expressed as a Poisson transform: $\phi(\missarg,v_0\barra \gamma)=\int_\Omega \Poiss(\missarg,v_0,\omega\barra \gamma)\,d\nu_{v_0}(\omega)$, where $\nu_{v_0}$ is the hitting distribution on $\Omega$ of the random walk starting at $v_0$ induced by $\mu_1$; indeed, $\Poiss(\missarg,v_0,\omega\barra \gamma)$ is an extremal $\gamma$-eigenfunction. For general trees, if $v,w\in V$ and $F(v,w)$ is the probability that the random walk starting at $v$ visits $w$, then
the harmonic Poisson kernel $\Poiss(v,v_0,\omega)$ (i.e., at the eigenvalue $\gamma=1$) is a limit of quotients: 
$\Poiss(v,v_0,\omega)=\lim_{w\to \omega}F(v,w)/F(v_0,w)$. The visit probability can be extended to an analytic function $\gamma\mapsto F(v,v_0\barra\gamma)$ on a corona in $\mathC$ as a Laurent series, and a similar rule holds for the generalized Poisson kernel at the eigenvalue $\gamma$, that is, $\Poiss(v,v_0,\omega\barra\gamma)=\lim_{w\to \omega}F(v,w\barra\gamma)/F(v_0,w\barra\gamma)$  \cite{Woess}*{Chapter 1, Section D; Chapter 4, Section C}. 
The function $F(v,v_0\barra\missarg)$ has a two-valued analytic continuation to  $\mathC\setminus \{0\}$, that we write as $F$ and $\widetilde F$. The other determination $\widetilde F$ diverges as $\gamma\to \infty$, whereas $F$ tends to zero at infinity (see \cite{Sava&Woess}).

As a consequence of a stopping time property of the hitting probability, on a homogeneous tree one has
 $\Poiss(v,v_0,\omega\barra\gamma)=\Poiss(v,v_0,\omega)^z$ for some complex number $z$ related to $\gamma$: this yields a nice parametrization of the eigenvalues, called the eigenvalue map $\gamma(z)=(q^z+q^{1-z})/(q+1)$ \cite{Figa-Talamanca&Picardello-JFA}. Thanks to the convolution estimates quoted above, it was shown in \cite{Figa-Talamanca&Picardello} that the spectrum of $\mu_1$ on $\ell^p(V)$ is the closure of the set $\{\gamma\in\mathC\colon \phi(\missarg,v_0\barra \gamma)\in\ell^r(V) \text{ for every }r>p\}$. Moreover, since $\mu_1$ is invariant under $\Aut T$, its eigenspaces give rise to spaces of representations of $\Aut T$,  called \emph{spherical representations}. Because of the complex power,
 each generalized Poisson kernel $\Poiss(v,\,v_0,\,\omega)^z$ is a character of an abelian group, thereby yielding an induced representation via Mackey's theory \cite{Casadio-Tarabusi&Gindikin&Picardello-Book}*{Section 2.8}.
  The  representation at the eigenvalue $\gamma$ is unitary if $\gamma$ belongs to the $\ell^2$-spectrum of $\mu_1$, and unitarizable if $\phi(\missarg,v_0\barra \gamma)$ is positive definite.

Here we consider trees that are \emph{semi-homogeneous}, i.e., with two  alternating homogeneity degrees $q_+$ and $q_-$; if $q_+\neq q_-$ we say that $T$ is \emph{strictly semi-homogeneous}.  The semi-homogeneous Laplace operator, that  we still denote by $\mu_1$, is defined on functions on $V$ again as the average operator on neighbors. In this context, we still have a two-valued analytic extension of the first visit probability of the random walk generated by $\mu_1$, with determinations $F(v,u\barra\gamma)$ and $\widetilde F(v,u\barra\gamma)$. A generalized Poisson kernel  at the eigenvalue $\gamma$ with respect to any reference vertex $v_0$ fixed in, say, $V_+$, can still be defined by means of either $F$ or $\widetilde F$, and it is constant on each horosphere tangent at $\omega$. We compute explicitly $F(v,u\barra\gamma)$, $\widetilde F(v,u\barra\gamma)$ and the corresponding generalized Poisson kernels, whence a long computation gives us the explicit values of the spherical functions. 
If  $\rho$  is the spectral radius of $\mu_1$ on $\ell^2(V)$, 
then $\phi(\missarg,v_0\barra\gamma)$ 
is a linear combination of two exponentials expressed in terms of nearest-neighbor values of $F$ and $\widetilde F$, except at the eigenvalues $\pm\rho$, where  the two exponentials coincide and 
the spherical function includes also a multiplicative term, namely a polynomial of degree 1 in $\dist(v,v_0)$.
This allows us to compute the $\ell^p(V)$-behavior of spherical functions (Section \ref{Sec:Spherical_functions}).

On the other hand, considerable differences arise with respect to the homogeneous setting. Indeed, the group $\Aut T$ has two orbits $V_+$ and $V_-$, hence it is not transitive and does not give rise to a convolution product. For the same reason, the natural definition of positive definite function does not make sense. As we have seen, the Poisson kernel can still be defined in terms of the first visit probability $F$ (Section \ref{Sec:Laplace_operator_and_Poisson_kernel}). 
In this setting, however, the Poisson kernel is not an exponential function of the horospherical index; moreover,  $\Poiss(v,v_0,\omega)^z$ is not an eigenfunction of $\mu_1$ for any complex $z\neq 0,1$, hence we lack a natural parametrization of the eigenvalues (Section \ref{Sec:Semi-homogeneous_Poisson_kernel_is_not_a_power}). As a consequence, we do not have the same estimates for convolution operators that have been used in the literature to compute the spectrum of $\mu_1$ in the homogeneous setting (see \cite{Figa-Talamanca&Picardello}*{Chapter 3}), and it is much harder to build an analytic series of spherical representations of $\Aut T$ by means of spherical functions.

Indeed, the essential problem raised by the existence of two orbits under $\Aut T$  is the lack of a suitable convolution algebra $\calR$ of radial functions. In the homogeneous setting, the spectrum of $\mu_1$ is obtained from the multiplicative functionals on the  Banach algebra generated by $\calR$ in the $\ell^1$-norm. Here we need a new approach. The operator $\mu_1$ corresponds to a nearest-neighbor transition operator, hence transitions occur between vertices of different homogeneities.
Its square $\mu_1^2$, however, preserves the homogeneity, and so  it acts on the two orbits separately, and is transitive on each of them. On the other hand, $\mu_1^2$ is related to the step-2 isotropic transition operator $\mu_2$, indeed, up to a scale factor, it coincides with $\mu_2$ plus a positive multiple of the identity (that yields a rightward shift of the spectrum). On strictly semi-homogeneous trees, $\mu_1$ is not self-adjoint on $\ell^2(V)$, but it is self-adjoint on $\ell^2(V,m)$ where $m$ is a suitable weight given by two values, one for each homogeneity. Then, clearly, the spectrum of $\mu_1$ on $\ell^p(V)$ is the same as its spectrum on $\ell^p(V,m)$, and it is obtained from the spectrum of $\mu_2$ on $\ell^p(V_+)$ and $\ell^p(V_-)$ by taking the complex square root. On the other hand, the notion of adjacency that $\mu_2$ defines on $V_\pm$, that corresponds to distance 2 in $V$, transforms $V_\pm$ into polygonal graphs, studied in \cite{Iozzi&Picardello-Springer}. This action is transitive and gives rise on, say, $V_+$ to a nice convolution algebra of functions radial in this graph around a fixed reference vertex $v_0\in V_+$. By a refinement of the results of \cite{Iozzi&Picardello-Springer}, we then compute the $\ell^p$-spectra of $\mu_2$, hence of $\mu_ 1^2$ (Section \ref{Sec:Spectrum_of_M2}
), and they are given, as in \cite{Figa-Talamanca&Picardello}*{Chapter 3, Theorem 3.3}, by the closure of the set $\{\gamma\in\mathC\colon \phi(\missarg,v_0\barra \gamma)\in\bigcap_{r>p} \ell^r(V)\}$.
These sets are elliptical regions in the complex plane, centered not at the origin as in the homogeneous setting, but at the point of the positive half-axis that corresponds to the rightward shift mentioned above. Hence, in Section \ref{Sec:Spectrum_of_M1}, we compute the $\ell^p$-spectra of $\mu_1$, obtained by extracting the square root, therefore obviously invariant under the reflection  $\gamma\mapsto -\gamma$ around the origin. Because of the rightward shift,
there is an interval of values of $p$ around $2$ where the $\ell^p$-spectrum of $\mu_1^2$ does not contain the origin. Then, for these values of $p$, the $\ell^p$-spectrum of $\mu_1$  does not contain the origin either and is disconnected, unlike for homogeneous trees.

Another very unusual feature, studied in Section 5, is that there can be a spherical function that belongs to $\ell^2(V)$, and also to $\ell^p(V)$ for some $p<2$. Indeed, the spherical function at the eigenvalue 0 belongs to $\ell^p(V)$ for every $p>1+\dfrac{\ln q_+}{\ln q_-}$, that is smaller than 2 if and only if $q_+<q_-$. Each other bounded spherical function $\phi(\missarg,v_0\barra \gamma)$ belongs to $\ell^p(V)$ for some $p>2$ depending on $\gamma$. 

We deal with the $\ell^p$-spectra of $\mu_1$ by means of generating functions for the induced Markov process. Alternatively, the explicit expression of spherical functions could be derived, as usual in the literature, by a direct, but here tiresome computation of the solutions of the difference equations satisfied by the eigenfunctions of $\mu_1$; for completeness, this alternative proof is given in Section \ref{Sec:spherical_functions_via_difference_equations}.

The lack of transitivity of the group of automorphisms gives rise to some interesting problems.
Since $\mu_1$ is $\Aut T$-invariant, i.e., it commutes with $\Aut T$,  its eigenspaces are also invariant under $\Aut T$, hence they give rise to representations of this group, namely,  the spherical representations. The representation space at the eigenvalue $\gamma$ is the span of all translates of $\phi(\missarg,v_0\barra \gamma)$. In the homogeneous setting, this representation is unitary or unitarizable if and only if $\phi(\missarg,v_0\barra \gamma)$ is positive definite (thereby providing the necessary Hilbert norm via the GNS construction). In the semi-homogeneous setting, the notion of positive definite makes sense for functions defined on this group, but does not extend to functions on $V$, that cannot be lifted to  $\Aut T$, since this group is not transitive. Even though  there are only two orbits, it is not clear how to equip the space of functions on $V$ with a suitable definition of positive definite function, or at least a Hilbert norm on the linear span of all translates of $\phi(\missarg,v_0\barra \gamma)$.

For the same reason, $\Aut T$ does not induce a convolution product on functions on $V$, but only on functions on $V_+$ and on $V_-$. Indeed, if $v\in V$ and $\lambda\in\Aut T$ is such that $v=\lambda v_0$, a convolution product of two functions $f, g$ on $V$ should be defined as $f*g(v)=\int_{\Aut T} f( \tau^{-1}\lambda v_0)\,g(\tau v_0) \,d\tau$, where the measure $d\tau$ is the Haar measure of the (unimodular) group $\Aut T$. On the other hand, this expression defines a convolution product only for the functions supported on the orbit of $v_0$. In order to extend the convolution product to the functions supported  in the other orbit, we need to extend it by choosing another reference vertex in the other orbit. This actually leads to two different convolution products, one for each orbit (see \cite{CCKP} for more details). It is not clear how to define the convolution of functions not supported on a single orbit in such a way that it coincides with the usual definition when the tree is homogeneous. Moreover, it is not clear how to formulate and prove an analogous of Haagerup's convolution estimate \cite{Haagerup}, that is the main tool for computing the spectra of the Laplacian in the well-established approach of \cite{Figa-Talamanca&Picardello}*{Chapter 3}.

Finally, the spherical representation $\pi_\gamma$ of a group $G$ acting simply transitively on a homogeneous tree is irreducible for $\gamma\in\spectrum_{\ell^1}(\mu_1)$ \cite{Figa-Talamanca&Picardello}*{Chapter 5}. The argument, later extended in \cite{Figa-Talamanca&Steger} to non-isotropic nearest-neighbor transition operators, proceeds by proving that the projector onto a cyclic vector is the weak limit, as $n\to\infty$, of the average of $\pi_\gamma(\tau)$ over all the elements $\tau\in G$ that move $v_0$ to vertices of length $n$. Here again, the lack of transitivity makes it unhandy to reproduce the same argument.
The authors will deal with these open problems in a forthcoming paper.

This article required rather complicated algebraic computations, that were all verified by computer-assisted symbolic manipulation. Most of the plots in the figures were drawn by using \emph{Mathematica} by Steven Wolfram, version 13.1.0.0.1.

\section{Semi-homogeneous trees and their boundaries}
\label{Sec:semi-homogeneous_horospheres_and_circles}

\subsection{Trees and their boundaries}\label{SubS:boundaries}

A \emph{tree} $T$ is a connected, countably infinite, locally finite graph without loops. The nodes of $T$ are called \emph{vertices}; the set of all vertices is denoted by $V$. Two distinct vertices $v,v'$ are \emph{adjacent}, or \emph{neighbors}, if they belong to the same edge: we shall write $v\sim v'$.
Let us fix a reference vertex $v_0$. The \emph{distance} $\dist(v,v')$ between two vertices $v,v'$ is the number of edges in the shortest path, called \emph{chain}, of consecutive neighbors from $v$ to $v'$. The distance $\dist(v,v_0)$ from the reference vertex is called the \textit{length} of $v$ and denoted by $|v|$. The number of neighbors of $v$ is $q_v+1$, where $q_v$ is the  \emph{homogeneity degree}, that is, the number of outward neighbors of $v$ with respect to $v_0$ if $v\neq v_0$.
Let us fix a \emph{parity}, that is, an alternating function $\epsilon\colon V\to \{\pm 1\}$. The level sets of $\epsilon$ are denoted by $V_+$ and $V_-$. 

A \emph{semi-homogeneous} tree $T=T_{q_+,\,q_-}$ has two alternating homogeneity degrees ${q_+}$ on $V_+$ and ${q_-}$ on $V_-$. If ${q_+}={q_-}$ then the tree is \emph{homogeneous}. We  assume $q_+, q_->1$ and choose  $v_0\in V_+$, that is,
\begin{equation}\label{eq:q_{v_0}=q_+}
    q_{v_0}=q_+,
\end{equation}
and the parity is $\epsilon(v)=(-1)^{|v|}$.

If $T$ is homogeneous, since it has no loops it is the Cayley graph of the free product $\calG$ 
of $q+1$ copies of the two-element group $\mathZ$. Indeed, this free product embeds in the group $G=\Aut T$ of \textit{automorphisms} of $T$, the invertible self-maps of $V$ that preserve adjacency: the generators of the factors are  automorphisms that reverse the edges that contain $v_0$ (see, for instance, \cite{Figa-Talamanca&Picardello}*{Chapter 3, Section V}). In particular, $\calG$ acts simply transitively on $T$ and induces a convolution product on functions on $V$. 
Therefore $G$ is transitive on $V$ if $T$ is homogeneous.  On the space of functions on $G$ that are two-sided-invariant  under the stability subgroup $G_{v_0}\subset G$ of $v_0$, the convolution operation induced by the group $G$ is a lifting of the convolution induced by 
$\calG$
(see \cite{Casadio-Tarabusi&Picardello-spherical_functions_on_edges}*{Appendix}).
 On the other hand, if ${q_+}\neq {q_-}$, then there are two orbits of $\Aut T$ on $V$, namely $V_{+}$ and  $V_{-}$.

We now give a short outline of the notion of boundary of $T$, taken from \cite{Cartier-Symposia}; the reader can find more details in this reference and in \cite{Casadio-Tarabusi&Gindikin&Picardello-Book}.
A  \textit{ray} in $V$ is a one-sided infinite chain starting at some vertex $v$. 

 On the set of all rays we introduce the equivalence relation such that two rays are equivalent if they merge after finite numbers of steps and coincide thereafter. Each equivalence class $\omega$ is called a \textit{boundary point} of $T$.  The \emph{boundary} of $T$, denoted by $\Omega$, is the set of all boundary points. In each equivalence class there is a unique representative that starts at the reference vertex $v_0$.

Let $v_1, v_2\in V$ and $\omega\in\Omega$. The \emph{horospherical index} of $v_1$ with respect to $v_2$ and $\omega$, introduced in \cite{Cartier-Symposia}, is $h(v_1,v_2,\omega)=\dist(v_2,j)-\dist(v_1,j)$ for any $j\in [v_1,\omega)\cap[v_2,\omega)$; the value of $h(v_1,v_2,\omega)$ is clearly independent of the choice of $j$. The closest of such vertices $j$ to $v_1$ and $v_2$ is often called the \emph{join} in the literature.

$\Omega$ is equipped with a natural totally disconnected compact topology.  For $u, v\in V$, $u\neq v$, let us consider the \emph{boundary arcs} $\Omega(u,v)$ given by all boundary points $\omega$ whose equivalence class includes rays that start at $u$ and contain $v$. Similarly, the vertices $w$ such that the chain from $u$ to $w$ contains a given vertex $v$ form the \textit{sector} $S(u,v)\subset V$. We shall also set $\Omega(v,v)=\Omega$ and $S(v,v)=V$. The collection of sets $\Omega(u,v)$ is a base for a compact, totally disconnected topology of $\Omega$; then a local base at a boundary point $\omega$ with representative $[v_0, v_1, \dots, v_n, \dots)$ consists of the boundary arcs $\Omega(v_0,v_n)$.
Similarly, the collection  $S(u,v)\cup \Omega(u,v)\subset V\cup\Omega$ is a base for a compact, totally disconnected topology of $V\cup\Omega$; on the subspace $V$ this is the discrete topology.

   Let us denote by $\widetilde q_v$ the opposite homogeneity  to $v$'s, that is,  $\widetilde q_v=q_\pm$ if $v\in V_\mp$, and observe that the number of vertices at distance $n$ from $v$ is
\begin{equation}\label{eq:cardinality_of_circles}
\left| \{w\colon \dist(w,v)=n\} \right|=
\begin{cases}
1&\qquad\text{if } n=0,\\[.2cm]
(q_v +1) \,\qmean^{n-1}  &\qquad\text{if $n$ is odd},\\[.2cm]
(q_v +1)\, \qmean^{n-2}\,\widetilde q_v  &\qquad\text{if $n$ is positive even.}
\end{cases}
\end{equation}
 There is a normalized measure $\nu_{v_0}$ naturally defined on the
   Borel $\sigma-$algebra generated by the open sets $\Omega(v_0,v)\subset \Omega$.   and invariant under the stabilizer $G_{v_0}$. 
Namely,
\begin{equation}\label {eq:semihomogeneous_invariant_measures}
\nu_{v_0}(\Omega(v_0,v)) = 1/\left| \{w\colon \dist(w,v_0)=n\} \right|.
\end{equation}
The measure $\nu_{v_0}$  is called the \textit{equidistributed boundary measure} with respect to $v_0$.
Note that $G_{v_0}$ acts transitively  on $\{v\in V\colon |v|=n\}$ for each $n\in\mathN$, hence on $\Omega$.

\section{The Laplace operator and the Poisson kernel}\label{Sec:Laplace_operator_and_Poisson_kernel}
We now  compute the Poisson kernel for the semi-homogeneous tree $T_{q_+,q_-}$. Let us consider the most general case of nearest-neighbor Laplace (or better, average) operator, depending on the two homogeneity degrees.  
At each vertex $v$, the Laplace operator applied to a function $f$ on $V$ is the average of the values of $f$ at the neighbors of $v$. The number of neighbors depends on the parity of the vertex. That is,
\begin{equation}\label{eq:semihomogenous_Laplacian}
\mu_1 f(v) =
\frac1{q_{\epsilon(v)}+1}\; \sum_{w\sim v} f(w).
\end{equation}

\begin{remark}\label{rem:recurrence_relation_of_the_semi-homogeneous_Laplacian}
Let us consider the space $\mathfrak R$ of summation operators with kernel, acting on $\ell^1(V)$, that is,  of the type $R f(v)=\sum_{w\in V} r(v,w) f(w)$, whose kernel $r(v,w)$ is of  finite range (that is, it vanishes if $\dist(v,w)>N$ for some $N\in\mathN$ depending on $r$). Then $\mu_1$ belongs to $\mathfrak R$. More generally, for $n\in\mathN$  we set
\[
\mu_n f(v)= \frac1{\left| \{w\colon \dist(w,v)=n\} \right|}\sum_{\dist(w,v)=n} f(w).
\]

An elementary computation shows that,
\begin{equation}\label{eq:recurrence_for_M1_first_step}
    \mu_1^2 f(v) = \frac 1{\widetilde q_v +1} f(v) + \frac {\widetilde q_v}{\widetilde q_v +1} \mu_2 f(v)
\end{equation}
and, for every $n>0$,
\begin{equation}\label{eq:recurrence_for_M1}
\mu_1\mu_n f(v)=\mu_n\mu_1f(v)= 
\frac 1{q_v +1}\; \mu_{n-1} + \frac {q_v}{q_v +1}\; \mu_{n+1}\,. 
\end{equation}
Hence, if $f$ is a $\gamma$-eigenfunction of $\mu_1$, then $f\left|_{V_\pm}\vphantom{|_{V_\pm}}\right.$ are eigenfunctions of $\mu_2^\pm$ with respective eigenvalues 

 \begin{equation}\label{eq:the_nonhomogeneous_eigenvalue_map}
     \gamma_\pm= \frac{(q_\mp +1) \gamma^2 -1}{q_\mp}\;.
 \end{equation}

It follows from \eqref{eq:recurrence_for_M1} that the vector space $\mathfrak R$ is a commutative algebra generated by $\mu_1$, and on every ray $[v_0, v_1,\dotsc)$ each $\gamma$-eigenfunction of $\mu_1$ radial around $v_0$ satisfies for every $n\in\mathN$ the recurrence relation 
\begin{equation}\label{eq:semihomogeneous_recurrence}
\begin{split}
\gamma\,f(v_0)  &= f(v_1) \qquad\text{if }|v_1|=1;
\\[.2cm]
\gamma\,f(v_n)&=\frac1{q_{v_n} +1} \; f(v_{n-1}) + \frac{q_{v_n}}{q_{v_n} +1}\; f(v_{n+1}) \qquad\text{if }|v_n|=n>0.
\end{split}
\end{equation}

In particular, for each eigenvalue $\gamma$ there is exactly one $\gamma$-eigenfunction $f$ radial around any given vertex $v_0$ and satisfying the initial condition $f(v_0)=1$. If $\gamma= 0$, this eigenfunction must vanish on each $v$ with $|v|=1$ by the first identity in \eqref{eq:semihomogeneous_recurrence}, henceafter on all of $V_-$ by the second identity; it is now easy to compute this eigenfunction, that we denote by $\psiphi (v,v_0\barra 0)$, in the usual case $v_0\in V_+$:
\begin{equation}\label{eq:the_semi-homogeneous_spherical_function_of_eigenvalue_zero}
\psiphi (v,v_0\barra 0)=
\begin{cases}
(-1)^{\frac{|v|}2} q_-^{-\frac{|v|}2}&\qquad\text{if $|v|$ is even},\\[.2cm]
0                       &\qquad\text{if $|v|$ is odd}.
\end{cases}
\end{equation}
A similar example was suggested to us by W.~Woess (personal communication).
Of course, if we choose $v_0\in V_-$, at the right-hand side we must replace $q_-$ by $q_+$.

With
\begin{equation}\label{eq:pcrit}
\pcrit=
\frac{\ln(q_+q_-)}{\ln q_-},
\end{equation}
this yields the following $\ell^p$-behavior: 
\begin{equation}\label{eq:phi(v,v_0,0)_in_l^p_iff_p>p_crit}
\psiphi (\missarg,v_0\barra 0)\in \ell^p(V) \text{ if and only if } p>\pcrit.
\end{equation}
In particular,
$
\psiphi (\missarg,v_0\barra 0)\in \bigcap_{p>\pcrit}\ell^p(V)$.
Note that $\pcrit=2$ in the homogeneous setup $q_+=q_-$, and  $\pcrit>2$ if and only if $q_+>q_-$. Therefore
\begin{equation}\label{eq:phi(v,v_0,0)_in_l^2_iff_q_+<q_-}
\psiphi (\missarg,v_0\barra 0)\in \ell^2(V) 
\Longleftrightarrow q_+<q_-.
\end{equation}

As a consequence, if (and only if) $q_+<q_-$, then $0$ belongs not only to $\spectrum(\mu_1;\,\ell^2(V))$ (see Corollary \ref{cor:0_is_in_the_l^2_spectrum}) below), but also to the discrete spectrum of $\mu_1$ as an operator on $\ell^2$ and more generally on $\ell^p$ for every $p\geqslant 2$.
\end{remark}

\begin{definition}\label{def:spherical_functions}
For $\gamma\in\mathC$, the $\gamma$-eigenfunction of $\mu_1$ radial around $v_0$ with value 1 at $v_0$ is called \emph{spherical function} with eigenvalue $\gamma$ and is denoted by $\phi(\missarg,v_0\barra \gamma)$.
\end{definition}

\begin{lemma}\label{lemma:measures_of_boundary_arcs_in_semihomogenous_trees}
Let $v_0,v_n\in V$ and denote by $[v_0,\dotsc,v_n]$ the chain from $v_0$ to $v_n$. 
For $0\leqslant k \leqslant n$, consider the arc $\Omega_k(v_0,v_n)$ of all the boundary points whose ``closest vertex'' in the path $[v_0,\dotsc,v_n]$ is $v_k$, that is,
\[
\Omega_k(v_0,v_n)=
\begin{cases}
\Omega(v_n,v_0)  &\qquad\text{ if } k=0;\\[.2cm]
\Omega(v_0,v_k)\setminus\Omega(v_0,v_{k+1})  &\qquad\text{ if } k=1,\dotsc,n-1;\\[.2cm]
\Omega(v_0,v_n)  &\qquad\text{ if } k=n.
\end{cases}
\]
Then, by \eqref{eq:q_{v_0}=q_+},  $q_{v_k}=q_{(-1)^k}$ and
\begin{equation}\label{eq:measures_of_arcs_subtended_by_subsequent_vertices_in_a_path-0}
\nu_{v_0}(\Omega_k(v_0,v_n))
=\begin{cases}
1
&\text{if $k=0=n$,}\\[.3cm]
\displaystyle
\frac{q_+}{q_+ +1}
&\text{if $k=0<n$},\\[.4cm]
\displaystyle
\frac{q_{+}-1}{q_+ +1}\; \qmean^{-k}
&\text{if $0<k<n$ and $k$ even,}\\[.4cm]
\displaystyle
\frac{q_- -1}{q_+ +1}\;  \qmean^{-k}\,\sqrt{\frac{q_+}{q_-}}
&\text{if $0<k<n$ and $k$ odd,}\\[.4cm]
\displaystyle
\frac{q_+}{q_+ +1}\; \qmean^{-n}
&\text{if $0<k=n$ even},\\[.4cm]
\displaystyle
\frac{q_-}{q_+ +1}\; \qmean^{-n}\,\sqrt{\frac{q_+}{q_-}}
&\text{if $0<k=n$ odd}.
\end{cases}
\end{equation}
\end{lemma}
\begin{proof}
The formulas for $0<k<n$ follow by applying \eqref {eq:semihomogeneous_invariant_measures} and \eqref{eq:cardinality_of_circles}, and are left to the reader.
For the two endpoints of the segment $[v_0,\dots,v_n]$, that is, $k=0$ or $k=n$ (here $n$ can be even or odd), the result follows directly from \eqref {eq:semihomogeneous_invariant_measures}:  $ \nu_{v_0}(\Omega_0(v_0,v_n)) = \frac{q_+}{q_+ +1}$, and
$\nu_{v_0}(\Omega_n(v_0,v_n)) = \frac{q_{(-1)^n}}{q_+ +1}\;q_+^{-\floor{\frac{n}2}}\, q_-^{-\ceiling{\frac{n}2}}$. 
\end{proof}

For every $v\in V$ the measure $\nu_v$ is absolutely continuous with respect to $\nu_{v_0}$. The \emph{Poisson kernel}  is the Radon--Nikodym derivative 
\begin{equation}\label{eq:Poisson_kernel_as_Radon-Nikodym_derivative}
\Poiss(v,v_0,\omega) = \frac {d \nu_{v}}{d\nu_{v_0}}(\omega).
\end{equation}

The \emph{Poisson transform} is the map $\Poisstr$ from $L^1(\Omega, \nu_{v_0})$ to functions on $V$ defined as
\begin{equation}\label{eq:Poisson_transform_expressed_on_Omega}
\Poisstr f (v)
 =  \int_{\Omega} f(\omega)\,d\nu_{v}(\omega) = \int_{\Omega} 
\Poiss(v,v_0,\omega)\, f(\omega)\,d\nu_{v_0}(\omega).
\end{equation}

 \begin{theorem}
 \label{theo:Poisson_kernel_for_vertices,semihomogeneous}
For $\omega\in \Omega_k(v_0,v)$ and $|v|=n$, 
\begin{equation}\label{eq:Poisson_kernel_for_vertices,semihomogeneous}
\Poiss(v,v_0,\omega) = \frac{\nu_{v}(\Omega_k(v_0,v))}{\nu_{v_0}(\Omega_k(v_0,v))} =
\begin{cases}

\; (q_+q_-)^{k-\frac n2} 
&\text{if $n$ is even}
,
\\[.4cm]
\dfrac {q_+ +1}{q_- +1}\; \sqrt{\dfrac{q_-}{q_+}}  (q_+q_-)^{k-\frac n2} &\text{if $n$ is odd}.
\end{cases}
\end{equation}
where $\qmean = \sqrt{q_+q_-}$.
\end{theorem}
\begin{proof} By \eqref{eq:q_{v_0}=q_+}, one has $q_{v_k}=q_{(-1)^k}$.
For $\omega\in \Omega_k(v_0,v_n)$ and $0<k<n$ with $n$ even, it follows from Lemma \ref{lemma:measures_of_boundary_arcs_in_semihomogenous_trees}
that
\begin{equation*}
\Poiss(v_n,v_0,\omega)=\frac{\nu_{v_n}(\Omega_k(v_0,v_n))}{\nu_{v_0}(\Omega_k(v_0,v_n))} = q_+^{ \floor{\frac k2} - \floor{\frac{n-k}2 }}\; q_-^{ \ceiling{\frac k2} - \ceiling{\frac{n-k}2 }} .
\end{equation*}
Since $n$ is even, then $\floor{\frac k2} - \floor{\frac{n-k}2} = k-\frac n2$, hence the last expression becomes
\begin{equation}\label {eq:Poisson_kernel,n_even}
\Poiss(v_n,v_0,\omega) = (q_+q_-)^{k-\frac n2}.
\end{equation}
Instead, if $n$ is odd, we obtain
\begin{equation*}
\Poiss(v_n,v_0,\omega)=
\frac {q_+ +1}{q_- +1}\;q_+^{ \floor{\frac k2} - \ceiling{\frac{n-k}2 }}\; q_-^{ \ceiling{\frac k2} - \floor{\frac{n-k}2 }} .
\end{equation*}
For odd $n$ one has $ \floor{\frac k2} - \ceiling{\frac{n-k}2 } = k-\frac{n+1}2$, and
$\ceiling{\frac k2} - \floor{\frac{n-k}2 } =  k-\frac{n-1}2$. Hence
\begin{equation}\label {eq:Poisson_kernel,n_odd}
\Poiss(v_n,v_0,\omega) = \frac {q_+ +1}{q_- +1}\; \sqrt{\frac{q_-}{q_+}}  (q_+q_-)^{k-\frac n2}.
\end{equation}

For both cases, $n$ even or odd, this yields the expression of the Poisson kernel in the statement.
%
Let us now deal with the two extreme cases $k=0$ and $k=n$;  here $n$ can be even or odd, and we assume $n>0$ otherwise the Radon-Nikodym derivative has value 1. 
One has $ \nu_{v_0}(\Omega_0(v_0,v_n)) = \frac{q_+}{q_+ +1}$, and
$\nu_{v_0}(\Omega_n(v_0,v_n)) = \frac{q_{(-1)^n}}{q_+ +1}\;q_+^{-\floor{\frac{n}2}}\, q_-^{-\ceiling{\frac{n}2}}$. The same formulas for $\nu_{v_n}$ instead of $\nu_{v_0}$ are obtained by exchanging $0$ with $n$. So, 
\begin{align*}
\frac{\nu_{v_n}(\Omega_0(v_0,v_n))}{\nu_{v_0}(\Omega_0(v_0,v_n))} &= \frac {q_+ +1}{q_{(-1)^n}+1} \; q_{(-1)^n}^{-\floor{n/2}} \, q_{(-1)^{n+1}}^{-\ceiling{n/2}}\;,\\[.2cm]
\frac{\nu_{v_n}(\Omega_n(v_0,v_n))}{\nu_{v_0}(\Omega_n(v_0,v_n))} &=  \frac {q_+ +1}{q_{(-1)^n}+1} \; q_+^{\floor{n/2}} \,q_-^{\ceiling{n/2}}.
\end{align*}
In both cases, one obtains again the formula in the statement. 
The same formula holds if we assume $q_{v_0}=q_-$ because the factor $(q_+q_-)^{k-\frac n2}$ is invariant if $q_+$ and $q_-$ are exchanged.
\end{proof}

\begin{remark}
\label{rem:semihomogeneous_Poisson_kernel_on_horospheres}
The exponent $2k-n$ in Theorem \ref{theo:Poisson_kernel_for_vertices,semihomogeneous} is precisely the horospherical index $h(v,v_0,\omega)$. Hence the Poisson kernel is constant on horospheres. This is a consequence of the fact that vertices in the same horosphere are at even distance, hence all of the same parity, either even or odd. On the other hand, the expression is given by a term $\qmean^{h(v,v_0,\omega)}$ multiplied by a constant factor that changes from  horospheres with even index to horospheres with odd index by the multiplication by the factor
\begin{equation}\label{eq:dissimetrizzation_factor}
 \frac {\sqrt{q_-}}{q_- +1}\bigg/ \frac{\sqrt{q_+} } {q_+ +1}
\end{equation}
or vice-versa. In particular, the Poisson kernel it is not an exponential of the type $\beta^{nz}=\beta^{zh(v,v_0,\omega)}$, contrarily to the setup of homogeneous trees (and also of symmetric spaces, see \cite{Helgason-GGA}).
\end{remark}

\section
{The semi-homogeneous Poisson kernel and radial averages of its powers}
\label{Sec:Semi-homogeneous_Poisson_kernel_is_not_a_power}

Each vertex $v\neq v_0$ has one backward neighbor $v'$ (one step closer to $v_0$) and $q_{\epsilon(v)}$ forward neighbors $v''$. 
Hence, if $f$ is radial around $v_0$,  for $v\neq v_0$ one has
\begin{equation}\label{eq:Laplacian}
\mu_1 f(v)= 
\frac{f(v') + q_{\epsilon(v)} f(v'')}{q_{\epsilon(v)}+1}\;;
\end{equation}
instead, if $v=v_0$, then $\mu_1 f(v_0) =  f(v_1)$, where $v_1$ denotes any neighbor of $v_0$.

In the case of a homogeneous tree, it was proved in \cites{Figa-Talamanca&Picardello-JFA,Figa-Talamanca&Picardello}, and for positive eigenvalues also in \cite{Cartier-Symposia}*{Th\'eor\`eme 2.1}, that for every $z\in\mathC$, up to multiples, there is only one non-zero radial eigenfunction of $\mu_1$, called \emph{spherical function}, and it is given by the following Poisson representation:
\begin{equation}\label{eq:def_of_semihomogeneous_spherical}
\varphi _z(v) 
= 
 \int_{\Omega}  \Poiss^z(v,v_0,\omega) \,d\nu_{v_0}(\omega).
\end{equation}

The eigenfunctions of the Laplacian on trees that are semi-homogeneous but not homogeneous  have a different behavior:
\begin{proposition}\label{prop:semihomogeneous_K^z_is_not_eigenfunction} 
For $T$ homogeneous and for every $\omega\in\Omega$, $v\mapsto \Poiss^z(v,v_0,\omega)$ is an eigenfunction of $\mu_1$ with eigenvalue $\gamma(z)=(q^z+q^{1-z})/(q+1)$. Instead,
for $T$ strictly  semi-homogeneous,  
$v\mapsto \Poiss^z(v,v_0,\omega)$ is an eigenfunction of $\mu_1$ if and only if $z=i k\pi/\ln \qmean$ or $z=1-i k\pi/\ln \qmean$ with $k\in\mathZ$, with eigenvalue $(-1)^k$. For every $v$ and $\omega$, we have $\Poiss^z(v,v_0,\omega) = 1$ if $k$ is even, and $\Poiss^z(v,v_0,\omega)=(-1)^{|v|}$ if $k$ is odd. In particular, on strictly semi-homogeneous trees the integrals $\int_\Omega \Poiss^z(v,v_0,\omega)\,d\mu(\omega)$, when $\mu$ is a \emph{distribution} (that is, a finitely additive measure) on $\Omega$, are in general not eigenfunctions of $\mu_1$, unlike the set-up of homogeneous trees. 
\end{proposition}
\begin{proof}
The first part of the statement is trivially verified; see also \cites{Figa-Talamanca&Picardello-JFA,Figa-Talamanca&Picardello}. For $T$ strictly semi-homogeneous, fix $\omega\in\Omega$ and let $u_z= \mu_1 \Poiss^z (\missarg,v_0,\omega)$.   Let us first compute $u_z(v_0)$. 
From the expression of the Poisson kernel 
we have
\[
u_z(v_0) = \left( \frac {\sqrt{q_-}}{q_- +1}\bigg/ \frac{\sqrt{q_+} } {q_+ +1} \right)^z\; \frac{\qmean^{z} + q_+ \qmean^{-z}}{q_+ +1}
\]
because, for each $\omega\in\Omega$, if $v$ is the  neighbor of $v_0$ closest to $\omega$ then, with notation as in Theorem \ref{theo:Poisson_kernel_for_vertices,semihomogeneous}, we have
$n=1$, $k=1$ and the horospherical index $h(v,v_0,\omega)=2k-n$ has value 1; on the other hand, at the remaining $q_+$ neighbors, $k=0$ and $2k-n=-1$.
Observe that the  right hand side, in the homogeneous case $q=q_+=q_-$, is precisely the eigenvalue 
\begin{equation}\label{eq:the_homogeneous_eigenvalue_map}
\gamma(z)=\frac{q^z+q^{1-z}}{q+1}=\frac{2\sqrt{q}}{q+1}\; \cosh\biggl(\Bigl(z-\frac12\Bigr)\ln q\biggr)
\end{equation}
and $2\sqrt{q}/(q+1)$ is the spectral radius $\rho(\mu_1)$ on $\ell^2(V)$ \cites{Figa-Talamanca&Picardello-JFA,Figa-Talamanca&Picardello}. The function $\gamma$ satisfies $\gamma(z)=\gamma(1-z)$ and is periodic along the imaginary direction with period $2{\pi i}/{\ln q}$.

Now let us compute $u_z(v_1)$ with $v_1\sim v_0$. For $\omega$ in $\Omega(v_0,v_1)$, by \eqref {eq:Poisson_kernel,n_odd},
\begin{equation}\label{eq:generalized_homogeneous_Poisson_kernel_as_exponential}
\Poiss(v_1,v_0,\omega)^z= \left(\dfrac{\sqrt{q_-}}{q_- +1}\bigg/ \dfrac{\sqrt{q_+}}{q_+ +1}
\right)^z \qmean^z.
\end{equation}
Let us consider the neighbor $v_2\sim v_1$ closest to $\omega$. Then $k=n=2$, $h(v_2,v_0,\omega)=2k-n=2$ and $ \Poiss^z(v_2,v_0,\omega) =\qmean^{2z}$ by \eqref {eq:Poisson_kernel,n_even}. On the other hand, for all  neighbors  $v\sim v_1$ different from $v_2$ we have $k=1$, horospherical index 0 and $\Poiss(v,v_0,\omega)= 1$. Thus
\[
u_z(v_1)=\mu_1 \Poiss^z(v_1,v_0,\omega) = \frac{q_-}{q_- +1} + \frac1{q_- +1}\;\qmean^{2z}.
\]
Thus, for general $z\in\mathC$,
\begin{equation*}
\begin{split}
\frac{ u_z(v_1)}{ u_z(v_0)} &= \frac{ \displaystyle  \frac{q_-+\qmean^{2z}}{q_- +1}  }  {\displaystyle  \left(  \frac {\sqrt{q_-}}{q_- +1}\bigg/ \frac{\sqrt{q_+} } {q_+ +1} \right)^z \frac{q_+ +\qmean^{2z}}{q_+ +1}} \;\qmean^z  
\\[.2cm]
&
\neq  \left(\frac{\sqrt{q_-}}{q_- +1} \bigg/ \frac{\sqrt{q_+}}{q_+ +1}
\right)^z \qmean^z 
= \Poiss^z(v_1,v_0,\omega) = \frac{ \Poiss^z(v_1,v_0,\omega)}{\Poiss^z(v_0,v_0,\omega)} .
\end{split}
\end{equation*}
\end{proof}

\section{Potential theory on semi-homogeneous trees}\label{Sec:Poisson_representation}

The goal of this Section is a boundary integral representation for  eigenfunctions radial with respect to $v_0$ of 
the Laplacian  \eqref{eq:semihomogenous_Laplacian}.

By Proposition \ref{prop:semihomogeneous_K^z_is_not_eigenfunction},
 the Poisson kernel is a harmonic function, but for $z\neq 0$ or 1, the function
$\Poiss(\missarg,v_0,\omega)^z$ is \emph{not} an eigenfunction of the Laplacian $\mu_1$ if $q_+\neq q_-$.
A priori, this does not mean that its radialization $\int_\Omega \Poiss(v,v_0,\omega)^z \,d\nu_{v_0}(\omega)$ is not an eigenfunction of $\mu_1$, but in fact it is not \cite{Casadio-Tarabusi&Gindikin&Picardello-Book}. Therefore, unlike in the homogeneous setting of \eqref{eq:def_of_semihomogeneous_spherical},  spherical functions on semi-homogeneous trees are not integrals of powers of the Poisson kernel.
\subsection{The Poisson kernel on a general tree as a quotient of hitting probabilities}\label{SubS:first_visit_probabilities}
In this and the next subsection, $T$ is a general tree and $P$ is a nearest-neighbor transition operator $P$. We say that $P$ is \emph{isotropic} if $p(v_0,u)=p(v_0,w)$ when $|u|=|w|$; the  operator $\mu_1$ on semi-homogeneous trees is a particular instance.

We know from the theory of transient Markov chains that at least  positive eigenfunctions of $P$ have a Poisson integral representation; see \cites{KSK,Cartier-Symposia, Woess} for more details. That is, every $\lambda$-eigenfunction has a boundary representation $h_\lambda=\int_\Omega \Poiss(\missarg,v_0,\omega|\lambda)\, d\mu(\omega)$ ($\mu$ being any finitely additive measure on $\Omega$).

Let us compute the reproducing kernel $\Poiss(v,v_0,\omega|\lambda)$. This can be done via
probability theory. Indeed, for $u,w\in V$, let $F(u,w)$ be the hitting probability, i.e., the probability that the random walk generated by $P$ starting at $u$ visits $w$.  That is, denoting by $f^{(n)}(u,w)$  the probability of reaching $w$ for the first time after $n\geqslant 0$ steps after starting at $u$, we let
\begin{equation}\label{eq:F_as_a_series_of_first_visit_probabilities_at_time_n}
F(u,w)=\sum_{n\geqslant 0} f^{(n)}(u,w).
\end{equation}

Assume $u,v,w\in V$ with $v$ belonging to the chain from $u$ to $w$. Since the transition operator is nearest-neighbor,  $F$ has the multiplicativity property 
\begin{equation}\label{eq:multiplicativity_rule}
F(u,w)=F(u,v)F(v,w)
\end{equation}
by the obvious stopping time argument given by the fact that, on a tree, the random walk from $u$ to $w$ must visit all intermediate vertices, since there are no loops.
As in the homogeneous setting, $\Poiss(v,v_0,\omega) = d\nu_v/d\nu_{v_0}(\omega)$ and $\nu_v$, $\nu_{v_0}$ are hitting distributions on $\Omega$ of the Markov process (random walk) generated by $P$ \cites{KSK,Woess}; then it follows \cite{Cartier-Symposia} that
\begin{equation}\label{eq:Poisson_kernel_as_quotient_of_first_visit_probabilities}
\Poiss(v,v_0,\omega)=\frac{F(v,j)}{F(v_0,j)}
\end{equation}
where, in Subsection \ref{SubS:boundaries}, $j$ is any vertex in $[v,\omega)\cap[v_0,\omega)$.

On any tree, for every nearest-neighbor random walk, that is, with transition probability $p(u,v)=0$ unless $u\sim v$, $F(u,v)$ satisfies the following recurrence relation \citelist{\cite{Cartier-Symposia}*{(2.27)}\cite{KPT}*{(4)}\cite{Woess}}: if $u\sim v$,
\begin{equation}\label{eq:recurrence_relation_for_F}
F(u,v)=p(u,v)+\sum_{\substack{w\sim u\\w\neq v}} p(u,w)F(w,v)= p(u,v)+\sum_{\substack{w\sim u\\w\neq v}} p(u,w)F(w,u)F(u,v).
\end{equation}
Indeed,  the nearest-neighbor random walk starting at $u$ (at time 0) and ending at the first visit to $v$ must either reach $v\sim u$ at time 1 or else it must visit a different neighbor $w$ of $u$ at time 1 and from there return to $v$ later on.  
In the latter case $u$ is intermediate between $w$ and $v$, hence $F(w,v)=F(w,u)F(u,v)$. 
In the setting of a homogeneous tree,  $q_+=q_-=q$, for $u\sim v$ let us write $F=F(u,v)$. Then \eqref{eq:recurrence_relation_for_F} leads to
$F=\dfrac{qF^2+1}{q+1}$. 
The solutions are
\begin{equation}\label{eq:F_for_homogeneous_trees}
F=
\begin{cases}
1&\text{(incompatible with the transience of $\mu_1$);}
\\[.2cm]
1/q\,.
\end{cases}
\end{equation}

\subsection{Nearest-neighbor hitting probabilities in semi-homogeneous trees}
In the setting of isotropic semi-homogeneous setting, when $u\sim v$ we have $p(u,v)=p(u,w)=1/(q_u + 1)$. Moreover, $F(u,v)$ is the same for all neighbors $v$ of $u$ and depends only on the parity of $u$, 
hence there are only two values $F^+$, $F^-$ of $F(u,v)$ for all pairs $(u,v)$ with $u\sim v$, namely $F(u,v)=F^{\epsilon(u)}$. 
Therefore \eqref{eq:recurrence_relation_for_F} becomes
\begin{equation}\label{eq:recurrence_relation_for_semihomogeneous_visit_probability-0}
F^\pm=\frac1{q_\pm +1} (1+q_\pm F^\mp F^\pm).
\end{equation}
As a consequence of the exponential growth of the tree,
the random walk is transient, hence $F^\pm <1$. It follows that
\[
F^\pm= \frac1{1+ q_\pm (1- F^\mp)}.
\]
By putting together these two identities, we obtain two quadratic equations, one for $F^+$ and one for $F^-$.
The solutions for positive radical are $F^\pm=1$, not acceptable. The other solutions are
\[
F^\pm= \frac{q_\mp +1}{q_\pm +1} \;\frac1{q_\mp}\;.
\]
In the homogeneous case $q_+=q=q_-$, we obtain, as in \cite{Cartier-Symposia}*{identity (4.50)}, 
\begin{equation}\label{eq:homogeneous_first_visit_probability_to_neighbors}
F^+=F^-=1/q.
\end{equation}

\subsection{The generalized Poisson kernel at the eigenvalue $\gamma$ on a general tree}\label{SubS:generalized_Poisson_kernel}
We have seen in Proposition \ref{prop:semihomogeneous_K^z_is_not_eigenfunction} that the boundary integrals of $\Poiss^z$ are not eigenfunctions of $\mu_1$ except in the trivial cases $z=2k\pi i \ln \qmean$ or $1+2k\pi i \ln \qmean$. Then, what is the Poisson representation of these eigenfunctions? It is given by a generalized Poisson kernel $\Poiss(v, v_0, \omega\barra \gamma)$ different from $\Poiss^z(v,v_0,\omega)$. Again, the expression of this generalized Poisson kernel is given by potential theory, as we now outline for general trees, along the guidelines of \cites{Cartier-Symposia,Picardello&Woess,Picardello&Woess-London,KPT,DiBiase&Picardello}.

For any two vertices $u,v$, not necessarily neighbors, the first visit probability $F(u,v)$ is strictly related to the \emph{Green kernel} $G(u,v)$, that is the expected number of visits to $v$ of the random walk starting at $u$ induced by the transition operator $P$:
\begin{equation}\label{eq:Green_kernel_as_series_of_P^n}
G(u,v)=\sum_{n=0}^\infty p^{(n)}(u,v),
\end{equation}
 where the numbers $p^{(n)}(u,v)$  are the entries of the operator power $\mu_1^n$ and yield the probability of the event that the random walk starting at $u$ generated by $\mu_1$ moves from $u$ to $v$ in  $n$ steps. $\mu_1$, $F$ and $G$ will be  regarded as kernels and as operators on spaces of functions on $V$. The same stopping argument used before implies that 
\begin{equation}\label{eq:F_is_a_quotient_of_G}
G(u,v)=F(u,v)\,G(v,v).
\end{equation}
Therefore, by \eqref{eq:Poisson_kernel_as_quotient_of_first_visit_probabilities},
\begin{equation} \label{eq:Poisson_kernel_as_quotient_of_first_visit_probabilities-bis}
\Poiss(v,v_0,\omega)=\frac{F(v,j)}{F(v_0,j)} = \frac{G(v,j)}{G(v_0,j))}\;,
\end{equation}
where $j$ is any vertex in $[v,\omega)\cap[v_0,\omega)$, as in Subsection \ref{SubS:boundaries}.

On the other hand, clearly $G=\sum_{n=0}^\infty P^n$ satisfies the resolvent equation 
\[
GP=PG= \sum_{n=1}^\infty P^n = G-\mathI
\]
of the operator $P$ at the eigenvalue 1: more precisely, $-G$ is the resolvent of $P$ at the eigenvalue 1.

In the particular case of a semi-homogeneous tree, the operator $P=\mu_1$ is not self-adjoint on $\ell^2(V)$: indeed, $\mu_1^*\neq \mu_1$ because $p(u,v)\neq p(v,u)$ for $u\sim v$. But $\mu_1$ is self-adjoint on $\ell^2(V,m)$, where $m$ is the weight 
$m(v)=\sqrt{(\widetilde q_v +1)/(q_v +1)}$.

If $|v|=n\in\mathN$ and $[v_0, v_1,\dotsc, v_n=v]$ is the path from $v_0$ to $v$, then 
\begin{equation}\label{eq:the_reversibility_weight}
m(v)=\sqrt{\dfrac{\widetilde q_{v_0}+1}{q_{v_0}+1}}\;\dfrac{p(v_0,v_1)\cdots p(v_{n-1},v_n)}{p(v_1,v_0)\cdots p(v_n,v_{n-1})}=\begin{cases}
\sqrt{\dfrac{q_{v_0}+1}{\widetilde q_{v_0}+1}} &\quad\text{ if $n$ is odd,}
\\[.6cm]
\sqrt{\dfrac{\widetilde q_{v_0}+1}{q_{v_0}+1}} &\quad\text{ if $n$ is even.}
\end{cases}
\end{equation}
\begin{remark}\label{rem:reversible}
Actually, self-adjointness of a nearest-neighbor stochastic transition operator $P$ on $\ell^2(V,m)$ holds in a much greater generality than the setup of homogeneous or semi-homogeneous trees: indeed, it holds on any graph if and only the operator is reversible, where reversible means that $p(w,v)\neq 0$ whenever $p(v,w)\neq 0$ (see \cite{Woess}*{Chapter 4, Example 4.3}).

 Since $m$ is bounded above and below by positive constants, the norms in $\ell^2(V,m)$ and $\ell^2(V)$ are equivalent, hence the spectra of $P$ as an operator on each of these Hilbert spaces coincide. Moreover, since $P$ is self-adjoint on $\ell^2(V,m)$, the spectrum is a subset of $\mathR$, hence every point in it is in the closure of the resolvent set.
 
  For $p\ne 2$, the spectrum is not real. However, on every tree,
for each $p\geqslant 1$  the spectrum of $\mu_1$ in $\ell^{p}(V)$ is symmetric around the origin, because  a semi-homogeneous tree is a bipartite graph (the  set of vertices is the union of  two disjoint sets and each edge is represented by to  a pair of vertices in the two different sets)
 \cite{Mohar} (see also \cite{Mohar&Woess}*{Theorem 4.8}).
We see in Figure \ref{Fig:semi-homogeneous_bounded_spherical_functions} that the spectrum in $\ell^{p}$ is disconnected for $p$ near 2, but connected for $p$ near 1. 
\end{remark}

Let $\rho(P)$ be the spectral radius of $P$ in $\ell^2(V)$;  in the corona $\barra \gamma|>\rho(P)$ the resolvent is a power series in $\gamma$, and its analytic continuation yields the resolvent $G(v, v_0\barra \gamma)$ \emph{(generalized Green kernel)}
outside the spectrum of $P$ on $\ell^2(V)$. The multiplicativity rule  \eqref{eq:F_is_a_quotient_of_G}
extends to the positive eigenvalues in the $\ell^2$-spectrum of the Laplacian, by the argument of  \cite{Woess}*{Theorem 1.38}, as follows. For $\gamma>\rho(P)$,
 \eqref{eq:Green_kernel_as_series_of_P^n} becomes
\begin{equation*}
G(x,y\barra \gamma)=\sum_{n\geqslant 0} p^{(n)}(x,y) \gamma^{-n-1},
\end{equation*}
 \eqref{eq:F_as_a_series_of_first_visit_probabilities_at_time_n} becomes
\begin{equation*}
F(x,y\barra \gamma)=\sum_{n\geqslant 0} f^{(n)}(x,y) \gamma^{-n},
\end{equation*}
whence
\begin{equation} \label{eq:F(x,x|gamma)=1}
F(x,x\barra \gamma)=1\quad \text{for every }\gamma.
\end{equation}

Note that $p^{(n)}(x,y)$ and $f^{(n)}(x,y)$ vanish if $n$ and $\dist(x,y)$ have opposite parities; therefore $G(x,y\barra\missarg)$ is even and $F(x,y\missarg)$ is odd for all $x\sim y$. 
Moreover, the functions $F$ and $G$ satisfy, for every $\barra \gamma|>\rho(P)$ and $u,v,w\in V$, 
\begin{align}\label{eq:non_homogeneous_multiplicativity}
F(u,v\barra \gamma) &= F(u,w\barra \gamma) F(w,v\barra \gamma),\\[.2cm]
G(u,v\barra \gamma)&=F(u,v\barra \gamma) \,G(v,v\barra \gamma),\notag
\end{align}
whence
$F(x,y\barra \gamma)=F(y,x\barra \gamma)$ for every $x,y\in V$.
In this notation, $G(x,y)=G(x,y\barra 1)$ and $F(x,y)=F(x,y\barra 1)$. 

Moreover, at the eigenvalue $\gamma$, the operator $G_\gamma$ with kernel $G(x,y\barra \gamma)$ satisfies the resolvent equation
\begin{equation}\label{eq:the_resolvent_equation}
P G_\gamma = \gamma G_\gamma - \mathI
\end{equation}
and, if $P$ is an isotri=opic operator,
\eqref{eq:recurrence_relation_for_F} becomes: for every $u,v\in V$ with $u\sim v$ and $\barra \gamma|>\rho(P)$,
\begin{equation}\label{eq:the_resolvent_equation_for_F}
\gamma F(u,v\barra \gamma)= \frac 1{q_u +1} \Bigl(1 + \sum_{\substack{w\sim u\\w\neq v}}    F(w,u\barra \gamma)\,F(u,v\barra \gamma) \Bigr). \end{equation}
It follows immediately from \eqref{eq:the_resolvent_equation} that
$P F(\missarg,v_0\barra\gamma)=\gamma
F(\missarg,v_0\barra\gamma)$ except at $v_0$

For every $v$, the resolvent $G(v_0,v\,|\missarg)$ has an analytic continuation outside the $\ell^2$-spectrum $\spectrum(P)$ of $P$.
Therefore, outside of the spectrum,  we have the identity
\[
F(v_0,v\barra \gamma) = \frac {G(v_0,v\barra \gamma)}{G(v,v\barra \gamma)}
\]
except at points $\gamma$ that satisfy the equation $G(v,v\barra \gamma)=0$;
this equation does not hold in the corona
 $\barra \gamma|>\rho(P)$, and more generally   outside $\spectrum_{\ell^2(V)}(P)$, as an easy adaptation of \cite{Picardello&Woess-PotentialAnalysis}*{Section 4} shows. This also shows (see \cite{Picardello&Woess-PotentialAnalysis}*{identities (2.6), (2.7)}) that
 \begin{corollary}\label{cor:G_and_F_non-zero}
If $\gamma\notin\spectrum_{\ell^2(V)}(P)$ then  $G(v,w\barra \gamma)$ and $F(v,w\barra \gamma)$ are both finite and non-zero for every $v\sim w $.
 \end{corollary}

 Thus, for $\gamma\notin\spectrum_{\ell^2(V)}(P)$, we can define the \emph{generalized Poisson kernel}
\begin{equation}\label{eq:generalized_Poisson_kernel}
\Poiss(v,v_0,\omega\barra \gamma) = \frac{F(v,j\barra \gamma)}{F(v_0,j\barra \gamma)}\; = \frac{G(v,j\barra \gamma)}{G(v_0,j\barra \gamma)}\,
\end{equation}
where $j$ is any vertex in $[v,\omega)\cap[v_0,\omega)$.
In particular, by writing $[v_0,\omega)=[v_0,\dotsc,v_n,\dotsc)$,
\begin{equation*}
\Poiss(v,v_0,\omega\barra \gamma) = \lim_{n\to\infty}\frac{F(v,v_n\barra \gamma)}{F(v_0,v_n\barra \gamma)}
=
\lim_{n\to\infty}\frac{G(v,v_n\barra \gamma)}{G(v_0,v_n\barra \gamma)},
\end{equation*}
since the sequences on the right-hand side are definitively constant.
Now \eqref{eq:the_resolvent_equation} shows that 
\begin{equation}\label{eq:the_generalized_Poisson_kernel_is_an_eigenfunction}
P \Poiss(\missarg,v_0,\omega\barra\gamma)=\gamma \,\Poiss(\missarg,v_0,\omega\barra\gamma),
\end{equation}
that is, $\Poiss(\missarg,v_0,\omega\barra\gamma)$ is a $\gamma$-eigenfunction of $P$;
moreover
\begin{equation}\label{eq:the_generalized_Poisson_kernel_is normalized_at_v_0}
\Poiss(v_0,v_0,\omega\barra \gamma) =1 \qquad\text{ for every}\quad \omega\in\Omega, \, \gamma\notin\spectrum_{\ell^2(V)}(P),
\end{equation}
confirming the case $\gamma=1$, already known from \eqref{eq:Poisson_kernel_as_Radon-Nikodym_derivative}.

All this was obtained in \cite{Cartier-Symposia} on a general tree. This led to a characterization of eigenfunctions of a nearest neighbor transition operator $P$ on a tree as Poisson integrals of distributions on $\Omega$ (in the present context $P=\mu_1$), obtained for positive eigenfunctions on general trees in \cite{Cartier-Symposia}, for all complex eigenfunctions on a homogeneous tree with isotropic nearest-neighbor transition operator in \cites{Mantero&Zappa, Figa-Talamanca&Picardello}, and for group-invariant nearest-neighbor non-isotropic transition operators on homogeneous trees in \cite{Figa-Talamanca&Steger} (see also \cite{Casadio-Tarabusi&Gindikin&Picardello-Book}). 
For a wide class of transient nearest-neighbor operators on general locally finite trees, the Poisson representation of harmonic functions and its inverse is in \cite{Picardello&Woess}; see also
\cites{Woess, Picardello&Woess-PotentialAnalysis, Picardello&Woess-London} and references therein.

 \section{The generalized Poisson kernel on a semi-homogeneous tree}
\label{Sec:Semi-homogeneous_generalized_Poisson_kernels}
\subsection{Generalized hitting probabilities  at the eigenvalue $\gamma$ on a semi-homogeneous tree}\label{SubS:generalized_Poisson_kernel_on_general_trees}
We now return to  the environment of homogeneous or semi-homogeneous trees. We only need radial eigenfunctions, whose boundary distributions are constant. We now compute $F(u,v\barra \gamma)$ when $u$ and $v$ are neighbors; as before, this takes only two values $F^+(\gamma),\,F^-(\gamma)$ according to the parity of $u$, that are odd functions of $\gamma$. 
If
$u_0\sim u_1\sim u_2\dots\sim u_n$ are consecutive vertices, then,
by
\eqref{eq:non_homogeneous_multiplicativity},
\begin{equation}\label{eq:mu;toplicativity_along_paths_in_terms_of_F+,F-}
F(u_0,u_n) = \prod_{j=1}^n F(u_{j-1},u_j)=
F^\pm(\gamma)^{\lceil  \frac n2 \rceil }
F^\mp(\gamma)^{\lfloor \frac n2 \rfloor}\qquad\text {if $u_0\in V_\pm$, or also if $u_n\in V_\mp$}.
\end{equation}
In the sequel we shall use the shorthand
\begin{align}
\label{eq:definition_of_newGamma}
\newGamma =(q_+ +1)(q_-+1)\gamma^2.
\end{align}
Identity
\eqref{eq:recurrence_relation_for_semihomogeneous_visit_probability-0} can be rewritten as 
\begin{equation}\label{eq:recurrence_relation_for_semihomogeneous_visit_probability-0-for_all_eigenvalues}
\gamma F^\pm(\gamma)
=\frac1{q_\pm +1} \bigl(1+q_\pm F^+(\gamma) F^-(\gamma)\bigr),
\end{equation}
where $\gamma\in\mathC\setminus \spectrum_{\ell^2(V)}(\mu_1)$.
We shall solve this equation later on, but
several interesting properties of $F^\pm(\gamma)$ can be deduced directly without solving it:
\begin{proposition}\label{prop:F_nonzero_eccetto_a_0}
\leavevmode
 \begin{enumerate}
  \item[$(i)$]
 $F^+$ and $F^-$ are odd functions;
 \item[$(ii)$] 
 If $\gamma\in\mathC\setminus \spectrum_{\ell^2(V)}(\mu_1)$, then
$
F^\pm(\gamma) $ 
is finite and non-zero for   $\gamma\neq 0
$;
  \item[$(iii)$]
 $\lim_{\gamma'\to\gamma} F^+(\gamma')\,F^-(\gamma')\neq 0$ along every path $\gamma'\to\gamma$  for which the limit exists. 
\item[$(iv)$]
if $\gamma\neq 0$, then, along a given path $\mathcal C$ converging to $\gamma$, or any sequence converging to $\gamma$, either all of the functions $F^+$, $F^-$ and $F^+F^-$ have limit (finite or infinite) or none does; if one of these three limits is finite, then they all are; if one of the three is non-zero, then they all are non-zero. 

\item[$(v)$] if $\gamma=0$ and $q_+\neq q_-$, then $F^+$ cannot have finite non-zero limit along any path $\mathcal C$ converging to $\gamma=0$. If $F^+$ has limit zero along such a path, then $F^+F^-$ tends to $-1/q_+$ and $F^-$ diverges;  $F^+$ cannot be bounded and bounded away from zero along such a path, indeed if $F^+$ is bounded on $\mathcal C$ then it tends to 0 along $\mathcal C$.

\item[$(vi)$] if $\gamma=0$ and $q_+= q_-=q$, then, if $F^+=F^-=F$ tends to a finite limit along a path convergent to 0, then the limit is $i/\sqrt q$ or $-i/\sqrt q$; also under the weaker assumption that $|F(\gamma)|=o(1/|\gamma|)$ for $\gamma\sim 0$ along the path, then  $F$ tends to $\pm i/\sqrt q$
 (see  \eqref{eq:behavior_of_F(gamma)_near_gamma=0} for an extension of this  result).

\item[$(vii)$]
for $\gamma\neq 0$,  $\lim_{\gamma'\to\gamma} F^\pm(\gamma')/F^\mp(\gamma')$ along a path exists and is finite if and only if, along the same path,  $\lim_{\gamma'\to\gamma} F^\mp(\gamma')\neq 0$.
\end{enumerate}
%
\end{proposition}
\begin{proof} We have already proved $(i)$. All the rest
 follows directly from \eqref{eq:recurrence_relation_for_semihomogeneous_visit_probability-0-for_all_eigenvalues}, that will be used in each step without explicit reference.

$(ii)$ If $\gamma\neq 0$ and $F^+(\gamma)=0$, then $F^-$ must be singular at $\gamma$ and $F^+F^-$ must have a non-zero finite limit. But then, reversing the role of $F^\pm$ in this argument, also $F^+$ should be singular at $\gamma$, a contradiction. 

$(iii)$ If $F^+(\gamma')\,F^-(\gamma')$ tends to zero along a path $\gamma'\to\gamma$, then 
both $F^\pm(\gamma)$ tend to a finite non-zero limit, a contradiction.

$(iv)$ If $\gamma\neq 0$ and the function $F^+F^-$ has a finite non-zero limit along a path convergent to $\gamma$, then 
also $F^+$ and $F^-$ have finite limits and of course these limits must be non-zero, whereas, if $F^+F^-$ has limit zero, then both $F^\pm$ must have finite non-zero limit, a contradiction. Conversely, if one among $F^+$ and $F^-$, say $F^+$, tends to zero, then it follows 
that 
 $F^+F^-$ has a finite non-zero limit; then also the other factor $F^-$ has finite limit, a contradiction.
Now suppose that $F^+$ has a finite non-zero limit along the path.
Then $F^+F^-$ must also have a finite limit, and so does $F^-$. If $F^+F^-$ tends to zero, then  also
$F^-$ must tend to zero, but then we have seen in the previous argument that this leads to a contradiction.

$(v)$ Let $\gamma=0$: the same argument shows that, if $F^+$ has a finite limit $\ell$ along a path to $\gamma$, then  $F^+F^-$
has a finite non-zero limit because the left-hand side of  \eqref{eq:recurrence_relation_for_semihomogeneous_visit_probability-0-for_all_eigenvalues} tends to zero. Thus, if $\ell\neq 0$, then also $F^-$ must have  finite  non-zero limit. Then it follows from the two equations in the identity 
that
$F^+F^-$ tends to $-1/q_\pm$, respectively: this leads to a contradiction since $q_+\neq q_-$.

On the other hand, if $\ell=0$, then $|F^-|$ must diverge along the path in order to let $F^+F^-$ have finite limit.
Finally,
 if $|F^+|$ is bounded along $\mathcal C$, let $\ell=\limsup |F^+(\gamma)|$ for $\gamma\in\mathcal C$. There is a sequence $\{\gamma_n\}$ in the path $\mathcal C$ such that $\lim_n |F^+(\gamma_n)| =\ell$. Then, by the previous argument,
$\ell$ must be zero, hence $F^+$ tends to zero along $\mathcal C$.
 
 $(vi)$ follows  from the same argument.
 
$(vii)$ Let us rewrite \eqref{eq:recurrence_relation_for_semihomogeneous_visit_probability-0-for_all_eigenvalues} as
$\bigl( (q_\pm +1)\gamma F^\pm(\gamma)-1\bigr)/q_\pm= F^+(\gamma)F^-(\gamma)$. This is a system of two equations with the same right-hand sides, whence elementary algebraic manipulations yield $ q_- (q_+ +1)\gamma - \bigl(q_+ (q_- +1)\bigr) \gamma F^-(\gamma)/F^+(\gamma) = (q_- - q_+) / F^+(\gamma)$, whence $(v)$.
\end{proof}

Observe that, by multiplying  side by side the two equations \eqref{eq:recurrence_relation_for_semihomogeneous_visit_probability-0-for_all_eigenvalues}, we obtain a quadratic equation in the unknown $F^+(\gamma)F^-(\gamma)$: this explains why there are two solutions in parts $(v)$ and $(vi)$ of Proposition \ref{prop:F_nonzero_eccetto_a_0}, and confirms the value found there at the limit for $\gamma=0$.

The solution of the quadratic system \eqref{eq:recurrence_relation_for_semihomogeneous_visit_probability-0-for_all_eigenvalues} involves the square root of 
\begin{equation}\label{eq:quartic_polynomial_Z}
\begin{split}
\calW &=\newGamma^2-2(q_+ +q_-)\newGamma+(q_+ - q_-)^2
\\[.2cm]
&
=
\bigl(\newGamma-(q_++q_-)\bigr)^2-4\qmean^2
=\bigl(\newGamma + (q_+-q_-) \bigr)^2-4q_+\newGamma,
\end{split}
\end{equation}
which, if $q_+\neq q_-$, has simple zeros at
\begin{equation}\label{eq:endpoints_of_the_L2-spectrum_of_M1}
a=\dfrac {\abs{\sqrt {q_+} - \sqrt {q_-}}} {\sqrt{(q_+ +1)(q_- +1)}}\,,\qquad
b=\dfrac {\sqrt {q_+} + \sqrt {q_-}} {\sqrt{(q_+ +1)(q_- +1)}},\qquad -a,\qquad-b,
\end{equation}
while, if $q_+=q_-$, then $b,-b$ are simple zeros and $a=0=-a$ is a double zero. 
For  $\gamma\in\mathR$,  $\calW$ is positive   outside the closed intervals $[-b,-a]$ and $[a,b]$ and negative in the open intervals; in the homogeneous setting, the two closed intervals coalesce to the unique interval $[-b,b]$ and $\mathcal Z(0)=0$.
Therefore $\sqrt{\calW}$ is a two-valued analytic function on the complex plane except at those zeros, where of course it vanishes.  Then:
\begin{definition}\label{def:the_root_R}
With $\calW$ as in \eqref{eq:quartic_polynomial_Z}, we set
\begin{equation*}
\theroot=\sqrt{\calW}
\end{equation*}
as the determination with branch cuts on the intervals $(-b,-a)$ and $(a,b)$, and
 positive for real $\gamma>b$. 
\end{definition}
\begin{proposition}\label{prop:change_of_sign_in_analytic_continuation_of_a_square_root}
The function $\theroot$ is even, and it positive real on $(-\infty,-b)\cup(b,+\infty)$ whereas, if $q_+\neq q_-$, it is negative real on the interval $(-a,a)$.
\end{proposition}

\begin{proof}
The two-valued function $\sqrt{\calW}$ is even because so is $\calW$, and its two determinations have opposite values. Therefore each of the two determinations is either an even or an odd function on $\mathC\setminus  (-b,-a)\cup (a,b)$. On the other hand, $\mathcal R(0)\neq 0$ because $\mathcal W(0)\neq 0$, therefore $\theroot$ cannot be odd.

By its definition, $\theroot$ is positive for $\gamma>b$. Let us consider a continuous closed loop $\mathcal C$ that winds once around the interval $[a,b]$. Since $a$ and $b$ are simple zeros of $\calW$, the root $\theroot$ maintains the same determination at the end of the loop. On the other hand, $\overline \calW=\mathcal W(\overline\gamma)$ because
$\calW$ is a polynomial with real coefficients, therefore
the argument of $\theroot$, as the loop crosses the interval $(-a,a)$, changes by $\pm \pi$ with respect to when the loop crosses $(b,+\infty)$.
\end{proof}

\subsection{Asymptotics of the nearest-neighbor generalized hitting probabilities}
With $\theroot$ as in Definition \ref{def:the_root_R}, the solutions of the quadratic system \eqref{eq:recurrence_relation_for_semihomogeneous_visit_probability-0-for_all_eigenvalues} are
\begin{align}\label{eq:n.n._semihomogeneous_visit_probabilities_at_eigenvalue_gamma}
F^\pm(\gamma) 
&=\dfrac{\newGamma\mp( q_+ - q_-) - \theroot}
{2 q_\mp (q_\pm+1)\gamma}
\\[.2cm]
\label{eq:n.n._semihomogeneous_visit_probabilities_at_eigenvalue_gamma-not_acceptable}
\widetilde F^\pm(\gamma) &=
\dfrac{\newGamma\mp( q_+ - q_-) + \theroot}
{2 q_\mp (q_\pm+1)\gamma}.
\end{align}%
The functions $F^+$ and $\widetilde F^+$ are the two different determinations of the same two-valued meromorphic function, hence analytic continuation of each other; the same is true for $F^-$, $\widetilde F^-$. Therefore  $\widetilde F^\pm$ are odd functions, by Proposition \ref{prop:F_nonzero_eccetto_a_0}\,$(i)$.
Moreover, $F^\pm(\gamma)$ and $\widetilde F^\pm(\gamma)$ are positive for real $\gamma>b$.

If $q_+\neq q_-$ and 
$a, b$ are as in \eqref{eq:endpoints_of_the_L2-spectrum_of_M1}, 
then
the zeros of the radicand $\calW$ 
 are the four real  numbers $-b$, $-a$, $a$, $b$ (and, for $\gamma\in\mathR$,  $\qFF\notin\mathR$ if and only if $\gamma\in[-b,-a]\cup[a,b]$). Therefore the above-described determination for $F^\pm(\gamma)$ can be uniquely continued along any path in $\mathC\setminus\{0\}$ from the corona to  $0$ that does not cross the real segments $[-b,-a]$ and $[a,b]$; on $\{\gamma\in\mathR\colon0<\abs{\gamma}<a\}$ the expression is still given by~\eqref{eq:n.n._semihomogeneous_visit_probabilities_at_eigenvalue_gamma} as before.

Since $\gamma$ appears at the denominator in \eqref{eq:n.n._semihomogeneous_visit_probabilities_at_eigenvalue_gamma}, it remains to assess the behavior of $F^\pm(\gamma)$ at $\gamma\sim 0$.

\begin{proposition} \label{prop:what_happens_to_F^pm_for_gamma<>0}
\leavevmode
\begin{enumerate}
\item[$(i)$]
The functions $F^\pm$ are continuous in $\mathC\setminus\bigl( (-b,-a)\cup \{0\} \cup(a,b)\bigr)$.
\item[$(ii)$]
 If $\gamma\neq 0$, then
$\lim_{\gamma'\to\gamma} |F^{\pm}(\gamma')|$ is finite and non-zero.  
\item[$(iii)$]
$ F^\pm(\gamma)=0$
  if and only if $\gamma=0$ 
  and $q_\pm > q_\mp$. 
  On the other hand, $\lim_{\gamma'\to\infty} F^\pm(\gamma')=0$.
  \item[$(iv)$]
  \begin{equation}\label{eq:behavior_of_F^+F^-(gamma)_near_gamma=0}
\lim_{\gamma\to 0} F^+(\gamma)F^-(\gamma)  = -\frac1{\max\{q_+,q_-\}}\;
\end{equation}

\item[$(v)$]
For every $\gamma$, $F^\pm(\gamma)\,\widetilde F^\pm(\gamma) =
\dfrac{q_\mp +1 } 
{q_\mp(q_\pm+1)}$.

\item[$(vi)$]
$ \lim_{\gamma'\to\gamma} \widetilde F^\pm(\gamma)=\infty$ only if $\gamma=\infty$ and if $q_\pm > q_\mp$ and
$\gamma=0$.

  \item[$(vii)$]
 $\lim_{\gamma'\to \gamma}|F^\pm(\gamma')|=| \lim_{\gamma'\to \gamma} \widetilde F^\pm(\gamma')|=\dfrac{\sqrt{(q_++1)(q_-+1)}}
{ \sqrt{q_\mp} (q_\pm+1)}$  for every $\gamma\in  (-b,-a)\cup \{0\} \cup(a,b)$. 
\item[$(viii)$]
Finally, $|\widetilde F^-(\gamma)|\geqslant |F^-(\gamma)|$  for every $\gamma$ if $q_+ < q_-$, 
and 
$|\widetilde F^+(\gamma)|\geqslant |F^+(\gamma)|$  for every $\gamma$ if $q_+ > q_-$.
\end{enumerate}

By $(iii)$, the choice of determination of $\theroot$  in \eqref{eq:n.n._semihomogeneous_visit_probabilities_at_eigenvalue_gamma}
makes $F^+$ and $F^-$  compatible with the probabilistic interpretation of generalized hitting probabilities (see \cite{Woess}). Instead, by $(vi)$, the other solution of \eqref{eq:recurrence_relation_for_semihomogeneous_visit_probability-0-for_all_eigenvalues} diverges at infinity and does not share a probabilistic meaning.
\end{proposition}

\begin{proof}
Part $(i)$ follows from  \eqref{eq:n.n._semihomogeneous_visit_probabilities_at_eigenvalue_gamma} and the subsequent comments.

Part $(ii)$ follows from Proposition \ref{prop:F_nonzero_eccetto_a_0}\,$(ii)$. 
We have seen that, for small real $\gamma$, the chosen determination of the square root in \eqref{eq:n.n._semihomogeneous_visit_probabilities_at_eigenvalue_gamma} is negative. Therefore the square root  tends to $-|q_+ - q_-|$ as $\gamma\to 0$;
hence,  for $q_+<q_-$ we see that $F^+(\gamma)$ has a pole of order $1$ as $\gamma\to 0$, namely
\begin{equation}\label{eq:F^+(0)_diverges}
F^+(\gamma) = \frac{q_- - q_+}{q_- (q_+ +1)}\;\frac 1\gamma + O(\gamma)\quad\text{for $q_+<q_-$}.
\end{equation}

Moreover, if $q_+<q_-$ the terms in the expression of $F^-(\gamma)$
that are divergent for $\gamma\to 0$ cancel  out, and $\lim_{\gamma\to 0} F^-(\gamma)=  0 $ linearly in $\gamma$. More precisely, 

\begin{equation}\label{eq:F^-_vanishes_linearly_as_gamma->0}
 F^-(\gamma)=  -\frac{q_+ +1}{q_- - q_+}\;\gamma + O(\gamma^3) \quad\text{for $q_+<q_-$}, 
\end{equation}
because for real $\gamma\to 0$, the square root in \eqref{eq:n.n._semihomogeneous_visit_probabilities_at_eigenvalue_gamma} is negative and one has
\begin{equation*}
\frac{\sqrt{\newGamma^2-2(q_+ +q_-)\newGamma+(q_+ - q_-)^2}} {2 q_+ (q_- +1)\gamma}
= 
\frac{1}{2 q_+(q_- +1)\gamma}\biggl(q_+-q_--\frac{q_+ +q_-}{q_+ -q_-}\newGamma \biggr) +O(\gamma^3).
\end{equation*}

Of course, 
the opposite happens if $q_+ > q_-$: 
\begin{equation}\label{eq:F^+(0)_and_F^-(0)_for_q+<q-}
 F^+(\gamma)=  -\frac{q_- +1}{q_+ - q_-}\;\gamma + O(\gamma^3), 
\qquad
F^-(\gamma) = \frac{q_+ - q_-}{q_+ (q_- +1)}\;\frac 1\gamma + O(\gamma)\quad\text{for $q_+>q_-$}.
\end{equation}
Analogous asymptotic estimates hold for $|\gamma|\to\infty$. Indeed,
by \eqref{eq:quartic_polynomial_Z}, for $|\gamma|\to\infty$, $\calW=\newGamma^2\Bigl(1+)(1/|\gamma|)$. Then it follows that $\lim_{|\gamma|\to\infty} F^\pm(\gamma)=0$.
This proves part $(iii)$.

Part [$(iv)$] follows from \eqref{eq:F^+(0)_diverges},  \eqref{eq:F^-_vanishes_linearly_as_gamma->0} and \eqref{eq:F^+(0)_and_F^-(0)_for_q+<q-}.
%
All this is in accord with Proposition \ref{prop:F_nonzero_eccetto_a_0}\,$(v)$.

For every $\gamma$,
by \eqref{eq:quartic_polynomial_Z},
\begin{align}\label{eq:F_per_Ftilde}
F^\pm(\gamma)\,\widetilde F^\pm(\gamma) &=
\dfrac{
(\newGamma\mp( q_+ - q_-))^2 - \calW
}
{(2 q_\mp (q_\pm+1)\gamma )^2}
=
\dfrac{q_\mp +1 } 
{q_\mp(q_\pm+1)}\;.
\end{align}
This proves $(v)$, and $(vi)$ follows immediately from $(iii)$.
To prove $(vii)$, let us consider the expression  $\dfrac{\newGamma\mp( q_+ - q_-) + \sqrt{\calW}}
{2 q_\mp (q_\pm+1)\gamma}
$ that appears in \eqref{eq:n.n._semihomogeneous_visit_probabilities_at_eigenvalue_gamma-not_acceptable} for $\widetilde F^\pm(\gamma)$, or \eqref{eq:n.n._semihomogeneous_visit_probabilities_at_eigenvalue_gamma} for $F^\pm(\gamma)$, according to the sign of the square root.
If
$\gamma\in(-b,-a)\cup(a,b)$ then $\calW<0$ and
$\sqrt{\calW}=i\sqrt{-\calW}$ are two purely imaginary values, namely, by \eqref{eq:quartic_polynomial_Z},
$\pm i\sqrt{-\newGamma^2+2(q_+ +q_-)\newGamma-(q_+ - q_-)^2}$.
 Therefore,
  by the last identity in \eqref{eq:quartic_polynomial_Z}, the moduli
  $|F^\pm|$ and $|\widetilde F^\pm|$ can be uniquely extended by continuity to
  $(-b,-a)\cup(a,b)$, and 
\begin{equation}\label{eq:|F(gamma)|=constant_in_the_spectrum}
\begin{split}
\lim_{\gamma'\to \gamma}|F^\pm(\gamma')|&=| \lim_{\gamma'\to \gamma} \widetilde F^\pm(\gamma')|=\frac{\sqrt{ (\newGamma \mp (q_+ -q_-))^2 +
(-\calW)}}
{2 q_\mp (q_\pm+1)|\gamma|}
=
\frac{\sqrt{ 4q_\mp\newGamma}}
{2 q_\mp (q_\pm+1)|\gamma|}
\\[.2cm]
&=\dfrac{\sqrt{(q_++1)(q_-+1)}}
{ \sqrt{q_\mp} (q_\pm+1)}
\qquad\text{if $\gamma\in[-b,-a]\cup[a,b]$}.
\end{split}
\end{equation}

Part $(viii)$ follows from $(iii)$
since $F^-$ is analytic except at 0 and in the intervals $[-b,-a]$ and $[a,b]$, whence it follows from the maximum principle that, for $q_+<q_-$, the maximum value of $|F^-|$ and the minimum value of $|\widetilde F^-|$ 
are attained in these intervals, where both these functions are constant and equal. 
Symmetrically, if $q_+ > q_-$, then $|\widetilde F^+(\gamma)|\geqslant |F^+(\gamma)|$.
\end{proof}

\begin{corollary}\label{cor:0_is_in_the_l^2_spectrum}
For every $q_+$ and $q_-$, $0$ belongs to
$\spectrum(\mu_1;\,\ell^2(V))$
(see also Corollary \ref{coro:generalized_semi-homogeneous_Poisson_kernel} and Theorem \ref{theo:boundedness_of_F+(gamma)F-(gamma)}).
\end{corollary}

\begin{proof}
In view of
\eqref{eq:F^+(0)_diverges},  \eqref{eq:F^-_vanishes_linearly_as_gamma->0} and \eqref{eq:F^+(0)_and_F^-(0)_for_q+<q-}, for every $q_+\neq q_-$ the statement follows
 from  Corollary \ref{cor:G_and_F_non-zero}.
If $q_+=q_-$, the same result holds  by  Remark 
 \ref{rem:reversible}, since, in the homogeneous setting, the spectrum is known to be connected (see \cite{Figa-Talamanca&Picardello} and references therein).
\end{proof}

\begin{remark}
It follows easily from \eqref{eq:F_per_Ftilde}, \eqref{eq:F^+(0)_diverges}, Corollary \ref{cor:0_is_in_the_l^2_spectrum} and \eqref{eq:F^+(0)_and_F^-(0)_for_q+<q-} that the asymptotic behavior of $F^\pm(\gamma)$ as $\gamma\to 0$ when $q_+>q_-$ is respectively the same as that of $\widetilde F^\pm(\gamma)$ when $q_->q_+$.
\end{remark}

In the homogeneous setting with $q=q_+=q_-$ we write $F=F^+=F^-$ and, for $\gamma^2>{4q}/{(q+1)^2}$ (the square of the spectral radius), we obtain from \eqref{eq:n.n._semihomogeneous_visit_probabilities_at_eigenvalue_gamma},  with the choice of the determination of the square root made there,
\begin{equation}\label{eq:n.n._homogeneous_visit_probabilities_at_eigenvalue_gamma}
F(\gamma) = \frac{(q +1)\gamma\
 - \sqrt{(q+1)^2\gamma^2-4q}} {2 q }\;,
\end{equation}
in agreement with \cite{Sava&Woess}*{formula (5)}; as in \eqref{eq:F_for_homogeneous_trees}, $F(1)=1/q$.
Hence there is only one cut, namely the interval $[-2\sqrt{q}/(q+1), 2\sqrt{q}/(q+1)]=[-\rho(\mu_1),\rho(\mu_1)]$,
and at $\gamma=0$ there are two possible analytic continuations according to whether $0$ is approached from above or below. They respectively yield
\begin{equation}\label{eq:behavior_of_F(gamma)_near_gamma=0} 
\lim_{\gamma \to 0} F(\gamma) = \pm\frac{i}{\sqrt{q}}\;,
\end{equation}
in agreement with \eqref{eq:behavior_of_F^+F^-(gamma)_near_gamma=0} and improving Proposition \ref{prop:F_nonzero_eccetto_a_0},$(vi)$ .

We now set 
\begin{equation}
\begin{split}
\label{eq:definition_of_qFF}
\qFF &=\qmean\,F^+(\gamma)\,F^-(\gamma),\\[.1cm]
\tildeqFF&=\qmean\widetilde F^+(\gamma)\widetilde F^-(\gamma).
\end{split}
\end{equation}
Then, by \eqref {eq:n.n._semihomogeneous_visit_probabilities_at_eigenvalue_gamma} and the last identity in \eqref{eq:quartic_polynomial_Z},  
\begin{equation}\label{eq:F+(gamma)F-(gamma)}
\begin{split}
\qFF
&=\dfrac{\newGamma-(q_+ +q_-)-\theroot}{2\qmean}
=\dfrac{\newGamma+q_+ -q_--\theroot}{2\qmean}
-\dfrac{q_+}\qmean,
%
\\[.2cm]
\tildeqFF
&=\dfrac{\newGamma-(q_+ +q_-)+\theroot}{2\qmean}
=\qFF^{-1}.
\end{split}
\end{equation}
Note that
\begin{equation}\label{eq:quadratica_equation_for_qFF}
\begin{split}
\qFF^2&=\dfrac{\newGamma-(q_+ +q_-)}{\qmean}\;\qFF -1,
\\[.2cm]
\tildeqFF^2 &=\dfrac{\newGamma-(q_+ +q_-)}{\qmean}\;\tildeqFF -1,
\end{split}
\end{equation}
and that, by the first identity in the first equation above and \eqref{eq:n.n._semihomogeneous_visit_probabilities_at_eigenvalue_gamma},
\begin{equation}\label{eq:B/qmean_in_terms_of_Fminus}
\frac{\qFF}{\qmean}=\frac{(q_\pm +1)\gamma F^\pm(\gamma) -1}{q_\pm}.
\end{equation}

By Definition \ref{def:the_root_R}, $\qFF$ is defined for $\gamma\in\mathC\setminus\bigl((-b,-a)\cup(a,b)\bigr)$; for  $\gamma\in(-b,-a)\cup(a,b)$,  $\calW$ is  negative and $\theroot$ is not defined.
In particular, $\qFF$ exists for every $\gamma\in\mathC\setminus\bigl((-b,-a)\cup(a,b)\bigr)$, by \eqref{eq:behavior_of_F^+F^-(gamma)_near_gamma=0} it takes the value 
$-\sqrt{\dfrac{\min\{q_+,q_-\}}
{\max\{q_+,q_-\}}}$ at $\gamma=0$ 
and tends to 0 as $\gamma\to \pm\infty$.

By \eqref{eq:|F(gamma)|=constant_in_the_spectrum},
\begin{equation}\label{eq:|B(gamma)|=1_in_the_spectrum}
|\qFF|=1\qquad\text{if $\gamma\in[-b,-a]\cup[a,b]$},
\end{equation}
hence
$|\tildeqFF|=1$ by the second equation
 \eqref{eq:F+(gamma)F-(gamma)}. 

Rewriting  $\calW=\bigl(\newGamma- (q_+-q_-) \bigr)^2-4q_-\newGamma$ we see that on a strictly semi-homogeneous tree $\calW= (q_+ - q_-)^2 +O(\gamma^2)>0$ for $\gamma\sim 0$, then $\theroot=-|q_+ - q_-|+O(\gamma^2)$ by Proposition \ref{prop:change_of_sign_in_analytic_continuation_of_a_square_root} and
\begin{equation}\label{eq:the_numerator_of_B(gamma)}
\qFF=\dfrac{\newGamma-(q_+ + q_-)-\theroot}{2\qmean}
=\dfrac{-(q_+ +q_-) +|q_+ -q_-|}{2\qmean} +O(\gamma^2)=-\dfrac{\min\{q_+,q_-\}}{\qmean}+O(\gamma^2),
\end{equation}
in agreement with
\eqref{eq:F^+(0)_diverges},  \eqref{eq:F^-_vanishes_linearly_as_gamma->0} and \eqref{eq:F^+(0)_and_F^-(0)_for_q+<q-}; see also \eqref{eq:behavior_of_F^+F^-(gamma)_near_gamma=0}.

\begin{corollary}\label{cor:inverse_of_homogeneous_eigenvalue_map}
In the homogeneous setting, the inverse function of the eigenvalue map $\gamma$ in \eqref{eq:the_homogeneous_eigenvalue_map} is the multiple-valued function
 \begin{equation}\label{eq:the_homoheneous_inverses-of_the_eigenvalue_map}
z=\frac1{\ln q}\arccosh\biggl(\frac{q+1}{2\sqrt{q}}\;\gamma\biggr) + \frac12
=\log_q \frac{(q+1)\gamma + \sqrt{ (q+1)^2 \gamma^2- 4q}}2,
\end{equation}
where both determinations of the complex square root are allowed.
If one of these two determinations yields a value $z$, then, by writing $z=\frac12 + w$, the other yields the value $1-z=\frac12-w$, since the $\arccosh$ function is even; then the other values of the inverse functions are, respectively, $z+2k\pi i/\ln q$ and $1-z+2k\pi i/\ln q$ for $ k\in\mathZ$.
Between the two determinations in \eqref{eq:the_homoheneous_inverses-of_the_eigenvalue_map},
of the homogeneous  eigenvalue map $\gamma(z)=\dfrac{q^z+q^{1-z}}{q+1}$ of \eqref{eq:the_homogeneous_eigenvalue_map}, the choice that is compatible with \eqref{eq:n.n._homogeneous_visit_probabilities_at_eigenvalue_gamma} and Proposition \ref{prop:semihomogeneous_K^z_is_not_eigenfunction} is 
\begin{equation}\label{eq:the_choice_of_inverse_homogeneous_eigenvalue_map_compatible_with_transience}
z=\log_q \frac{ (q+1)\gamma + \sqrt{ (q+1)^2 \gamma^2- 4q}}2\qquad \text{ for } \gamma\in\mathR, \; |\gamma|\geqslant \frac {2\sqrt q}{q+1}\,,
\end{equation}
where the logarithm is the real-valued logarithm, and the square root is the positive square root. Then, for every $\gamma\in\mathC$ outside the interval $\{|\gamma|\leqslant 2\sqrt q/(q+1)\}=\spectrum_{\ell^2(V)}(\mu_1)$, the compatible choice is obtained by analytic continuation.
\end{corollary}
\begin{proof}
Formula \eqref{eq:the_homoheneous_inverses-of_the_eigenvalue_map} follows immediately from \eqref{eq:the_homogeneous_eigenvalue_map}.

Since $\gamma=1$ is real and outside the cut $[-\rho(\mu_1),\,\rho(\mu_1)]$, \eqref{eq:n.n._homogeneous_visit_probabilities_at_eigenvalue_gamma} gives the correct determination of
$F(1)= \dfrac{(q+1)-\sqrt{(q-1)^2}}{2q}
=\dfrac1q\,,$ as expected. (The other possible value of the analytic continuation, with the opposite sign for the square root,
gives $F(1)=1$,
not acceptable because the random walk induced by $\mu_1$ is transient; this is in agreement with \eqref{eq:F_for_homogeneous_trees} and \cite{Sava&Woess}*{Lemma 3.2}).
By Remark \ref{rem:semihomogeneous_Poisson_kernel_on_horospheres},
the homogeneous Poisson kernel, in accord with  Theorem \ref{theo:Poisson_kernel_for_vertices,semihomogeneous}, is $q^{h(v,v_0,\omega)}$  (divergent as $v\to\omega$). By \eqref{eq:generalized_Poisson_kernel}, Corollary \ref{cor:0_is_in_the_l^2_spectrum}, \eqref{eq:non_homogeneous_multiplicativity} and \eqref{eq:n.n._homogeneous_visit_probabilities_at_eigenvalue_gamma}, the homogeneous generalized Poisson kernel at any complex eigenvalue $\gamma$ is $\Poiss(v,v_0,\omega\barra \gamma) = F(\gamma)^{-h(v,v_0,\omega)} =
\biggl(\dfrac{ (q+1)\gamma - \sqrt{ (q+1)^2 \gamma^2 - 4q}}{2q}\biggr)^{-h(v,v_0,\omega)}$. By Proposition \ref{prop:semihomogeneous_K^z_is_not_eigenfunction}, this kernel can be written as $q^{zh(v,v_0,\,\omega)}$ for some complex power $z=z(\gamma)$. Therefore this power $z$ is related to the eigenvalue $\gamma$ by the rule
\begin{align*}
z &= -\log_{q} \left( \frac{ (q+1)\gamma - \sqrt{ (q+1)^2 \gamma^2- 4q}}{2q} \right) 
=
-\log_{q} \left( \frac{ (q+1)\gamma - \sqrt{ (q+1)^2 \gamma^2- 4q}}{2\sqrt q} \right)+\frac12
\\[.2cm]
&=
\log_{q} \left( \frac{ (q+1)\gamma + \sqrt{ (q+1)^2 \gamma^2- 4q}}{2\sqrt q} \right) + \frac12 
=
\log_{q} \left( \frac{ (q+1)\gamma + \sqrt{ (q+1)^2 \gamma^2- 4q}}{2} \right). 
\end{align*}
Consider the other determination of the complex square root in \eqref{eq:n.n._homogeneous_visit_probabilities_at_eigenvalue_gamma}, that yields the other variant of the generalized hitting probability, which we write
$\widetilde F(\gamma) = \dfrac{(q +1)\gamma\
 + \sqrt{(q+1)^2\gamma^2-4q}} {2 q }$
 (see also \cite{Sava&Woess}). Using the symmetry property of the inverse eigenvalue map, we have shown that $\widetilde F$
is associated with the value $1-z(\gamma)$ of the inverse eigenvalue map. 

Here is a direct proof.
 For the neighbor $v_1\sim v_0$ which belongs to the ray $[v_0,v_1,\cdots,\omega)$, by \eqref{eq:Poisson_kernel_as_quotient_of_first_visit_probabilities} and \eqref{eq:F_for_homogeneous_trees} $\Poiss(v,v_0,\omega)=F(v_0,v)=1/q$. Since $\Poiss^{z}$ is an eigenfunction of $\mu_1$ with eigenvalue $\gamma(z)$, by \eqref{eq:generalized_Poisson_kernel} we have
  $\Poiss(v,v_0,\omega)^{z}=\Poiss(v,v_0,\omega\barra\gamma)=F(\gamma)$, and
 for these $v$ and $\omega$ 
 \eqref{eq:n.n._homogeneous_visit_probabilities_at_eigenvalue_gamma} yields
\[
\begin{split}
\Poiss(v,v_0,\omega)^{1-z} &= \dfrac1q\; \dfrac 1{ F(\gamma)}=\frac1q\; \frac{2q}{ (q+1)\gamma - \sqrt{ (q+1)^2 \gamma^2- 4q}} \\[.2cm]
&= \frac{ (q+1)\gamma + \sqrt{ (q+1)^2 \gamma^2- 4q}}{2q}
=\widetilde F(\gamma).
\end{split}
\]
\end{proof}

\subsection{The semi-homogeneous  Poisson representation of eigenfunctions}
For the sake of completeness, here is the explicit expression of $\Poiss(v,v_0,\omega\barra\gamma)$ as a function of  $h(v,v_0,\omega)$.

By \eqref{eq:non_homogeneous_multiplicativity}, 
$
\Poiss(v,v_0,\omega\barra\gamma)=\dfrac{F(v,j\barra\gamma)}{F(v_0,j\barra\gamma)}=\dfrac{F(v,w\barra\gamma)}{F(v_0,w\barra\gamma)}$ for every vertex $j$ in $[v,\omega)\cap[v_0,\omega)$.
We can assume $|j|$ even, whence
\[
\Poiss(v,v_0,\omega\barra\gamma)=
\begin{cases}
\dfrac{(F^+(\gamma)F^-(\gamma))^{\dist(v,j)/2}}{(F^+(\gamma)F^-(\gamma))^{|j|/2}}&\text{if  $\dist(v,j)$ is even},
\\[.5cm]
\dfrac{(F^+(\gamma)F^-(\gamma))^{ (\dist(v,j)-1)/2}\,F^-(\gamma)}{(F^+(\gamma)F^-(\gamma))^{|j|/2}}&\text{if  $\dist(v,j)$ is odd}.
\end{cases}
\]
In analogy with \eqref{eq:mu;toplicativity_along_paths_in_terms_of_F+,F-},
we 
construct a function $\widetilde F$ multiplicative along paths by making use of $\widetilde F^+$ and $\widetilde F^-$ as follows. 
Let 
$u_0\sim u_1\sim u_2\dots\sim u_n$ be consecutive vertices,
and set 
\begin{equation}\label{eq:definition_of_Ftilde}
\widetilde F(u_0,u_n\barra \gamma) =  \prod_{j=1}^n  \widetilde F(u_{j-1},u_j\barra \gamma).
\end{equation}
Then, by \eqref{eq:mu;toplicativity_along_paths_in_terms_of_F+,F-} and \eqref{eq:F_per_Ftilde},
 \begin{equation}
\label{eq:expression_of_Ftilde(v,u)_for_every_v,u}
\begin{split}
 \widetilde F(u_0,u_n\barra \gamma) 
 &=
 \widetilde F^\pm(\gamma)^{\lceil \frac n2\rceil }
 \widetilde F^\mp(\gamma)^{\lfloor\frac n2\rfloor}
 =
 \biggl(\dfrac{q_\mp +1 } 
{q_\mp(q_\pm+1)} \; \dfrac 1{F^\pm(\gamma) }     
\biggr)^{\lceil \frac n2\rceil }
\biggl(  \dfrac{q_\pm +1 }{q_\pm(q_\mp+1)} \;
\dfrac 1{ F^\mp(\gamma)}\biggr)^{\lfloor\frac n2\rfloor}
\\[.4cm]
&=\begin{cases}
\dfrac 1{\qmean^{n} F(u_0,u_n\barra \gamma)} &\text {if $n$ is even},
\\[.4cm]
 \dfrac 1{\qmean^{ n} F(u_0,u_n\barra \gamma)}\; \dfrac{\sqrt {q_\pm} (q_\mp +1)}  {\sqrt {q_\mp} (q_\pm +1)} & \text {if $n$ is odd and  $u_0\in V_\pm$)}.
 \end{cases}
 \end{split}
 \end{equation}

By \eqref{eq:generalized_Poisson_kernel} and the fact that $h(v,v_0,\omega)$ is even if and only if $|v|$ is even, one has

\begin{corollary}
\label{coro:generalized_semi-homogeneous_Poisson_kernel}
With notation as in \eqref{eq:n.n._semihomogeneous_visit_probabilities_at_eigenvalue_gamma}  with $q_{v_0}=q_+$ for $\gamma\notin\spectrum_{\ell^2(V)}(\mu_1)$, the generalized Poisson kernel is given by

\[
\begin{split}
\Poiss(v,v_0,\omega\barra \gamma)&=  F^+(\gamma)^{-\ceiling{\frac12 h(v,v_0,\omega)}} F^-(\gamma)^{-\floor{\frac12 h(v,v_0,\omega)}}
 \\
&=\begin{cases}
   \bigl(F^+(\gamma)F^-(\gamma)\bigr)^{-\frac12 h(v,v_0,\omega)} &\text{if 
   $|v|$
   is even,}
    \\[.2cm]
     \dfrac{(F^+(\gamma)F^-(\gamma))^{-\frac12 (h(v,v_0,\omega)-1)}}{F^+(\gamma)}
    =       \bigl(F^+(\gamma)F^-(\gamma)\bigr)^{-\frac12 (h(v,v_0,\omega)+1)}  F^-(\gamma)
     &\text{if
    $|v|$
    is odd.}
\end{cases}
\end{split}
\]
By \eqref{eq:behavior_of_F^+F^-(gamma)_near_gamma=0}, \eqref{eq:F^+(0)_diverges}, \eqref{eq:F^-_vanishes_linearly_as_gamma->0} and \eqref{eq:F^+(0)_and_F^-(0)_for_q+<q-}, this implies that 
\[
\lim_{\gamma\to 0}\Poiss(v,v_0,\omega\barra \gamma) =
\begin{cases}
\biggl( -\dfrac1{\min\{q_+,q_-\}}\biggr)^{-\frac12 h(v,v_0,\omega)}
&\qquad \text{for every $v\in V_+$},
\\[.2cm]
0 &\qquad \text{for every $v\in V_-$, if $q_+<q_-$},
\\[.2cm]
\infty &\qquad \text{for every $v\in V_-$, if $q_+>q_-$}.
\end{cases}
\]
By Corollary \ref{cor:G_and_F_non-zero}, this confirms Corollary \ref{cor:0_is_in_the_l^2_spectrum}.

We also introduce,
for $\gamma\notin\spectrum_{\ell^2(V)}(\mu_1)$, the \emph{alternative generalized Poisson kernel} defined similarly to  \eqref{eq:Poisson_kernel_as_quotient_of_first_visit_probabilities-bis}
but now expressed in terms of $\widetilde F^\pm$, via  
\eqref{eq:expression_of_Ftilde(v,u)_for_every_v,u}. Namely, if $j$ is any vertex in $[v,\omega)\cap[v_0,\omega)$, we set
\begin{equation}\label{eq:generalized_Poisson_kernel,non-extremal}
\widetilde\Poiss(v,v_0,\omega\barra \gamma) = \dfrac{\widetilde F(v,j\barra \gamma)}{\widetilde F(v_0,j\barra \gamma)}
=
\begin{cases}
\qmean^{h(v,v_0,\omega)} 
\dfrac{F(v_0,j\barra \gamma)}{F(v,j\barra \gamma)} = \qmean^{h(v,v_0,\omega)}  \Poiss(v,v_0,\omega\barra\gamma)&\text{ if $|v|$ is even},
\\[.5cm]
\begin{aligned}
&\qmean^{h(v,v_0,\omega)} 
\dfrac{F(v_0,j\barra \gamma)}{F(v,j\barra \gamma)}\;
\dfrac{\sqrt {q_+} (q_- +1)}  {\sqrt {q_-} (q_+ +1)}\\[.2cm]
&\qquad= \qmean^{h(v,v_0,\omega)}  \Poiss(v,v_0,\omega\barra\gamma)\, \dfrac{\sqrt {q_+} (q_- +1)}  {\sqrt {q_-} (q_+ +1)}
\end{aligned}
&\text{ if $|v|$ is odd};
\end{cases}
\end{equation}
By Theorem \ref{theo:Poisson_kernel_for_vertices,semihomogeneous}, in the case $\gamma=1$ this yields 
\begin{equation*}
\widetilde\Poiss(v,v_0,\omega\barra 1)\equiv 1.
\end{equation*}
\end{corollary}

Therefore the generalized Poisson kernel differs from an exponential function of the horospherical index  by the multiplication by a function that depends only on the parity; see Remark  \ref{rem:semihomogeneous_Poisson_kernel_on_horospheres} for the special case $\gamma=1$ of the (harmonic) Poisson kernel.

It is known \cite{Picardello&Woess-PotentialAnalysis}*{Theorem 3.7}  that the eigenfunctions of $\mu_1$ are the Poisson integrals $\int_\Omega \Poiss(v,v_0,\omega\barra \gamma)\, d\mu(\omega)$ of distributions $\mu$ on the boundary, for eigenvalues $\gamma$ outside the $\ell^2$-spectrum (see also \cite{Woess}).  So, in particular, $\Poiss(v,v_0,\omega\barra \gamma)$ is an eigenfunction with eigenvalue $\gamma$. For the sake of completeness, however, we verify this fact directly, and actually we extend this property to all eigenvalues $\gamma\neq 0$:

\begin{corollary} \label{cor:semihomogeneous_generalized_Poisson_kernel_is_eigenfunction}
If $\mu$ is a distribution on $\Omega$,  the integral $\int_\Omega \Poiss(\missarg,v_0,\omega\barra \gamma)\,d\mu(\omega)$
is a  $\gamma$-eigenfunction of $\mu_1$, and if $\mu=\nu_{v_0}$ and $\gamma\notin\spectrum_{\ell^2(V)}(\mu_1)$ one obtains in this way the spherical function: 
 \[
 \phi(\missarg,v_0,\omega\barra\gamma)=\int_\Omega \widetilde\Poiss(\missarg,v_0,\omega\barra \gamma)\, d\nu_{v_0}.
 \]
 
Also
$\widetilde\Poiss(\missarg,v_0,\omega\barra \gamma)$ is a  $\gamma$-eigenfunction, and  the same properties hold. However, 
$\Poiss(\missarg,v_0,\omega\barra \gamma)$ is a minimal $\gamma$-eigenfunction whereas $\widetilde\Poiss(\missarg,v_0,\omega\barra \gamma)$ is not \citelist{\cite{Picardello&Sjogren-Canberra}*{Proposition 2} \cite{Picardello&Woess-London}*{Theorem 4.14}}.
The kernels
$\Poiss(v,v_0,\omega\barra \gamma)$,
and $\widetilde\Poiss(v,v_0,\omega\barra \gamma)$ are the only reproducing kernels for $\gamma$-eigenfunctions of $\mu_1$, see  \cite{Picardello&Woess-London}; for the homogeneous case, we obtain $\widetilde\Poiss^{z(\gamma)}=\Poiss^{1-z(\gamma)}$ (see Corollary \ref{cor:inverse_of_homogeneous_eigenvalue_map} and its proof). 
 
Moreover,  for every $\gamma\notin\spectrum_{\ell^2(V)}(\mu_1)$,we also have 
\[\phi(\missarg,v_0,\omega\barra\gamma)=\int_\Omega \widetilde\Poiss(\missarg,v_0,\omega\barra \gamma)\, d\nu_{v_0}
\]
\end{corollary}

\begin{proof}
It was shown in   \eqref{eq:the_generalized_Poisson_kernel_is_an_eigenfunction}
that, on general trees,  for every complex $\gamma\notin\spectrum_{\ell^2(V)}(\mu_1)$, the function $\Poiss(\missarg,v_0,\omega\barra \gamma)$ is  a $\gamma$-eigenfunction of the nearest-neighbor transition operator $P$ (that is $\mu_1$ in our setting). Since the values of 
$ \Poiss(v,v_0,\omega\barra \gamma)$
depend only on the relative positions of $v$, $v_0$ and $\omega$, the integral $\int_\Omega \Poiss(\missarg,v_0,\omega\barra \gamma)\,d\nu_{v_0}(\omega)$ is a radial function around $v_0$.
The first sentence of the statement follows.
In view of the definition of $\widetilde F$ given in
\eqref{eq:definition_of_Ftilde}, by exactly
 the same argument   (that is,  the multiplicativity rule), or equivalently by  \eqref{eq:the_generalized_Poisson_kernel_is_an_eigenfunction}, 
  also
$\widetilde\Poiss(\missarg,v_0,\omega\barra \gamma)$ is a  $\gamma$-eigenfunction.
 The spherical function
$\phi(\missarg,v_0,\omega\barra\gamma) = \int_\Omega \Poiss(\missarg,v_0,\omega\barra \gamma)\, d\nu_{v_0} $
and the function
$\int_\Omega \widetilde\Poiss(\missarg,v_0,\omega\barra \gamma)\, d\nu_{v_0}$ are both radial $\gamma$-eigenfunctions of $\mu_1$, hence, by Remark
\ref{rem:recurrence_relation_of_the_semi-homogeneous_Laplacian}, they coincide up to multiples. 
By \eqref{eq:generalized_Poisson_kernel,non-extremal}, 
$\int_\Omega \widetilde\Poiss(v_0,v_0,\omega\barra \gamma)\, d\nu_{v_0}=1$, hence the two functions coincide, that is,
 $\phi(\missarg,v_0,\omega\barra\gamma)=\int_\Omega \Poiss(\missarg,v_0,\omega\barra \gamma)\, d\nu_{v_0}$.
\end{proof}

Not every spherical function is of the type $\int_{\Omega} \Poiss(\missarg,v_0,\omega\barra \gamma)\,d\nu_{v_0}(\omega)$: indeed, the $0$-eigenfunction $\psiphi(\missarg,v_0\barra 0)$ introduced in \eqref{eq:the_semi-homogeneous_spherical_function_of_eigenvalue_zero} is not of this type because $0\notin\spectrum_{\ell^2(V)}(\mu_1)$
by Corollary \ref{cor:0_is_in_the_l^2_spectrum}.

\section{$L^p$-behavior of spherical functions}\label{Sec:Spherical_functions}

The next statement extends to the semi-homogeneous setting the well-known formulas for the spherical functions on homogeneous trees \cites{Figa-Talamanca&Picardello-JFA, Figa-Talamanca&Picardello}; in particular,  \eqref{eq:semi-homogeneous_spherical_function_as_Poisson_transform-2}, in terms of $\qFF$ and  
$\tildeqFF$ extends  \cite{Sava&Woess}*{formula (9)}.

\begin{theorem}\label{theo:computation_of_semi-homogeneous_spherical_functions_via_Poisson_kernel}
For $\gamma\notin\spectrum_{\ell^2(V)}(\mu_1)$, every 
spherical function $\phi (v,v_0\barra \gamma)$ on the semi-homogeneous tree with valences $q_+$, $q_-$ ($q_+$ being chosen as the valence of a reference vertex $v_0$) can be expressed as follows. 

If  $\qFF\neq \pm 1$, then, in terms of the parity $\epsilon(v)=(-1)^{|v|}$,
\begin{equation}\label{eq:semi-homogeneous_spherical_function_as_Poisson_transform-2}
\begin{split}
\phi (v,v_0\barra \gamma) 
&=\kappa\bigl(\gamma,\epsilon(v)\bigr)\,\biggl(\frac{\qFF}{\qmean}\biggr)^{\lfloor\frac{|v|}2 \rfloor}
+\widetilde\kappa\bigl(\gamma,\epsilon(v)\bigr)\, \biggl(\frac{\tildeqFF}{\qmean}\biggr)^{  \lfloor\frac { |v|}2 \rfloor}\qquad\text{with}\\[.2cm]
\kappa(\gamma,\epsilon)&=
\begin{cases}
\dfrac{1}{(q_+ +1)\bigl(\qFF^2 -1\bigr)}\;
\biggl(\qFF -       \dfrac 1   \qmean \biggr)
\biggl(\qFF + \sqrt{\dfrac{q_-}{q_+} }\biggr)q_+
&\text{if $\epsilon=+1$,}
\\[.3cm]
\dfrac{1}{(q_+ +1)\bigl(\qFF^2 -1\bigr)}\;
\biggl(\qFF -       \dfrac 1   \qmean \biggr)
\biggl(\qFF + \sqrt{\dfrac{q_-}{q_+} }\biggr)q_+ F^-(\gamma)
&\text{if $\epsilon=-1$,}
\end{cases}
\\[.4cm]
\widetilde\kappa(\gamma,\epsilon)&=
\begin{cases}
\dfrac{1}{(q_+ +1)\bigl(\tildeqFF^2 -1\bigr)}\;
\biggl(\tildeqFF -       \dfrac 1   \qmean \biggr)
\biggl(\tildeqFF + \sqrt{\dfrac{q_-}{q_+} }\biggr)q_+
&\text{if $\epsilon=+1$,}
\\[.3cm]
\dfrac{1}{(q_+ +1)\bigl(\tildeqFF^2 -1\bigr)}\;
\biggl(\tildeqFF -       \dfrac 1   \qmean \biggr)
\biggl(\tildeqFF + \sqrt{\dfrac{q_-}{q_+} }\biggr)q_+\widetilde F^-(\gamma)
&\text{if $\epsilon=-1$.}
\end{cases}
\end{split}
\end{equation}
So, for either parity of $v$, the spherical function, multiplied by the square root $\qmean^{\frac{|v|}2}$ of the volume growth, is a linear combination of $\qFF^{\lfloor \frac{|v|}2 \rfloor }$ and $
\tildeqFF^{\lfloor \frac n2 \rfloor}$.

Instead, 
with $a$, $b$ as in \eqref{eq:endpoints_of_the_L2-spectrum_of_M1}, 
we have
\begin{equation}\label{eq:B(gamma)=+1}
\qFF=1 \text{ if and only if } \gamma=\pm b,
\end{equation}
and
\begin{equation}\label{eq:B(gamma)=+1:spherical_function}
\phi (v,v_0\barra\pm b)
=\begin{cases}
\qmean^{-\frac {|v|}2}
\Biggl(1+\dfrac{q_+}{q_+ +1}\biggl( 1-\dfrac1{\qmean}\biggr)\biggl( 1+\sqrt{\dfrac{q_-}{q_+}}\biggr) \dfrac {|v|}2
\Biggr) 
&\quad\text{for $|v|$ even,}
\\[.6cm]
\qmean^{-\frac {|v|-1}2}
\Biggl(\dfrac{q_+ +\qmean} {q_+ +1}+\dfrac{q_+}{q_+ +1}\biggl( 1-\dfrac1{\qmean}\biggr)\biggl( 1+\sqrt{\dfrac{q_-}{q_+}}\biggr) \dfrac {{|v|}-1}2
\Biggr)
F^-(\pm b) 
&\quad\text{for $|v|$ odd.}
\end{cases}
 \end{equation}
where
$F^-(\pm b)=\widetilde F^-(\pm b)=\pm\sqrt{\dfrac{q_+ +1}{q_+(q_- +1)}}$ and
$F^+(\pm b)=\widetilde F^+(\pm b)=\pm\sqrt{\dfrac{q_- +1}{q_-(q_+ +1)}}$ by \eqref{eq:n.n._semihomogeneous_visit_probabilities_at_eigenvalue_gamma} and \eqref{eq:n.n._semihomogeneous_visit_probabilities_at_eigenvalue_gamma-not_acceptable}, since $\theroot=0$ at $\gamma=\pm b$; this is in agreement with \eqref{eq:F_per_Ftilde}.

Finally,
\begin{equation}\label{eq:B(gamma)=-1}
\qFF=-1 \text{ if and only if } \gamma=\pm a,
\end{equation}
\begin{equation}\label{eq:B(gamma)=-1:spherical_function}
\phi (v,v_0\barra\pm a)
=\begin{cases}
(-\qmean)^{-\frac {|v|}2}
\Biggl(1+\dfrac{q_+}{q_+ +1}\biggl( -1-\dfrac1{\qmean}\biggr)\biggl( -1+\sqrt{\dfrac{q_-}{q_+}}\biggr) \dfrac {|v|}2
\Biggr) 
&\quad\text{for $|v|$ even,}\\[.8cm]
\begin{aligned}
&(-\qmean)^{-\frac {|v|-1}2}
\Biggl(\dfrac{q_+ -\qmean} {q_+ +1}
\\[.1cm]
&\qquad+\dfrac{q_+}{q_+ +1}\biggl(- 1-\dfrac1{\qmean}\biggr)\biggl( -1+\sqrt{\dfrac{q_-}{q_+}}\biggr) \dfrac {{|v|}-1}2
\Biggr)
F^-(\pm a)
\end{aligned}
&\quad\text{for  $|v|$ odd,}
\end{cases}
\end{equation}
where
\begin{align*}
F^-(\pm a)&=\widetilde F^-(\pm a)=\mp\sqrt{\dfrac{q_+ +1}{q_+(q_- +1)}}\sign(q_+ -q_-),
\\[.2cm]
F^+(\pm a)&=\widetilde F^+(\pm a)=\mp\sqrt{\dfrac{q_- +1}{q_-(q_+ +1)}}\sign(q_+ -q_-).
\end{align*}

We shall see in Theorem \ref{theo:the_spectrum_of_M1} that, in the semi-homogeneous case, $[-b,-a]\cup[a,b]$ is the $\ell^2$-spectrum of $\mu_1$. Note that, if $q_+=q_-$, i.e., in the homogeneous setting, then $b=\dfrac{2\sqrt{q} }{q +1}$  is the $\ell^2$-spectral radius of $\mu_1$ \cite{Kesten}; see also \cite{Figa-Talamanca&Picardello}*{Chapter~2, Corollary 3.13}; instead, $a=0$ and the interval $[-b,-a]\cup[a,b]$ becomes $\left[-\dfrac{2\sqrt{q} }{q +1}, \dfrac{2\sqrt{q} }{q +1}\right]$, the $\ell^2$-spectrum of the homogeneous Laplacian.
\end{theorem}

\begin{proof}
First of all, it follows easily from \eqref{eq:F+(gamma)F-(gamma)}
that
\begin{equation}\label{eq:B(gamma)=+-1_in_terms_of_gamma^2}
\qFF=\pm 1 \quad\text{ if and only if \quad $\gamma^2= \frac{  (q_+ + q_-)\pm 2\sqrt{q_+q_-}} {(q_+ +1)(q_- +1)}$}, 
\end{equation}
that yields \eqref{eq:B(gamma)=+1} and \eqref{eq:B(gamma)=-1}.

The fact 
that horospherical index depends only on the relative positions of $v, v_0, \omega$ implies that the $\nu_{v_0}$-integral of $\omega\mapsto \Poiss(v,v_0,\omega)$ is the same as the  sum of  $v\mapsto \Poiss(v,v_0,\omega)$ for $|v|$  constant. Thus, for each $\gamma\notin\spectrum_{\ell^2(V)}(\mu_1)$, $\gamma\neq 0$, we obtain a spherical function $\phi (\missarg,v_0\barra \gamma)$  
by the rule
\[
\phi (v,v_0\barra \gamma) = \int_\Omega \Poiss(v,v_0,\omega\barra \gamma)\,d\nu_{v_0}(\omega)
\]
(see also \eqref{eq:semihomogeneous_invariant_measures}).

With the terminology of Lemma  \ref{lemma:measures_of_boundary_arcs_in_semihomogenous_trees}, for each $k=0,\dotsc,n$ we denote by  $v_k$  the $k$-th vertex of the finite ray from $v_0$ to $v=v_n$ and observe that $\Poiss (v,v_0,\omega\barra \gamma)$ is constant on each of the sets $\Omega_k(v_0,v_n)$. 
We have $\phi (v_0,v_0\barra \gamma)=1$ because $\Poiss(v_0,v_0,\omega\barra \gamma)=1$ for every $\omega, \gamma$ by \eqref{eq:the_generalized_Poisson_kernel_is normalized_at_v_0}; hence we may assume $v=v_n\neq v_0$, that is, $n>0$. Then
\begin{align}\label{eq:the_semihomogeneous_vertex_spherical_function_computed_from_the_Poisson_integral}
\phi (v,v_0\barra \gamma) = \sum_{k=0}^{n} \Poiss (v,v_0,\omega\barra \gamma)|_{\Omega_k(v_0,v)} \,\nu_{v_0}(\Omega_k(v_0,v)) 
\end{align}
Since $h(v,v_0,\omega) = 2k-n$ when $\omega \in \Omega_k(v_0,v) $ because $v=v_n$, it follows from Corollary \ref{coro:generalized_semi-homogeneous_Poisson_kernel} that
\begin{equation*}
\begin{split}
 \Poiss (v,v_0,\omega\barra \gamma)|_{\Omega_k(v_0,v)}
 &=
 \begin{cases}
\biggl(\dfrac\qFF\qmean\biggr)^{\frac{|v|}2 - k} & 
\text{if } |v| \;\text{is even}\\[.4cm]
 \biggl(\dfrac\qFF\qmean\biggr)^{\frac{|v|+1}2 - k} \dfrac1{F^+(\gamma)}
=\biggl(\dfrac\qFF\qmean\biggr)^{\frac{|v|-1}2 - k}         F^-(\gamma)
&
\text{if } |v| \;\text{is odd}. 
\end{cases}
\end{split}
\end{equation*}

By \eqref{eq:measures_of_arcs_subtended_by_subsequent_vertices_in_a_path-0}, identity  \eqref{eq:the_semihomogeneous_vertex_spherical_function_computed_from_the_Poisson_integral}  becomes

\begin{equation}\label{eq:expression_of_phi_with_summations}
\phi (v,v_0\barra \gamma)
=\begin{cases}
\begin{aligned}
\dfrac{\qmean^{-\frac n2}}{q_+ +1}
&\biggl(  q_+                                                                   \qFF^{ \frac n2    }
       +(q_+ -1)                      \sum_{\substack{0<k<n\\\text{$k$ even}}} \qFF^{ \frac n2 - k}\\[.2cm]
&\qquad+(q_- -1)\sqrt{\frac{q_+}{q_-}}\sum_{\substack{0<k<n\\\text{$k$ odd} }} \qFF^{ \frac n2 - k}
       + q_+                                                                   \qFF^{-\frac n2    }\biggr)
\end{aligned}
&\quad\text{for even $n=|v|$,}\\[1.6cm]
\begin{aligned}
&\dfrac{\qmean^{-\frac {n+1}2}}{q_+ +1}\frac{1}{F^+(\gamma)}
\biggl(  q_+                                                                   \qFF^{ \frac {n+1}2    }
       +(q_+ -1)                      \sum_{\substack{0<k<n\\\text{$k$ even}}} \qFF^{ \frac {n+1}2 - k}\\[.2cm]
&\qquad+(q_- -1)\sqrt{\frac{q_+}{q_-}}\sum_{\substack{0<k<n\\\text{$k$ odd} }} \qFF^{ \frac {n+1}2 - k}
       + q_-    \sqrt{\frac{q_+}{q_-}}                                         \qFF^{-\frac {n-1}2    }\biggr)
\end{aligned}
&\quad\text{for odd $n=|v|$.}
\end{cases}
\end{equation}
If $\qFF^2\neq 1$ then
\begin{align*}
\sum_{\substack{0<k<n\\\text{$k$ even}}} \qFF^{-k}
&=\begin{cases}
\displaystyle
\sum_{j=1}^{n/2-1} \qFF^{-2j}
=\frac{1 - \qFF^{2-n}}{\qFF^2 -1}
&\text{for even $n>0$,}
\\[.6cm]
\displaystyle
\sum_{j=1}^{(n-1)/2} \qFF^{-2j}
=\frac{1-\qFF^{1-n}} {\qFF^2 -1}
&\text{for odd $n$,}
\end{cases}
\\[.3cm]
\sum_{\substack{0<k<n\\\text{$k$ odd}}} \qFF^{-k}
&=\begin{cases}
\displaystyle
\sum_{j=1}^{n/2} \qFF^{1-2j}
=\qFF\;\frac{1-\qFF^{-n}} {\qFF^2 -1}
&\text{for even $n>0$,}
\\[.6cm]
\displaystyle
\sum_{j=1}^{(n-1)/2} \qFF^{1-2j}
=\qFF\;\frac{1-\qFF^{1-n}} {\qFF^2-1}
&\text{for odd $n$.}
\end{cases}
\end{align*}
therefore
\begin{equation*}
\phi (v,v_0\barra \gamma) 
=\begin{cases}
\begin{aligned}
&\frac{\qmean^{-\frac n2}}{(q_+ +1)(\qFF^2 -1)}\\[.2cm]
&\qquad\cdot\biggl( \Bigl(q_+ \qFF^2 - 1      + (q_- -1)\sqrt{\frac{q_+}{q_-}}\qFF\Bigr) \qFF^{ \frac n2}\\[.2cm]
&\qquad\qquad -\Bigl(q_+        - \qFF^2 + (q_- -1)\sqrt{\frac{q_+}{q_-}}\qFF\Bigr) \qFF^{-\frac n2}\biggr)
\end{aligned}
&\quad\text{for  $|v|=n$ even,}\\[2cm]
\begin{aligned}
&\frac{\qmean^{-\frac{n+1}2}}{(q_+ +1)(\qFF^2 -1)}\frac{1}{F^+(\gamma)}\\[.2cm]
&\qquad\cdot\biggl( \Bigl( q_+     \qFF^2 - 1 + (q_- - 1     )\sqrt{\frac{q_+}{q_-}}\qFF   \Bigr) \qFF^{ \frac {n+1}2}\\[.2cm]
&\qquad\qquad -\Bigl((q_+ -1) \qFF       + (q_- - \qFF^2)\sqrt{\frac{q_+}{q_-}}       \Bigr) \qFF^{-\frac {n-1}2}\biggr)\\
&\quad=\frac{
\qmean^{-\frac{n-1}2}}{(q_+ +1)(\qFF^2 -1)}F^-(\gamma)\\[.2cm]
&\quad\qquad\cdot\biggl( \Bigl( q_+     \qFF^2 - 1 + (q_- - 1     )\sqrt{\frac{q_+}{q_-}}\qFF   \Bigr) \qFF^{ \frac {n-1}2}\\[.2cm]
&\quad\qquad\qquad -\Bigl((q_+ -1) \qFF       + (q_- - \qFF^2)\sqrt{\frac{q_+}{q_-}}       \Bigr) \qFF^{-\frac {n+1}2}\biggr)
\end{aligned}
&\quad\text{for $|v|=n$ odd.}
\end{cases}
\end{equation*}%
where the last equality holds because, for odd $n$, by \eqref{eq:definition_of_qFF}
\begin{align*}
  \qmean^{-\frac{n+1}{2}}\qFF^{ \frac{n+1}{2}}\dfrac 1{F_+(\gamma)}
&=\qmean^{-\frac{n-1}{2}}\qFF^{ \frac{n-1}{2}}         F_-(\gamma), \\
  \qmean^{-\frac{n+1}{2}}\qFF^{-\frac{n-1}{2}}\dfrac 1{F_+(\gamma)}
&=\qmean^{-\frac{n-1}{2}}\qFF^{-\frac{n+1}{2}}         F_-(\gamma).
\end{align*}
By writing $\tildeqFF^{-1}$ as in \eqref{eq:F+(gamma)F-(gamma)}, we
express the factors of the exponential terms as in the statement.

Instead, if $\qFF=1$, we obtain
\begin{align*}
\sum_{\substack{0<k<n\\ \text{$k$ even}}}   \qFF^{-k} &= 
\begin{cases}
\displaystyle
\sum_{j=1}^{n/2-1}  1 = \frac n2 -1 &\text{if $n$ is even},
\\[.6cm]
\displaystyle
\sum_{j=1}^{(n-1)/2}  1 = \frac{n-1}2 &\text{if $n$ is odd},
\end{cases} \\[.3cm]
\sum_{\substack{0<k<n\\ \text{$k$ odd}}}   \qFF^{-k} &=
\begin{cases}
\displaystyle
 \sum_{j=1}^{n/2}  1 = \frac n2  &\text{if $n$ is even},
\\[.6cm]
\displaystyle
 \sum_{j=1}^{(n-1)/2}  1 =  \frac{n-1}2 &\text{if $n$ is odd},
\end{cases}
\end{align*}
whereas, if $\qFF=-1$, the first identity holds but the right-hand sides of the second have the opposite sign. So, if $\qFF =\pm 1$, \eqref{eq:expression_of_phi_with_summations} becomes
\begin{equation}\label{eq:semihomogeneous_spherical_function_as_Poisson_transform-2-degenerate_case_qF=1}
\phi (v,v_0\barra \gamma)
=\begin{cases}
\dfrac{\qmean^{-\frac n2}}{q_+ +1}
\biggl(2q_+ + (q_+ -1)\biggl(\dfrac n2 -1\biggr)
\pm (q_- -1)\sqrt{\dfrac{q_+}{q_-}}\,\dfrac n2\biggr) (\pm1)^{\frac n2}
&\quad\text{for even $n$,}\\[.6cm]
\dfrac{\qmean^{-\frac {n+1}2}}{q_+ +1}\dfrac1{F^+(\gamma)}
\biggl(q_+ + (q_+ -1)\dfrac{n-1}2
\pm (q_- -1)\sqrt{\dfrac{q_+}{q_-}}\,\dfrac{n-1}2
\pm  \qmean           \biggr) (\pm1)^{\frac {n+1}2}
&\quad\text{for odd $n$.}
\end{cases}
\end{equation}
Since, by \eqref{eq:definition_of_qFF}, $\dfrac1{F^+(\gamma)}=\pm\qmean F^{-}(\gamma)$ if $\qFF=\pm 1$, the statement now follows from \eqref{eq:B(gamma)=+1:spherical_function}, \eqref{eq:B(gamma)=-1:spherical_function} and algebraic manipulations.
\end{proof}

\begin{remark}
$\qFF$ has been defined in terms of the choice of determination of the complex square root in $\theroot$ made in Definition \ref{def:the_root_R}. Since $\tildeqFF$ is defined by the opposite determination, the linear expansion \eqref{eq:semi-homogeneous_spherical_function_as_Poisson_transform-2} of the spherical functions, being symmetric in $\qFF$ and $\tildeqFF$, does not change if we make the opposite choice.
\end{remark}

\begin{theorem} \label{theo:the_Lp_behavior_of_the_zero-spherical_function}
Without loss of generality, assume as usual $v_0\in V_+$.
For $q_+>q_-$, $\lim_{\gamma\to 0}\phi(v,v_0\barra \gamma)=\psiphi(v,v_0\barra 0)$ for every $v\in V$: that is, the radial zero-eigenfunction of $\mu_1$ introduced in
\eqref{eq:phi(v,v_0,0)_in_l^2_iff_q_+<q_-}
is a limit of spherical functions for $\spectrum_{\ell^2(V)}\not\ni\gamma\to 0$. For $q_+<q_-$, the limit $\lim_{\gamma\to 0}\phi(v,v_0\barra \gamma)$ exists only for $v\in V_+$, hence the radial zero-eigenfunction, that by \eqref{eq:phi(v,v_0,0)_in_l^2_iff_q_+<q_-} belongs to $\ell^2(V)$, 
is not a limit of spherical functions with $\gamma\notin\spectrum_{\ell^2(V)}(\mu_1)$.

Finally, $\phi (\missarg,v_0\barra \gamma) \in \ell^p(V)$ for some $p<2$ if and only if $\gamma= 0$ and $q_+<q_-$, and in general $\phi (\missarg,v_0\barra \gamma)\in \ell^p(V)$ for every $p>\pcrit$.
\end{theorem}

\begin{proof}
By Remark \ref{rem:reversible}, the $\ell^2(V)$-spectrum of $\mu_1$ is real, hence it has no interior points. If $\gamma'\notin\spectrum_{\ell^2(V)}(\mu_1)$ tends to $\gamma\in\spectrum_{\ell^2(V)}(\mu_1)$, then by \eqref{eq:F+(gamma)F-(gamma)}
 $\lim_{\gamma'\to\gamma} \qFF$ exists and is finite.
   Then, by \eqref{eq:F^+(0)_and_F^-(0)_for_q+<q-}, for $q_+>q_-$ 
  the limit for $\gamma\to 0$ of the 
   expression \eqref{eq:semi-homogeneous_spherical_function_as_Poisson_transform-2} of $\phi(\missarg,v_0\barra\gamma)$ exists, and gives rise to a spherical function  even though $0\in\spectrum_{\ell^2(V)}(\mu_1)$, namely $\phi(\missarg,v_0\barra 0)$. On the other hand, for $q_+<q_-$, the limit  exists for $v\in V_+$, but not for $v\in V_-$, because, for $|v|$ odd,
    the coefficients $\kappa(\gamma, |v|)$ and $\widetilde\kappa(\gamma, |v|)$ of Theorem \ref{theo:computation_of_semi-homogeneous_spherical_functions_via_Poisson_kernel}  diverge as $\gamma\to 0$ by \eqref{eq:F^+(0)_diverges} and \eqref{eq:F^-_vanishes_linearly_as_gamma->0}.
Therefore
   identities \eqref{eq:semi-homogeneous_spherical_function_as_Poisson_transform-2} and
   \eqref{eq:semihomogeneous_spherical_function_as_Poisson_transform-2-degenerate_case_qF=1} define a spherical function $\phi(\missarg,v_0\barra 0)$ also at the eigenvalue $0$ for $q_+>q_-$, but not in the case $q_+<q_-$, where the spherical function at $\gamma=0$, that was computed by a different approach in  \eqref{eq:the_semi-homogeneous_spherical_function_of_eigenvalue_zero}, is not the limit for $\gamma\to 0$ of the spherical function $\phi(\missarg,v_0\barra \gamma)$ from $\gamma\notin\spectrum_{\ell^2(V)}(\mu_1)$. It is clear from these expressions that $\phi(\missarg,v_0\barra \gamma)$ does not belong to $\ell^2(V)$ if $|\qFF| =1$; more precisely, $\phi(\missarg,v_0\barra \gamma)$ belongs to $\ell^2(V)$ if and only if
   $q_+<q_-$ and $\gamma=0$, as seen in \eqref{eq:the_semi-homogeneous_spherical_function_of_eigenvalue_zero}.

Let us now determine for which $\gamma,\, q_+,\, q_-$ the spherical function belongs to $\ell^p(V)$ for some $p\leqslant 2$.
The cardinality of the circle of radius $n$ grows as $\qmean^n$. Then, by \eqref{eq:semihomogeneous_spherical_function_as_Poisson_transform-2-degenerate_case_qF=1}, $\phi(\missarg,v_0\barra \gamma)\notin\ell^p(V)$ for any $p\leqslant 2$ if $\qFF=\pm 1$, that is if $\gamma$ is as in \eqref{eq:B(gamma)=+-1_in_terms_of_gamma^2}. For the other values of $\gamma$, $\phi (\missarg,v_0\barra \gamma)\in \ell^2(V)$ if and only if, on the right-hand side of the expression \eqref{eq:semi-homogeneous_spherical_function_as_Poisson_transform-2} of the spherical function,  
the coefficient $\kappa$ or $\widetilde\kappa$ of the largest 
 of the two exponentials $\qFF^{\frac n 2}$ and $\qFF^{-\frac n 2}$ vanishes. We now show  that this does not happens if $\gamma\neq 0$.

Because $\gamma\notin\spectrum_{\ell^2(V)}(\mu_1)$, then $\qFF$ is finite by Corollary \ref{cor:G_and_F_non-zero}. As $\tildeqFF=\qFF^{-1}$, at least one between $\qFF$ and $\tildeqFF$ has modulus not less than 1, and only one of the two can be strictly larger than 1. Then, for $\qFF\neq \pm 1$,
 \eqref
{eq:semi-homogeneous_spherical_function_as_Poisson_transform-2} shows that
$\phi(\missarg,v_0\barra \gamma)\in\ell^2(V)$ if and only if 
\begin{equation}\label{eq:conditions_for_spher.function_not_in_ell^2-first}
\begin{cases}
|\qFF| \geqslant 1,
\\[.2cm]
\biggl(\qFF -       \dfrac 1   \qmean \biggr)
\biggl(\qFF + \sqrt{\dfrac{q_-}{q_+} }\biggr)
= 0,
\end{cases}
\quad\text{or}\quad
\begin{cases}
|\tildeqFF| \geqslant 1,
\\[.2cm]
\biggl(\tildeqFF -       \dfrac 1   \qmean \biggr)
\biggl(\tildeqFF + \sqrt{\dfrac{q_-}{q_+} }\biggr)
= 0.
\end{cases}
\end{equation}
Since $|1/\qmean|<1$, there are solutions with $\qFF\neq\pm 1$ only if  $q_- > q_+$, and in this case there is only solution for each system, namely  $\qFF=-\sqrt{q_-/q_+}$ for the former, and $\qFF=-\sqrt{q_+/q_-}$ for the latter.

It follows immediately from \eqref{eq:B(gamma)=+1} and \eqref{eq:B(gamma)=+1} that $\phi(\missarg,\,v_0\barra\gamma)\notin\ell^2(V)$ if $\qFF=\pm 1$.

In the homogeneous setting, the same argument shows that, for any eigenvalue such that $\qFF\neq\pm1$, the  solutions of the two systems \eqref{eq:conditions_for_spher.function_not_in_ell^2-first} do not provide a spherical function in $\ell^2(V)$. On the other hand,
 \eqref{eq:B(gamma)=+1:spherical_function} and \eqref{eq:B(gamma)=-1:spherical_function} show that, if $\qFF=\pm1$, then $\phi(\missarg,v_0\barra\gamma)\notin\ell^2(V)$.  Therefore
there are no square-integrable spherical functions. This is in agreement with \eqref{eq:phi(v,v_0,0)_in_l^2_iff_q_+<q_-}, since, in the homogeneous setting, $\gamma=0$ if and only if $\qFF=-1$, by \eqref{eq:B(gamma)=+-1_in_terms_of_gamma^2}.

Equating with  \eqref{eq:F+(gamma)F-(gamma)} the only compatible solution of each of the systems \eqref{eq:conditions_for_spher.function_not_in_ell^2-first}, from Definition \ref{def:the_root_R} we respectively obtain
\[
\newGamma -q_+ + q_- = \theroot \qquad\text{ or }\quad \newGamma -q_+ + q_- = -\theroot,
\]
which entail $\newGamma=0$, 
i.e., $\gamma=0$. 
This shows that the spherical function $\phi(\missarg,v_0\barra 0)$ computed in
\eqref{eq:the_semi-homogeneous_spherical_function_of_eigenvalue_zero} is, up to normalization, the only radial $\ell^2$-eigenfunction of $\mu_1$ at the eigenvalue $0$.

      Finally, let us determine for which $p$ a spherical function belongs to $\ell^p(V)$. 
      By \eqref{eq:phi(v,v_0,0)_in_l^2_iff_q_+<q_-}, $\phi(\missarg,v_0\barra 0)\in\ell^p(V)$ if and only if $p>\pcrit$, hence $\phi(\missarg,v_0\barra 0)$ belongs to some $\ell^p$-spaces with $p<2$ if and only if $q_+<q_-$. We have already proved that no other spherical function belongs to $\ell^2(V)$, hence, a fortiori, to any $\ell^p(V)$ with $p<2$.
\end{proof}

\begin{theorem}\label{theo:boundedness_of_F+(gamma)F-(gamma)} Given $p>2$ (whence $p'=\bigl(1-\frac1p\bigr)^{-1}<2$), the set of eigenvalues $\gamma$ such that  the spherical function $\phi (\missarg,v_0\barra \gamma)$ belongs to $\ell^{p}(V)$, with the possible exception of   $\gamma=0$, is 
\begin{equation}\label{eq:inequality_satisfied_by_the_semihomogeneous_lp_spectrum}
\left\{\gamma\in\mathC\colon\; \qmean^{\frac 1{p}-\frac 1{p'}}<|\qFF|<\qmean^{\frac 1{p'}-\frac 1{p}}\right\}.
\end{equation}
Apart from $\gamma=0$, the set of eigenvalues such that $\phi(\missarg,v_0\barra \gamma)$ belongs to $\bigcap_{p> 2} \ell^p(V)$ is
\begin{equation}\label{eq:inequality_satisfied_by_the_semihomogeneous_l2_spectrum}
\left\{\gamma\in\mathC\colon\; |\qFF| = 1\right\}\,.
\end{equation}
As seen in Theorem \ref{theo:the_Lp_behavior_of_the_zero-spherical_function},  there is exactly one spherical function that belongs to $\ell^2(V)$ if and only if $q_+ < q_-$, and more generally to $\ell^p(V)$ for every $p>\pcrit$, namely the function
 $\psiphi(\missarg,v_0\barra 0)$. In other words, 0 belongs to the discrete spectrum of $\mu_1$ on $\ell^2(V)$ for $q_+<q_-$; if $q_+>q_-$, then, by Corollary \ref{cor:0_is_in_the_l^2_spectrum}, \,$0$ belongs to the continuous spectrum of $\mu_1$, but not to the discrete spectrum.

The spherical function with eigenvalue $\gamma$ is bounded if and only if 
\begin{equation}\label{eq:inequality_satisfied_by_the_semihomogeneous_l-infty_spectrum}
\left\{\gamma\in\mathC\colon \;
\frac 1 \qmean \leqslant |\qFF| \leqslant \qmean\right\}.
\end{equation}
\end{theorem}
\begin{proof}
By \eqref {eq:cardinality_of_circles}, the number of vertices of length $n$ grows as $\qmean^n$. Then result follows from the fact that,
by \eqref{eq:semi-homogeneous_spherical_function_as_Poisson_transform-2} and
\eqref{eq:semihomogeneous_spherical_function_as_Poisson_transform-2-degenerate_case_qF=1},
the semi-homogeneous spherical function $\phi $ is bounded if and only if $\qmean^{-1}\leqslant| \qFF| \leqslant \qmean$, and it belongs to $\ell^{p}(V)$, \,$2<p<\infty$, if and only if, for large $n$,
$        \bigl|\frac\qFF\qmean\bigr|^{ \frac{np}2}<\qmean^{-n}$ and
$q^{-np}\bigl|\frac\qFF\qmean\bigr|^{-\frac{np}2}<\qmean^{-n}$, that is, $\qmean^{\frac 1{p}-\frac 1p'}<|\qFF|<\qmean^{\frac 1{p'}-\frac 1{p}}$.
 
Finally, we have seen that   $\psiphi(\missarg,v_0\barra 0)\in\ell^2(V)$ if and only if $q_+>q_-$. 
Of course, outside the  $\ell^2$-spectrum  there cannot be $\ell^2$-eigenfunctions.
\end{proof}

The expression of $\qFF$ in \eqref{eq:F+(gamma)F-(gamma)}  includes a complex square root, hence it is a two-valued complex function, indeterminate up to a sign. We now show that, for what concerns inequalities \eqref{eq:inequality_satisfied_by_the_semihomogeneous_lp_spectrum}, this indeterminacy is irrelevant. Indeed,  let us rewrite \eqref{eq:F+(gamma)F-(gamma)} as
\begin{equation}\label{eq:qmean_Fmean}
\qFF=\delta\pm\sqrt{\delta^2-1}\;
\qquad\text{where}\quad
\delta=\frac{\newGamma-(q_+ + q_-)}{2\qmean}\;.
\end{equation}
Then
\begin{proposition}\label{prop:spectra}
For $1\leqslant p \leqslant 2$, if the two-sided inequality in \eqref{eq:inequality_satisfied_by_the_semihomogeneous_lp_spectrum} holds for one of the two values of $\qFF$ then it also holds for the other, and expressed in the variable $\delta$ of \eqref{eq:qmean_Fmean},
the set in
\eqref{eq:inequality_satisfied_by_the_semihomogeneous_lp_spectrum} 
is the ellipse
$ \{\delta\colon|\delta +1| + |\delta -1| \leqslant \qmean^{1-\frac2{p}} + \qmean^{\frac2{p} -1} = \qmean^{\frac2{p}-1} +\qmean^{\frac2{p'}-1} \}$
 with center 0, foci at $\pm 1$ and major axis of length  $\qmean^{1-\frac2{p}} +\qmean^{1-\frac2{p'}}$ (while it is not an ellipse in the variable $\gamma$ given by the eigenvalue, unless the tree is homogeneous: see Proposition \ref{prop:ellipses_and_their_squares} below).
\end{proposition}

\begin{proof}
With $\delta$ as in \eqref{eq:qmean_Fmean}, \eqref{eq:inequality_satisfied_by_the_semihomogeneous_lp_spectrum} becomes
\[
\qmean^{1-\frac2p}  \leqslant \delta\pm \sqrt{\delta^2 -1} \leqslant \qmean^{1-\frac2{p'}} \;.
\]
An elementary computation shows that
\begin{equation}\label{eq:reciprocity}
\delta\pm \sqrt{\delta^2 -1} = \frac 1{\delta\mp \sqrt{\delta^2 -1}}\;.
\end{equation}
Observe that, for $t>0$, the function $\Psi(t)=t+1/t$ has its minimum at $t=1$ and for every $c>1$ its maxima in the interval $[1/c, \,c]$ are at the extreme points. Thus, for $t>0$ and $c>1$, one has $1/c < t < c$ if and only if $ t+1/t < c+1/c$. Now let $t=
\delta\pm \sqrt{\delta^2 -1}$ and $c=\qmean^{1-\frac2{p'}}$. Then $1/c = \qmean^{1-\frac2{p}}$, and  \eqref{eq:inequality_satisfied_by_the_semihomogeneous_lp_spectrum} becomes 
\begin{equation}\label{eq:semihomogeneous_lp_spectrum_in_the_variable_delta}
|\delta + \sqrt{\delta^2 -1}| + |\delta - \sqrt{\delta^2 -1}| =t+\frac1t \leqslant  c+\frac1c = \qmean^{1-\frac2{p}} + \qmean^{\frac2{p}-1}.
\end{equation}
 Now it suffices to verify that
\begin{equation}\label{eq:the_identity_that_makes_the_spectrum_an_ellipse}
|\delta + \sqrt{\delta^2 -1} | + |\delta - \sqrt{\delta^2 -1} | =  |\delta +1| + |\delta -1|.
\end{equation}
This identity is easily proved by taking squares on both sides. 

Since $\newGamma$ is quadratic in $\gamma$,   so is $\delta$ by \eqref{eq:qmean_Fmean}. Hence the set in
\eqref{eq:inequality_satisfied_by_the_semihomogeneous_lp_spectrum}, that is an ellipse in the variable $\delta$,  is also an ellipse in the variable $\gamma$ if and only if it is centered at the origin. By \eqref{eq:qmean_Fmean}, this happens if and only if $q_+=q_-$.
\end{proof}

\begin{figure}
\tikzmath{
 \ticklength=.02;
 \scaling=3.1;
 \qplus =5;
 \qminus=2;
 \qmeanvalue =sqrt(\qplus*\qminus);
 \pcritvalue=1+ln(\qplus)/ln(\qminus);
 \shrink=2*\qmeanvalue/((\qplus+1)*(\qminus+1));
 \prercenter=(\qplus+\qminus)/(2*\qmeanvalue);
function re(\recp,\th){return{(\prercenter+cosh((1-2*\recp)*ln(\qmeanvalue))*cos(\th))};};
function im(\recp,\th){return{(            sinh((1-2*\recp)*ln(\qmeanvalue))*sin(\th))};};
function mo(\recp,\th){                  return{(sqrt (re(\recp,\th)^2+im(\recp,\th)^2))};};
function ar(\recp,\th){if \th>-pi then {return{(atan2(im(\recp,\th)  ,re(\recp,\th)  ))};}
                                  else {if \recp*\pcritvalue>1 then {return   0;}
                                                                else {return -pi;};};};
}
{ 
\tikzset{
every picture/.style={trig format=rad},
every    node/.style={font=\footnotesize},
every   label/.style={font=\footnotesize},
     spectrum/.pic  ={
 \draw[gray,->](-1.1, . )--( 1.1        , .          );
 \draw[gray,->](  . ,-.6)--(  .         , .6         );
 \draw[gray   ](-1. , . )--(-1          ,-\ticklength) node[black,below  left]{$-1  $};
 \draw[gray   ]( 1. , . )--( 1          ,-\ticklength) node[black,below right]{$+1  $};
 \draw[gray   ](  . ,-.5)--(-\ticklength,-.5         ) node[black,below  left]{$-i/2$};
 \draw[gray   ](  . , .5)--(-\ticklength, .5         ) node[black,above  left]{$+i/2$};
 \draw[line width=.75pt,variable=\th,domain=-pi:pi,samples=100,smooth,fill=gray,fill opacity=.5]
\silenceable{
  plot({ sqrt(\shrink*mo(\recp,\th))*cos(ar(\recp,\th)/2)},
        { sqrt(\shrink*mo(\recp,\th))*sin(ar(\recp,\th)/2)})
  plot({-sqrt(\shrink*mo(\recp,\th))*cos(ar(\recp,\th)/2)},
        {-sqrt(\shrink*mo(\recp,\th))*sin(ar(\recp,\th)/2)})
};
}
}
\shoveleft{
\begin{tikzpicture}[scale=\scaling]
\tikzmath{\recp=1/2;}
\pic[scale=\scaling]{spectrum};
\node at(-.8,.5){\small $p=2$};
\end{tikzpicture}
}
\\\vspace{-.8cm}
\shoveright{
\begin{tikzpicture}[scale=\scaling]
\tikzmath{\recp=2/(2+\pcritvalue);}
\pic[scale=\scaling]{spectrum};
\node at( .6,.5){\small $2<p<1+\dfrac{\ln q_+}{\ln q_-}$};
\end{tikzpicture}
}
\\\vspace{-.8cm}
\shoveleft{
\begin{tikzpicture}[scale=\scaling]
\tikzmath{\recp=1/\pcritvalue;}
\pic[scale=\scaling]{spectrum};
\node at(-.8,.5){\small $p=1+\dfrac{\ln q_+}{\ln q_-}$};
\end{tikzpicture}
}
\\\vspace{-.8cm}
\shoveright{
\begin{tikzpicture}[scale=\scaling]
\tikzmath{\recp=1/(1.2*\pcritvalue);}
\pic[scale=\scaling]{spectrum};
\node at( .6,.5){\small $p>1+\dfrac{\ln q_+}{\ln q_-}$};
\end{tikzpicture}
}
\\\vspace{-.8cm}
\shoveleft{
\begin{tikzpicture}[scale=\scaling]
\tikzmath{\recp=0;}
\pic[scale=\scaling]{spectrum};
\node at(-.8,.5){\small $p=+\infty$};
\end{tikzpicture}
}
}
\caption{The set of complex eigenvalues such that the radial eigenfunctions of the Laplacian $\mu_1$ on the semi-homogeneous tree with $q_+>q_-$ belong to $\ell^{p}(V)$, for several values of $p$ in the range $2\leqslant p \leqslant \infty$ (here $q_+=5$ and $q_-=2$). For $2< p\leqslant \pcrit=  1+ \dfrac {\ln q_+}{\ln   q_-}$ this set of eigenvalues includes the origin. 
\newline
If $q_+<q_-$, each picture is the same except for the fact that $1<\pcrit<2$ and the eigenvalue $0$ gives rise to a non-zero radial eigenfunction in $\ell^{r}$ for every $ r>\pcrit$.
\newline
Each of these sets is  the spectrum of the Laplacian on $\ell^p(V)$ and also the spectrum of the Laplacian on the corresponding dual space $\ell^{p'}(V)$ (Theorem \ref{theo:the_spectrum_of_M1}). So, in the case $q_+<q_-$, these pictures of the $\ell^p$-spectra include also the origin, for every $p\geqslant 1$.
\newline
All the spectra are invariant under conjugation because so is $\mu_1$.}
\label{Fig:semi-homogeneous_bounded_spherical_functions}
\end{figure}
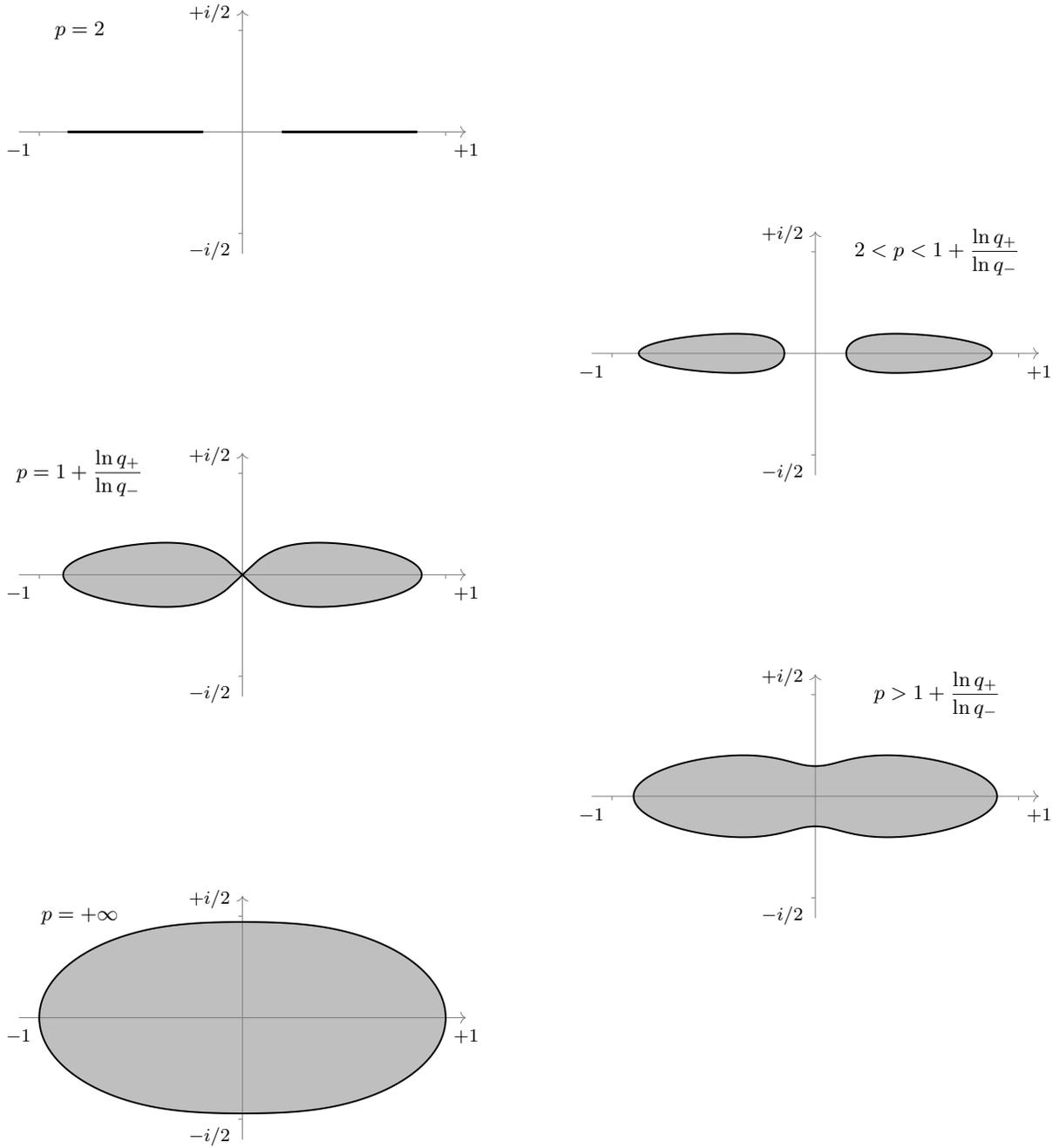


\begin{figure}[h!]
\begin{center}
    \includegraphics[scale=0.9]
{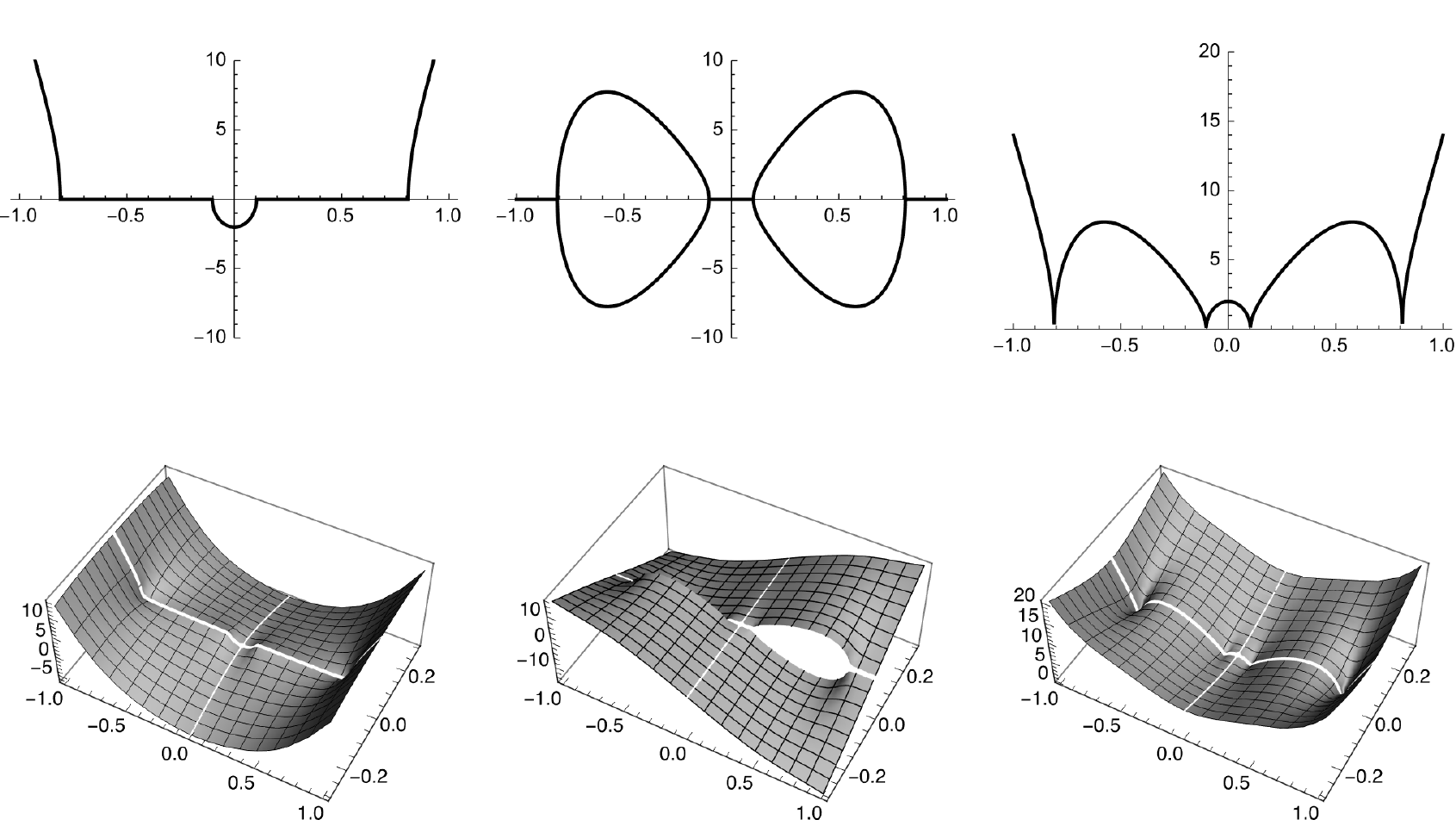}
\end{center}
\caption{Plots of $\theroot$ of Definition \ref{def:the_root_R}, for $q_+=3$, $q_-=5$. From left to right: the real and imaginary parts and the modulus. Above, the restrictions to the real axis; below, the 3D-graphs on $\mathC$. $\theroot$ is chosen as positive when $\gamma$ is positive and large, namely, for $\gamma>b=\rho_{\ell^2(V)}(\mu_1)$, and with branch cuts in $(-b,-a)\cup(a,b)$, where $ a$ and  $b$ are as in \eqref{eq:endpoints_of_the_L2-spectrum_of_M1} (it  follows from Corollary \ref{cor:0_is_in_the_l^2_spectrum},
Proposition \ref{prop:spectrum_of_M1^2} and Theorem \eqref{theo:the_spectrum_of_M1}
that $[-b,-a]\cup \{0\} \cup[a,b]$ is the $\ell^2$-spectrum of $\mu_1$). This choice of branch cuts and of determination  for $\gamma>b$ forces $\theroot$ to be negative in $-a<\gamma<a$ (Proposition \ref{prop:change_of_sign_in_analytic_continuation_of_a_square_root}).
 Along the two intervals the real part of $\theroot$ is not differentiable, and the imaginary part has a jump.
}
\label{Fig:the_square_root_in_B(gamma)}
\end{figure}

\begin{figure}[h!]
\begin{center}
    \includegraphics[scale=0.6]
{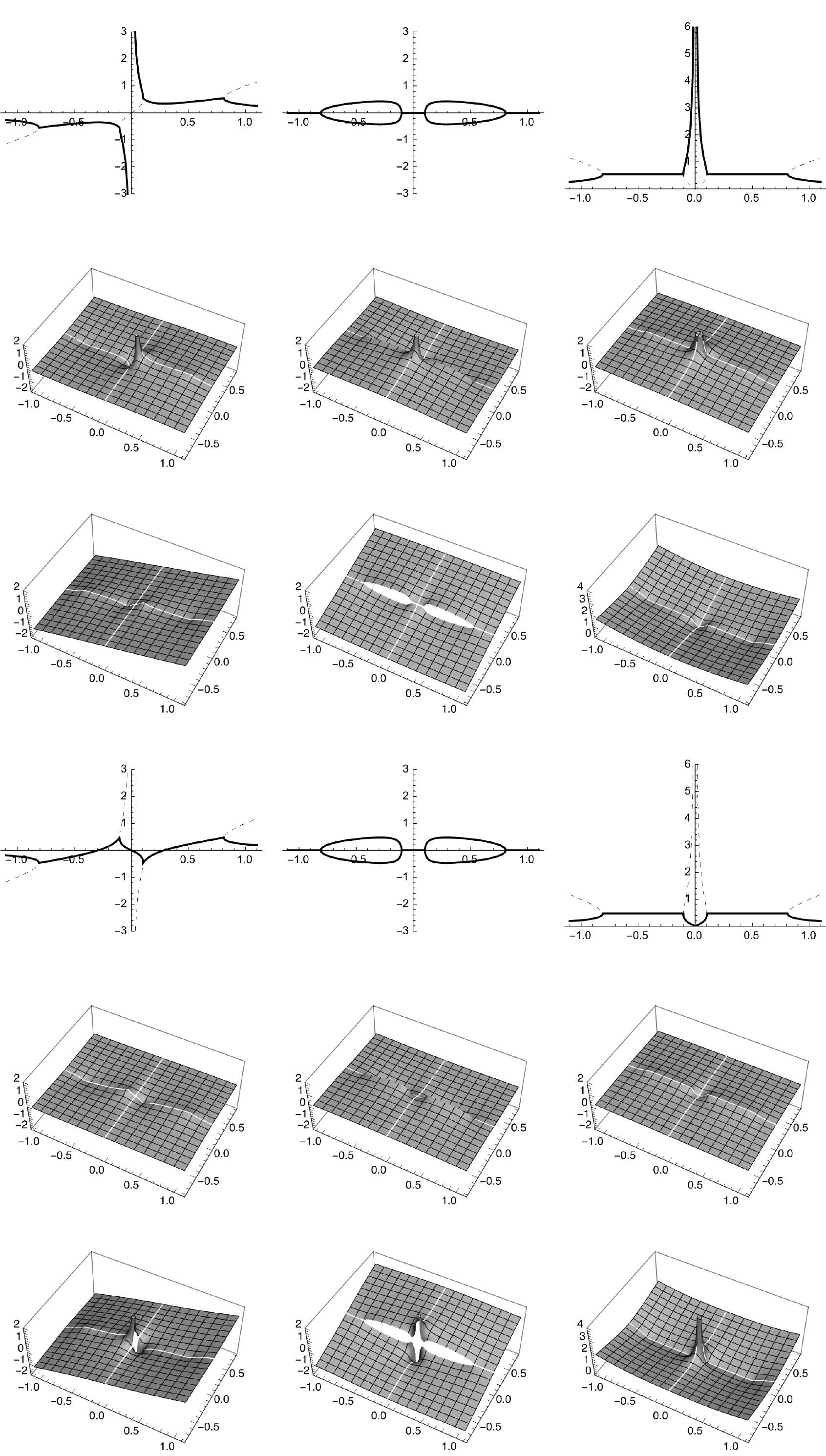}
\end{center}
\caption{The functions $F^\pm$, $\widetilde F^\pm$
for $q_+=3$, $q_-=5$. Above, in the first row, the plot of $F^+$ (solid) and $\widetilde F^+$ (dashed where different)
for real $\gamma$ (from left to right: real and imaginary parts and modulus). In the next two rows, the 3D-plots of $F^+$, respectively $\widetilde F^+$,.
Below, the same plots for $F^-$, $\widetilde F^-$. The top-right plot confirms \eqref{eq:|F(gamma)|=constant_in_the_spectrum}.
}
\label{Fig:graphs_of_F+_andF-}
\end{figure}

\begin{figure}[h!]
\begin{center}
    \includegraphics[scale=0.9]
{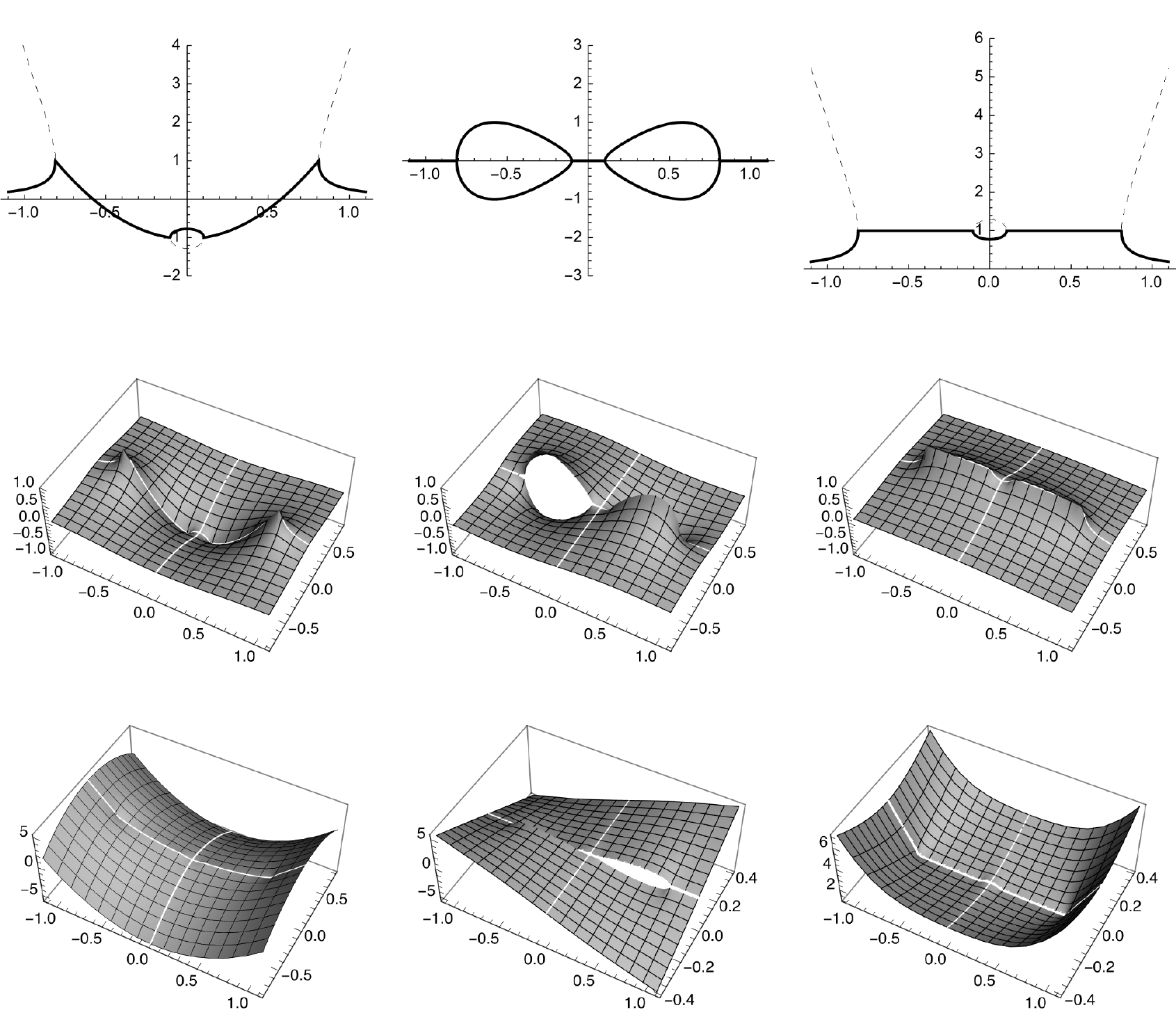}
\end{center}
\caption{The functions $\qFF$ and $\tildeqFF$
for $q_+=3$, $q_-=5$. Above, in the first row, the plot of $\qFF$ (solid) and $\tildeqFF$ (dashed where different)
for real $\gamma$ (from left to right: real and imaginary parts and modulus). In the next two rows, the 3D-plots of $\qFF$, respectively $\tildeqFF$, for complex $\gamma$.  The values of $\qFF$ and $\tildeqFF$ coincide at  $\pm a$ and $\pm b$, those of
$\Real\qFF$, $\Real\tildeqFF$ and of $|\qFF|$, $|\tildeqFF|$ coincide in $[-b,-a]\cup[a,b]$, because, for real $\gamma$,  $\qFF$ is real if and only if $\gamma$ does not belong to these branch cuts; the imaginary part has jumps there. The top right plot confirms \eqref{eq:|B(gamma)|=1_in_the_spectrum}.
}
\label{Fig:the_graph_of_B(gamma)}
\end{figure}

\section{The spectrum of the square of the Laplacian}\label{Sec:Spectrum_of_M2}
The stochastic isotropic transition operators of step 2 on $V_\pm$ are
\[
\mu_2 f(v)= \frac 1{(q_v +1)\widetilde q_v}\sum_{\dist(w,v)=2} f(w).
\]
Let us write $\mu_2^\pm = \mu_2\left|_{V_\pm}\vphantom{|_{V_\pm}}\right.$. Then, by \eqref{eq:recurrence_for_M1},
\begin{equation}\label{eq:recurrence_for_M1^2}
\mu_1^2 \mu_{2n} 
=
\frac 1{(q_+ +1)(q_- +1)} \,\bigl(\mu_{2n-2} + (q_+ + q_-) \mu_{2n} + q_+q_-\mu_{2n+2}\bigr),
\end{equation}
and ,
by \eqref{eq:recurrence_for_M1_first_step},
\begin{equation}\label{eq:the semi-homogeneous_Laplacian_squared_and_M2}
\mu_2 f(v)= \frac {\widetilde q_v +1} {\widetilde q_v} \mu_1^2 f(v) - \frac 1{\widetilde q_v} f(v).
\end{equation}
Hence
\begin{equation}\label{eq:recurrence_for_M2}
\begin{split}
(\mu_2 \mu_{2n})|_{V_\pm} 
&=
\frac 1{(q_\pm +1)q_\mp} \,\bigl(\mu_{2n-2} + (q_+ + q_-) \mu_{2n} + q_+q_-\mu_{2n+2}\bigr)  - \frac1{q_\mp}\;\mu_{2n}\\
&= \frac 1{(q_\pm +1)q_\mp}  \,\bigl(\mu_{2n-2} + (q_\mp -1) \mu_{2n} + q_+q_-\mu_{2n+2}\bigr).
\end{split}
\end{equation}
The expression for $\mu_1\mu_n f(v)$ in \eqref{eq:recurrence_for_M1} depends on the parity of $v$, that is, is different on $V_+$ and $V_-$, and by iterating \eqref{eq:recurrence_for_M1_first_step} we see that the analogous expression for $\mu_1^2\mu_{2n}$ also does. Similarly, the expression for $\mu_2\mu_{2n}$ does depend on the parity. Denote by $S$ the operator that acts on each function $f$ supported on $V_{\pm}$ by multiplication by a constant $c_\pm$. Now let $S_1$ be the operator of this type obtained by choosing $c_\pm=-1/q_\mp$, and $S_2$ the operator associated to the choice $c_\pm=(q_\mp +1)/q_\mp$. Then \eqref{eq:the semi-homogeneous_Laplacian_squared_and_M2}
becomes $\mu_2= S_2 \mu_1^2 + S_1$. It is clear that $S_1$ and $S_2$ commute with $\mu_1^2$, and that $S_1$ commutes with  $S_2 \mu_1^2$. Moreover, the orthogonal subspaces $\ell^2(V_\pm)\subset \ell^2(V)$ are invariant under $\mu_1^2$, $\mu_2$, $S_1$ and $S_2$. Therefore the spectrum  $\spectrum (\mu_1^2)$ on $\ell^2(V)$ is the union of the spectra of its restrictions to $\ell^2(V_\pm)$, and on each subspace $\ell^2(V_\pm)$ one has, by \eqref
{eq:recurrence_for_M1_first_step},
\begin{equation}\label{eq:sum_of_spectra}
\spectrum(\mu_1^2)|_{\ell^2(V_\pm)} = \frac {q_\mp+1}{q_\mp}\,\spectrum(\mu_2)|_{\ell^2(V_\pm)} - \frac 1{q_\mp}.
\end{equation}

The operator $\mu_1$ acts on $\ell^2(V)=\ell^2(V_+)\oplus\ell^2(V_-)$ and interchanges the complementary orthogonal subspaces $\ell^2(V_+)$ and $\ell^2(V_-)$; by 
Remark \ref{rem:reversible}, it is not self-adjoint on $\ell^2(V)$ but is self-adjoint on 
 $\ell^2(V, m)$. By \eqref{eq:the_reversibility_weight}, the weight $m$ is constant
 on each of $V_+$ and $V_-$ (with two different constants); hence, by \eqref{eq:recurrence_for_M1_first_step}, $\mu_1^2$ is self-adjoint on each of its invariant subspaces $\ell^2(V_+)$ and $\ell^2(V_-)$. Since these subspaces are orthogonal, $\mu_1^2$ is self-adjoint on $\ell^2(V)$; for the same reason, $\mu_1^2$ is self-adjoint on $\ell^2(V,m)$.
 Thus $\mu_1$ is self-adjoint on $\ell^2(V,m)$,  hence it is a square root of $\mu_1^2$ acting on the same space. Therefore $\spectrum(\mu_1;\,\ell^2(V,m))$ is a square root of $\spectrum(\mu_1^2;\,\ell^2(V,m))$, and it is invariant under multiplication by $-1$ because for each $\gamma$-eigenfunction $f$ the function $\epsilon f$ is an eigenfunction of eigenvalue $-\gamma$, where $\epsilon$ is the parity function, $\epsilon(v)=(-1)^{|v|}$.

By Remark \ref{rem:reversible},   $\spectrum(\mu_1^2;\,\ell^2(V)) = \spectrum(\mu_1^2;\,\ell^2(V,m))$, and $\spectrum(\mu_1;\,\ell^2(V)) = \spectrum(\mu_1;\,\ell^2(V,m))$.
 Then it easily follows from functional calculus for operators on Banach spaces (see for instance \cite{Bishop}) that the last identity holds also when $\mu_1$ is regarded as an operator on $\ell^p(V)$ and $\ell^p(V,m)$.

The restrictions of each $\gamma$-eigenfunction of $\mu_1^2$ to its invariant subspaces $\ell^2(V_\pm)$ are separately $\gamma$-eigen\-functions; conversely, every $\gamma$-eigenfunction of $\mu_1^2$ acting on $\ell^2(V_\pm)$ extends to the $\gamma$-eigenfunction  of $\mu_1^2$ on $\ell^2(V)$ that vanishes in $V_\mp$.
Hence 
$
 \spectrum(\mu_1^2;\ell^2(V))= \spectrum(\mu_1^2;\ell^2(V_+)) \cup \spectrum(\mu_1^2;\ell^2(V_-))$,
and so
\begin{equation}\label{eq:spectrum_of_M1}
\spectrum(\mu_1;\,\ell^2(V))=\sqrt{\spectrum(\mu_1^2;\ell^2(V_+))} \cup \sqrt{\spectrum(\mu_1^2;\ell^2(V_-))}.
\end{equation}
Therefore, in order to compute the spectrum of $\mu_1$ on $\ell^2(V)$, by \eqref{eq:sum_of_spectra} it is enough to compute the spectrum of $\mu_2$ on $\ell^2(V_\pm)$ and extract the compex square root..
In the terminology of \cite
{Iozzi&Picardello-Springer}, $V_\pm$ is a symmetric polygonal graph built with complete polygons consisting of $q_\mp +1$ vertices, and each vertex belongs to $q_\pm +1$ polygons. Indeed each neighbor $w$  of a vertex $v$ in the graph $V_\pm$ has distance 2 from $v$ in $V$, and the middle vertex of the path $[v,w]$ is one of the $q_\mp+1$ neighbors of $v$ in $V$.
The interested reader can find more details in the next Remark. 
\begin{remark}\label{rem:barycentric_subdivision}
Let us consider the graphs $\graphGamma_{q_+,\,q_-}$ with set of vertices $ V_+$, and $\graphGamma_{q_-,\,q_+}$ with set of vertices  $V_-$,   such that two vertices are adjacent if they have distance $2$ in $T=T_{q_+,q_-}$.
Then $\graphGamma_{q_+,\,q_-}$ is the polygonal graph where $q_+ +1$ complete polygons (including all diagonal edges), each  consisting of $q_- +1$ vertices,  join at each vertex; a symmetric description holds for $\graphGamma_{q_-,\,q_+}$. These graphs were introduced in \ocite{Iozzi&Picardello-Springer}, and further studied in \ocites{Kuhn&Soardi, Faraut&Picardello, Casadio-Tarabusi&Picardello-algebras_generated_by_Laplacians}; for instance, if we would permit $q_-=1$ then $\graphGamma_{q_+,\,q_-}$ would be the homogeneous tree with homogeneity degree $q_+ +1$. A  bijective correspondence between $V_-$ and the
set  $\mathcal P_+$ of polygons of $\graphGamma_{q_+,\,q_-}$ is obtained by mapping
each $v\in V_-$ to the polygon whose vertices are the $T$-neighbors of $v$ in $V_+$ and whose edges are all the edges of $\graphGamma_{q_+,\,q_-}$ that connect any two such neighbors.
Define $\tau\colon V\to  V _+\cup \mathcal P_+$ as the identity map on $V_+$ and so that $\tau |_{V_-}$ is the correspondence defined above. Then, for every $v_-\in V_-$ and $v_+\in V_+$, we have that
$\tau(v_+)=v_+$ is a vertex of the polygon $\tau(v_-)\in \mathcal P_+$ if and only if $v_-\sim v_+$.
In other words, we are considering the \emph{barycentric subdivision} of $\graphGamma_{q_+,\,q_-}$, 
that 
preserves the existing vertices (i.e., the set $V_+$), replaces every polygon $p\in\mathcal P_+$ with a new vertex $v_p$ (its \emph{barycenter}), deletes every existing edge, and adds a new edge from $v_p$ to each vertex of $p$. By this process we obtain  a graph isomorphic to $T$. Indeed, 
  each vertex of the graph $\graphGamma_{q_+,\,q_-}$ belongs to $V_+$ and to $q_+ +1$ polygons $p_1,\dotsc,p_{q_+ +1}$, and its neighbors in $T$  are the $q_+ + 1$ vertices $\tau^{-1}(p_1),\dotsc,\tau^{-1}(p_{q_+ +1})$; therefore the correspondence $\tau$ intertwines the group of automorphisms of $\graphGamma_{q_+,\,q_-}$ with the subgroup of $\Aut T$ that preserves parity, that is the whole of $\Aut T$ in the strictly semi-homogeneous setting.
In this settings, the group of automorphisms of $T$ (and of $\graphGamma_{q_+,\,q_-}$) has a simply transitive  subgroup isomorphic to the free product of cyclic groups $*_{i=1}^{q_+ +1} \mathZ_{q_- +1}$ \cite{Iozzi&Picardello-Springer}. The Laplace operator 
of $\graphGamma_{q_+,\,q_-}$ is the operator on functions on  $V_+$ that at each vertex yields the average on all its neighbors in the graph. 
Since the $\graphGamma_{q_+,\,q_-}$-neighbors of  $v\in V_+$ are its two-step neighbors in $T$, this operator
  is 
  the restriction to $V_+$ of $\mu_2$ acting on $V(T)$. 
 
 A symmetric description holds for
 the Laplace operator 
on $\graphGamma_{q_-,\,q_+}$.
 
\end{remark}

 It was observed in \cite
{Iozzi&Picardello-Springer}*{Lemma 1} that the stochastic transition operator $P^+_n$ of step $n$ on $V_+$ in the sense of $\graphGamma_{q_+,\,q_-}$ satisfies the recurrence relation
\begin{equation}\label{eq:recurrence_for_Laplacian_on_symmetric_graph}
P_{1 }^+  P_{n }^+
=
\frac 1{(q_+ +1)q_-} \, \bigl( P_{n-1}^+ 
+ (q_- -1) P_{n }^+ 
+ q_+q_- P_{n+1}^+ \bigr),
\end{equation}
and a similar relation holds for $P^-_n$ on $V_-$ by exchanging $q_+$ and $q_-$.
This is exactly the same  as \eqref{eq:recurrence_for_M2}. By \eqref{eq:the_reversibility_weight} $m$ is identically one on both $V_+$ and $V_-$, because vertices of the same parity have even distance in $V$. The nearest neighbor transition operator $P$ on $V_\pm$ is $\mu_2$, and it is self-adjoint on both $V_+$ and $V_-$ by Remark  \ref{rem:reversible}.

Therefore the spectra of the Laplacians $\mu^\pm_2$ on $\ell^2(V_\pm)$
are the same as in
\cite
{Iozzi&Picardello-Springer}*{Theorem 3}. We now obtain from this reference the following useful parametrization, with respect to a suitable complex variable $z$, of the eigenvalue of $\mu_2^{\pm}$, 
respectively:
\begin{equation}\label{eq:Gamma}
\Gamma_\pm(z)=
\dfrac{\qmean^{2z}+\qmean^{2(1-z)}+q_\mp -1}{(q_\pm +1)q_\mp}.
\end{equation}
We limit attention to the spectrum on $\ell^2(V_+)$, hence, for simplicity, from now on we write $\Gamma(z)$ instead of $\Gamma_+(z)$ and $\mu_2$ instead of $\mu_2^+$; the case of $\ell^2(V_-)$ will then follow by interchanging $q_+$ and $q_-$.
Then \cite
{Iozzi&Picardello-Springer}*{Theorem 3} states that the spectrum of $P_1$ on  $\ell^p(V_+)$ for $1\leqslant p \leqslant 2$ is 
$E_p\cup D_p$, where 
$E_p$ is the region
\[
E_p = \Bigl\{\Gamma(z)\colon 1-\frac1p\leqslant \Real z \leqslant \frac1p\Bigr\}.
\]
and
\[
D_p = 
\begin{cases}
-1/q_- &\text { if $q_->3$\; and \;$p>1+\dfrac{\ln q_+}{\ln q_-}$}\,,\\[.2cm]
\emptyset &\text { otherwise. }
\end{cases}
\]
This statement, however, can be made simpler, because $D_p$ is actually in the resolvent of $\mu_2$ for $1\leqslant p \leqslant 2$. Indeed, for $v\in V_+$
the proof of \cite{Iozzi&Picardello-Springer}*{Theorem 3} shows that the function
$u_z(v)=\qmean^{-z|v|}$ satisfies the resolvent equation for $\mu_2$ acting on $V_+$, i.e.,
\begin{equation}\label{eq:resolvent_equation_for_M2}
(\mu_2-\Gamma(z) \delta_{v_0})* f_z = \delta_{v_0}
\end{equation}
where the convolution product is induced by the action of $G$ on $V_+$, hence 
\begin{equation*}
(\mu_2-\Gamma(z) \delta_{v_0})* f_z(v_0) = \qmean^{-2z} -\Gamma(z).
\end{equation*}
Therefore the resolvent equation \eqref{eq:resolvent_equation_for_M2} has a solution only if 
the right-hand side $\qmean^{-2z} -\Gamma(z)$ is non-zero.

The solutions of
$\qmean^{-2z} = \Gamma(z)$
are $z\in Y_1=\left\{ i k\pi /\ln \qmean,\, k\in\mathZ\right\}$ and $z\in Y_2=\left\{ \dfrac{ \ln q_- + i(2k+1)\pi}{\ln (q_+ q_-)} \right\}$ with $k\in\mathZ$. 
In the homogeneous set-up, $\qmean=q_+=q_-=q$, $\Gamma(z)=\dfrac{q+1}q\;\gamma(z)^2 -\dfrac 1q$, and  if $z\in Y_1$ then $\gamma(z)=(-1)^k$ and $\Gamma(z)=1$, whereas
 if $z\in Y_2$ then $\gamma(z)=0$ and $\Gamma(z)=-1/q$.

For general semi-homogeneous trees, if $z\in Y_1$, then $|u_z|$ is constant, hence $u_z$ is not a bounded convolution operator on $\ell^p$. For $z\in Y_2$ one has  $\Gamma(z)\in D_p$; however, $u_{1-z}$ is an eigenfunction of $\mu_2$ with the same eigenvalue, because $\Gamma(1-z)= \Gamma(z)$, but now $\qmean^{-2(1-z)} \neq \Gamma(1-z)$. Hence, for $z\notin E_p$ and $\Real z> 1/p$, the resolvent equation of $\mu_2$ is satisfied by
 both $f_z=u_z/u_z(v_0)$ and $f_{1-z}$, and the non-vanishing condition at $v_0$ is satisfied by one of these two functions.  We have
\[
f_z(|v|)=\frac {(q_+ +1)q_-} {q_- \qmean^{-z} - \qmean^z - q_- +1} \; \qmean^{-z|v|}
\]
that, by the proof of \cite{Iozzi&Picardello-Springer}*{Theorem 3}, acts on on $\ell^p(V_+)$
as a bounded convolution operator. Therefore the singleton $D_p$ is not contained in the spectrum of $\mu_2$ on $\ell^p(V_+)$. The spectrum on $\ell^p(V_-)$ has a similar description. 

In particular, for $1\leqslant p \leqslant 2$, the spectrum of $\mu_2$ on $\ell^p(V_+)$ is 
\[
\biggl\{ \Gamma(z)\colon 1-\frac1p \leqslant \Real z \leqslant \frac1p\biggr\}
=\biggl\{w\in\mathC\colon
 \biggl(\dfrac{(q_+ +1)q_- \Real w - (q_- -1)}{\qmean^{\frac2p} + \qmean^{{\scriptscriptstyle 2}-\frac2p}}\biggr)^2
+\biggl(\dfrac{(q_+ +1)q_- \Imag w           }{\qmean^{\frac2p} - \qmean^{{\scriptscriptstyle 2}-\frac2p}}\biggr)^2
\leqslant 1
\biggr\}
\]
whose boundary 
is the ellipse
\[
\biggl\{\Gamma(z)\colon z=\frac1p +it,\;t\in\mathR\biggr\}
=\biggl\{
\frac{(\qmean^{\frac2p} + \qmean^{{\scriptscriptstyle 2}-\frac2p})\cos(t\ln \qmean)+q_- -1}{(q_+ +1)q_-} +i \frac{(\qmean^{\frac2p} - \qmean^{{\scriptscriptstyle 2}-\frac2p})\sin(t\ln \qmean)}{(q_+ +1)q_-} \colon \;0\leqslant t< \frac{2\pi}{\ln \qmean}\biggr\}
\]
with real major axis, 
center $C=\dfrac{q_- -1}{(q_+ +1)q_-}$\,,
real vertices
$C\pm\dfrac{\qmean^{\frac2p} + \qmean^{{\scriptscriptstyle 2}-\frac2p}}{(q_+ +1)q_-}\;,
$
top and bottom vertices
$
C \pm i\,\dfrac{\qmean^{\frac2p} - \qmean^{{\scriptscriptstyle 2}-\frac2p}}{(q_+ +1)q_-}$\;.

Note that $\Gamma(z)=-1/q_-$ if and only if $\qmean^{2z}+\qmean^{2(1-z)}+q_ + +q_-=0$, that is, $z= \dfrac{ \ln q_\pm + i(2k+1)\pi}{\ln (q_+ q_-)}$ for $k\in\mathZ$, or equivalently $z\in Y_2\cup(1-Y_2)$.

It now  follows from \eqref{eq:recurrence_for_M1_first_step} that the spectrum of $\mu_1^2$ on $\ell^p(V_+)$ is given by 
\begin{equation}\label{eq:spectrum_of_M1^2}
\Sigma_p=\frac{q_-}{q_- +1} E_p + \frac 1{q_- +1},
\end{equation}
and similarly for $\ell^p(V_-)$.  
By carrying out the computations, we obtain:

\begin{proposition}\label{prop:spectrum_of_M1^2}
 The spectrum of $\mu_1^2$ on $\ell^p(V_+)$ is the elliptical region $\Sigma_p$ with axes parallel to the real and imaginary axes, vertical semi-axis of length $\alpha=\dfrac{\qmean^{\frac2p}+\qmean^{2-\frac2p}}{(q_++1)(q_-+1)q_-}$, horizontal semi-axis of length $\beta=\dfrac{\qmean^{\frac2p}-\qmean^{2-\frac2p}}{(q_++1)(q_-+1)q_-}$
and center in the positive real axis, independent of $p$, at $\eta=\dfrac{q_+ + q_-}{(q_++1)(q_-+1)q_-}$. This region degenerates to the segment connecting the foci if $p=2$.
\end{proposition}

Equivalently, in accordance with   \eqref{eq:the_nonhomogeneous_eigenvalue_map}
and
\eqref{eq:Gamma}, each eigenfunction of $\mu_1$ on $V$ with eigenvalue $\gamma(z)$, restricted to $V_\pm$, is an eigenfunction of $\mu_2^\pm$ with eigenvalue $\Gamma_\pm(z)$ such that
\begin{equation}\label{eq:eigenvalues_of_M1^2_and_M2}
\gamma(z)^2=\dfrac{q_\mp \Gamma_\pm(z) +1}{q_\mp +1},
\end{equation}
By \eqref{eq:Gamma}, the right-hand side is invariant by exchanging the plus and minus signs. In fact, we also have
\begin{equation}\label{eq:the_semihomogeneous_eigenvalue_map_squared}
\gamma(z)^2
=\dfrac{\qmean^{2z}+\qmean^{2(1-z)}+q_+ +q_-}{(q_+ +1)(q_- +1)}
=\dfrac{(q_+^z +q_-^{1-z})(q_-^z +q_+^{1-z})}{(q_+ +1)(q_- +1)} = \frac{q_+ + q_- + 2\qmean \cosh\Bigl(\ln(q_+q_-)\bigl(z-\frac12\bigr)\Bigr)}{(q_+ +1)(q_- +1)}.
\end{equation}
Clearly, $\gamma(z)^2=\gamma(1-z)^2$.
The function $\gamma(z)^2$ vanishes at the points $\dfrac12 + \dfrac{\pm\frac12(\ln q_+ - \ln q_-) + (2k+1)\pi i}{\ln q_+ + \ln q_-}$ with $k\in \mathZ$, that belong to the strip $\{0<\Real z < 1\}$.

Now we can 
extend the  eigenvalue map \eqref{eq:the_homogeneous_eigenvalue_map} of $\mu_1$ of a homogeneous tree to the semi-homogeneous set-up. Namely,
\begin{equation}\label{eq:the_semihomogeneous_eigenvalue_map}
\gamma(z)=\sqrt{\dfrac{(q_+^z +q_-^{1-z})(q_-^z +q_+^{1-z})}{(q_+ +1)(q_- +1)}}
,
\end{equation}
a two-valued function such that $\gamma\biggl(\dfrac12 + \dfrac{ik\pi}{\ln \qmean}\biggr)=\pm b$ and $\gamma\biggl(\dfrac12 + \dfrac{i(k+\frac12)\pi}{\ln \qmean}\biggr)=\pm a$ for every $k\in\mathZ$. In the homogeneous setting, the right-hand side of \eqref{eq:the_semihomogeneous_eigenvalue_map} becomes $\pm\dfrac{q^z +q^{1-z}}{q +1}$, with the same sign at every $z$; the usual choice in the literature is the plus sign.

If we choose a smooth curve $s\mapsto z(s)$ that avoids all zeros of $\gamma$, then 
each of the two values of $\gamma(z(s))$ changes continuously along the curve and
the value at any $s\neq 0$ depends on the curve. For instance, let $z(s)=z(0) + i s$. If   $\left|\Real z(0)-\dfrac12\right|<\dfrac{|\ln q_+ - \ln q_-|}{2(\ln q_+ + \ln q_-)}$ (necessarily assuming $q_+\neq q_-$), then $s\mapsto \gamma(z(s))$ is periodic with period $\pi/\ln \qmean$. Instead, if  $\left|\Real z(0)-\dfrac12\right|>\dfrac{|\ln q_+ - \ln q_-|}{2(\ln q_+ + \ln q_-)}$, then the period is $2\pi/\ln \qmean$; this also holds when $q_+=q_-$.  In either case, however, the pair of values of $\gamma(z)$ repeats after a shift of $z$ by  $i\pi/\ln \qmean$. On the other hand,
whichever of the two values of  $\gamma(\frac12)$ is chosen and then continued along  any  smooth path $s \mapsto z(s)$ with  $z(0)=\frac12$ that avoids all zeros of $\gamma$, the resulting value of $\gamma(1-z(s))$ on the symmetric path $s\mapsto 1-z(s)$ is  the same  as for $\gamma(z(s))$, as a consequence of the invariance of the radicand of \eqref{eq:the_semihomogeneous_eigenvalue_map} under this interchange. 
Hence, value by value,
\begin{equation}\label{eq:gamma(z)=gammma(1-z)}
\gamma(z)=\gamma(1-z) \quad\text{for every $z\in\mathC$,}
\end{equation}
 as in the homogeneous setting.

From \eqref{eq:the_semihomogeneous_eigenvalue_map}  we derive the expression of the inverse function of the eigenvalue map $z\mapsto \gamma(z)$ in the semi-homogeneous setting,  
namely
 \begin{equation}\label{eq:the_semihomoheneous_inverses-of_the_eigenvalue_map}
 \begin{split}
z&=\frac1{2\ln\qmean}\arccosh\frac{(q_+ +1)(q_- +1)\gamma^2-(q_+ +q_-)}{2\qmean} + \frac12\\[.2cm]
&=\frac12\log_\qmean \frac{\newGamma-(q_+ +q_-)
+\sqrt{\newGamma^2-2(q_+ +q_-)\newGamma+(q_+ -q_-)^2}}2.
\end{split}
\end{equation}
This expression extends \eqref{eq:the_homoheneous_inverses-of_the_eigenvalue_map} to the semi-homogeneous setting for every $z$ such that $\gamma(z)\notin\spectrum_{\ell^2(\mu_1)}$; indeed, as seen after \eqref{eq:n.n._semihomogeneous_visit_probabilities_at_eigenvalue_gamma}
and
\eqref{eq:F+(gamma)F-(gamma)}, the radicand vanishes at the endpoints $\pm a, \,\pm b$ of the spectrum computed in  \eqref{eq:endpoints_of_the_L2-spectrum_of_M1}, hence the determination of the square root changes sign along every simple loop around one of these four points, that is, crossing the intervals $\gamma(z)\in[-b,-a]\cup[a,b]$ once.
 
Since $\gamma$ is periodic in the imaginary direction with period $2\pi i/\ln \qmean$, and $\gamma(z)=\gamma(1-z)$, 
for every $\gamma$ exactly two of the infinite values of the inverse function $z=z(\gamma)$ belong to the strip $-\pi / \ln \qmean < \Imag z \leqslant \pi/\ln \qmean$,
 according to the determination of the (two-valued) complex square root.

We have seen in Corollary \ref{cor:inverse_of_homogeneous_eigenvalue_map} that, in the homogeneous setting, one of the two values of the inverse function was not compatible with transience of $\mu_1$. It is an open problem if the same is true in the semi-homogeneous setting.

In terms of the parameter $z$, the degeneracy conditions $\qFF=\pm 1$ given in \eqref{eq:B(gamma)=+-1_in_terms_of_gamma^2}
amount respectively to
\begin{equation}\label{eq:B(gamma)=+-1_n_terms_of_z}
\begin{alignedat}{5}
\qFF&= 1 &\quad\text{iff} \quad \qmean^{z}&=\qmean^{1-z}, &\quad\text{i.e., iff} \quad
z&= \dfrac12 + \dfrac{ik\pi}{\ln \qmean} \text{\quad for $k\in\mathZ$} &\quad\text{(that is, $\gamma=\pm b$)},
\\[.2cm]
\qFF&= -1 &\quad\text{iff}
\quad
\qmean^{z}&=-\qmean^{1-z},&\quad\text{i.e., iff}\quad
z&= \dfrac12 + \dfrac{i(k+\frac12)\pi}{\ln \qmean} \text{\quad for $k\in\mathZ$} &\quad\text{(that is, $\gamma=\pm a$)}.
\end{alignedat}
\end{equation}

\section{Connectedness of the spectrum of $\mu_1$ and spectral radius}\label{Sec:Spectrum_of_M1}
If $p=2$, then $z>0$ for every $z\in \Sigma_2$, as expected since $\mu_1^2$ is a positive self-adjoint operator. As already observed, by functional calculus for operators on Banach spaces  \cite{Bishop}
we can now compute the spectrum of $\mu_1$ on $\ell^p(V_+)$ by extracting the square root of the elements in $\Sigma_p$, and we obtain again the drawings shown in Figure \ref{Fig:semi-homogeneous_bounded_spherical_functions}. 
That is:
\begin{theorem}\label{theo:the_spectrum_of_M1}
The spectrum $S_p=\spectrum_{\ell^p(V)}(\mu_1)$, for $1<p<\infty$, is $S_p=\sqrt{\Sigma_p}$, where $\Sigma_p$ is as in \eqref{eq:spectrum_of_M1^2}. For $p>2$,  $S_p=\{\gamma\in\mathC\colon \phi(\missarg,v_0\barra \gamma)\in \ell^{p'}(V)\}$, and $S_2=\{\gamma\in\mathC\colon \phi(\missarg,v_0\barra \gamma)\in \bigcup_{p\geqslant 2} \ell^{p'}(V)\}=[-b,-a]\cup[a,b]$, with $a$ and $b$ as in \eqref{eq:endpoints_of_the_L2-spectrum_of_M1}. Moreover, $S_p=S_{p'}$.
For completeness, we give  the analytic expression of the boundary curve $b(\theta)$, for $\theta\in [0,2\pi)$, of the domain $\Sigma_p$:
\[
\begin{split}
\Real b(\theta)&= \dfrac{q_+ + q_-}{(q_+ +1)(q_- +1)} + \frac{2\qmean}{(q_+ +1)(q_- +1)}\;\cosh\biggl(\Bigl(1-\frac2p\Bigr)\ln\qmean\biggr) \cos\theta,
\\[.2cm]
\Imag b(\theta) &= \frac{2\qmean}{(q_+ +1)(q_- +1)}\;\sinh\biggl( \Bigl(1-\frac2p\Bigr)\ln\qmean\biggr) \sin\theta;
\end{split}
\]
The boundary curve of  $\spectrum_{\ell^p(V)}(\mu_1)$ is obtained by extracting the square root.
\end{theorem}

\begin{proof}
The spherical function 
of $\mu_1^2$ on $V_+$ with eigenvalue $\gamma^2\in\mathC$ extends to two spherical functions of $\mu_1$ on $V$, namely $\phi(\missarg,v_0\barra\gamma)$ and $\phi(\missarg,v_0\barra-\gamma)=\epsilon(\missarg)\,\phi(\missarg,v_0\barra\gamma)$ , where $\epsilon(v)=(-1)^{|v|}$ is the alternating function. Conversely, these two $\mu_1$-eigenfunctions are the only  ones that restrict to $\psi$ on $V_+$. Therefore, by the remarks above,
$\spectrum_{\ell^p(V)}(\mu_1)=\sqrt{\spectrum_{\ell^p(V_+)}(\mu^2_1)}$. By the same argument, $\sqrt{\spectrum_{\ell^p(V_+)}(\mu^2_1)}= \sqrt{\spectrum_{\ell^p(V_-)}(\mu^2_1)}$ and $\spectrum_{\ell^p(V_+)}(\mu^2_1) = \spectrum_{\ell^p(V_-)}(\mu^2_1)$. Therefore $S_p=\sqrt{\spectrum_{\ell^p(V_+)}(\mu^2_1)}= \sqrt{\spectrum_{\ell^p(V_-)}(\mu^2_1)}=\sqrt{\Sigma_p}$.

\end{proof}
We know from Remark \ref{rem:reversible} that, for every $p$, the spectrum of $\mu_1$ in $\ell^{p}(V)$ is symmetric around the origin: this is in agreement with
 Figure \ref{Fig:semi-homogeneous_bounded_spherical_functions}, where we also see that the spectrum in $\ell^{p}$ is disconnected for $p$ near 2, but connected for $p$ near 1. 
 
Let us compute for which $p>2$ the spherical function with eigenvalue $\gamma$ belongs to $\ell^{p}$ for $\gamma \sim 0$. By \eqref{eq:inequality_satisfied_by_the_semihomogeneous_lp_spectrum},  this amounts to the condition  $\qmean^{\frac 1{p}-\frac 1{p'}}<|\qFF|<\qmean^{\frac 1{p'}-\frac 1{p}}$ for small $\gamma$. By \eqref{eq:F+(gamma)F-(gamma)} and \eqref{eq:the_numerator_of_B(gamma)},
\begin{equation}
|\qFF|
= \frac{\bigl|-(q_+ + q_-)  +  \sqrt{(q_+ - q_-)^2}\bigr|} {2\qmean}\; + O(\gamma^2)
= \frac{\min\{q_+, q_-\}}                                 \qmean \; + O(\gamma^2).
\end{equation}
%
%
Therefore $\qmean^{\frac 1{p}-\frac 1{p'}}<|\qFF|<\qmean^{\frac 1{p'}-\frac 1{p}}$  for small $\gamma$ if and only if 
 $\qmean^{\frac 1{p}-\frac 1{p'}} < \dfrac{ \min\{q_+, q_-\}} \qmean 
< \qmean^{\frac 1{p'}-\frac 1{p}}$, that is, $\qmean^{2(1-\frac1{p'})} < \min\{q_+, q_-\} < \qmean^{2(1-\frac1{p})}$. The left-hand side inequality is equivalent to $ 2/p < \log_{\qmean} \min\{q_+, q_-\}$, that is, $p>2\; \dfrac {\ln \qmean}{\ln  \min\{q_+, q_-\}}$,  but this value is larger than 2, while $p'<2$, so the left-hand side inequality is automatically satisfied.
Similarly, the inequality on the right hand side is equivalent to $p'
< 2\; \dfrac {\ln \qmean}{\ln  \min\{q_+, q_-\}}$. Therefore, by continuity, the critical value of $p$ at which the $\ell^{p}$-spectrum splits into the two disjoint components of Figure \ref{Fig:semi-homogeneous_bounded_spherical_functions} corresponds to $p'
= 2\; \dfrac {\ln \qmean}{\ln  \min\{q_+, q_-\}}$, that is,
\[
  p= 2 \;\frac {\ln \qmean}{\ln  \max\{q_+, q_-\}} =   \frac {\ln (q_+ q_-)}{\ln  \max\{q_+, q_-\}}.
\]
This value of $p$ is exactly $p_{\mathrm{crit}}$ of \eqref{eq:pcrit} if $q_+<q_-$ and $p_{\mathrm{crit}}' $ otherwise.

\begin{remark}\label{rem:spectrum_in_the_variable_gamma^2}
By \eqref{eq:qmean_Fmean} and Proposition \ref{prop:spectrum_of_M1^2}, the spectrum of $S_p$ of $\mu_1$ on $\ell^p(V)$, expressed in terms of the variable $\gamma^2$, is also an elliptical region, with foci at the non-negative real points
\begin{equation}\label{eq:foci}
\dfrac{q_+ + q_-\pm 2\qmean}{(q_+ +1)(q_- +1)}
\end{equation}
independently of $p$. If $\alpha$ and $\eta$ are as in Proposition \ref{prop:spectrum_of_M1^2},
the center of $S_p$ is the real point $\eta$, independent of $p$,
$2\alpha$, 
with endpoints at $\dfrac{q_+ + q_- \pm (\qmean^{\frac2{p}} + \qmean^{\frac2{p'}})}
{(q_+ +1)(q_- +1)}
$.
The spectrum is invariant if we interchange $q_+$ and $q_-$.
One of those major axis endpoints is $0$ exactly if $\eta=\alpha$; this happens if and only if
$(q_+q_-)^{\frac1p}$ equals $q_+$ 
or $q_-$ (or equivalently $(q_+q_-)^{\frac1{p'}}$ equals $q_-$  or $q_+$, respectively). 
This amounts to
$p'/p=\ln q_+/\ln q_-$ or $p'/p=\ln q_-/\ln q_+$, respectively.
In other words, the boundary of $S_p$ contains 0 if and only if
5
\[
p=\begin{cases}
\dfrac {\ln (q_+ q_-)}{\ln  \min\{q_+, q_-\}} \qquad&\text{if } p\geqslant 2,\\[.5cm]
\dfrac {\ln (q_+ q_-)}{\ln  \max\{q_+, q_-\}} \qquad&\text{if } 1\leqslant p\leqslant 2.
\end{cases}\\[.1cm]
\]
If $p\geqslant 2$, it  coincides with $\pcrit$ if $q_+>q_-$ and with $\pcrit'$  otherwise. If $p\leqslant 2$ these two choices are interchanged: $p=\pcrit'$ if $q_+>q_-$ and with $\pcrit$  otherwise.

Again by \eqref{eq:qmean_Fmean}, the condition  $\qFF=\pm 1$ of Theorem \ref{theo:computation_of_semi-homogeneous_spherical_functions_via_Poisson_kernel} holds
if and only if $\delta =\pm 1$, that is, exactly when $\gamma^2$ coincides with one of the foci \eqref{eq:foci}; this is equivalent to \eqref{eq:B(gamma)=+1} and \eqref{eq:B(gamma)=-1}.
In the homogeneous case $q_+ =q_- =q$, the values of $\gamma$ that satisfy \eqref{eq:B(gamma)=+1} and \eqref{eq:B(gamma)=-1} are $\gamma=0, \,\pm\rho$, where $\rho$ is the spectral radius of $\mu_1$ on $\ell^2(V)$. In fact, the  eigenvalue map \eqref{eq:the_homogeneous_eigenvalue_map} 
 shows that the eigenvalue $\gamma=\rho$ is attained at $z=\frac12 + 2ki\pi/\ln q$, and the eigenvalue $\gamma=-\rho$ at $z=\frac12 + (2k+1)i\pi/\ln q$ ($k\in\mathZ$), and these are the values of $z$ for which the spherical function $\phi _z(n)$ is an exponential in the variable $n=|v|$ multiplied by a polynomial.
For all the other eigenvalues, the spherical function 
$\phi_z(v)$ 
 splits as a linear combination of two exponentials, 
 \begin{equation}\label{eq:homogeneous_spherical_function}
 \phi _z(v)=c(z)\, q^{-z|v|}+c(1-z)\,q^{-(1-z)|v|} 
\qquad\text{where}\quad c(z)=\dfrac 1{q+1}\; \dfrac{q^{1-z}-q^{z-1}}{q^{-z}-q^{z-1}}
 \end{equation}
 \cite{Figa-Talamanca&Picardello}*{Chapter 3, Theorem 2.2(i)}.
In particular, this expression holds for the eigenvalue $\gamma(z)=0$, attained at $z= \frac12 +i(2k+1)\pi/(2\ln q), \;k\in\mathN$, and 
 for these values of $z$  we have $c(z)=c(1-z)=\dfrac 12$
and $\phi_z(v)=\phi(v, v_0\barra 0)$  is as in \eqref{eq:the_semi-homogeneous_spherical_function_of_eigenvalue_zero}.

\graphicspath{{Immagini_degli_spettri_semiomogenei/}}

\end{remark}
It is known from \cite{Figa-Talamanca&Picardello}*{Chapter 3, Section III} that , on homogeneous trees, the spectrum of $\mu_1$ on $\ell^p(V)$ is the interior of an ellipse centered at the origin, for every $1\leqslant p <\infty$. We show that this is not true on strictly semi-homogeneous trees.

\begin{proposition}\label{prop:ellipses_and_their_squares}
For every $1\leqslant p <\infty$, the spectrum $S_p$ of $\mu_1$ on $\ell^p(V)$ 
\end{proposition}
\begin{proof}
We have already remarked that $S_p$ is invariant under mirror reflection around $0$, so, if it is an elliptical region, its center is $0$.

Let $E$ be an ellipse centered at the origin with semi-axes $a$ and $b$: with respect to $0\leqslant \theta < 2\pi$, the points of $E$ are parametrized by $s(\theta)=a\cos\theta + ib\sin\theta$. It is immediately seen that $s(\theta)^2=(a^2+b^2)/2\,\cos(2\theta) + iab\sin(2\theta) + a^2 - b^2$.
Then $E^2$ is an ellipse too, its semiaxes (on the real and imaginary axes, respectively) are $\alpha=(a^2+b^2)/2$, $\beta=ab$ and its center is $\eta=(a^2-b^2)/2$. In particular, 
\[
\alpha^2-\eta^2=\beta^2
\]
(and
$E^2$ is centered at $0$ if and only if $E$ is a circle). On the other hand, the spectrum $\Sigma_p$ of $\mu_1^2$ on $\ell^p(V_+)$, computed in Proposition \ref{prop:spectrum_of_M1^2}, is the square of the spectrum of $\mu_1$ on $\ell^p(V)$ and is an elliptical region,  with $\alpha$, $\beta$ and $\eta$ as 
Proposition \ref{prop:spectrum_of_M1^2} (see also Remark \ref{rem:spectrum_in_the_variable_gamma^2}). It is immediately verified that $\alpha^2-\eta^2-\beta^2=(q_+ - q_-)^2$, whence the statement.
\end{proof}

Let us now compute the spectral radius of $\mu_1$ on $\ell^{p}$. This is the maximum of the distance from $0$ of the boundary of the spectrum, and is attained at a value $\gamma=\rho(\mu_1)$ such that
\begin{equation}\label{eq:rho_p}
\rho(\mu_1)>0,
\end{equation}
because the center $\eta$  
of the ellipse belongs to the positive real half-line; 
$\rho(\mu_1)$ is a real point because of \eqref{eq:foci}. Namely, by Corollary \ref{theo:boundedness_of_F+(gamma)F-(gamma)}, this value of $\gamma$ is the largest  such that
one of the identities $ |\qFF| =\qmean^{\frac 1{p'}-\frac 1{p}}$ or $ |\qFF| = \qmean^{\frac 1{p}-\frac 1{p'}}$   is satisfied. 

\begin{theorem}\label{theo:spectral_radius_in_semi-homogeneous_lp}
For every $1\leqslant p \leqslant \infty$, the spectral radius on $\ell^p(V)$ of the Laplace operator on the semi-homogeneous tree $T_{q_+,\,q_-}$ is
\begin{equation*}
\sqrt{\dfrac{q_+ +q_- +\qmean^{2/p}+\qmean^{2/p'}}{(q_+ +1)(q_- +1)}}.
\end{equation*}
For $p=2$, this is exactly the largest positive value of $\gamma$ such that $\qFF=1$, see \eqref{eq:B(gamma)=+1}.
\end{theorem}
\begin{proof}
By \eqref{eq:F+(gamma)F-(gamma)}, writing $\newGamma$ as in \eqref{eq:definition_of_newGamma} we have

\begin{equation}\label{eq:F+(gamma)F-(gamma-again)}
\qFF = -\frac{q_+ + q_- - \newGamma}{2\qmean} 
- \frac{ \sqrt{(q_+ - q_-)^2 -2 (q_+ + q_-)\newGamma +\newGamma^2}}{2\qmean}\, ,
\end{equation}
where the square root is positive  when $\newGamma\in\mathR$ is positive and large.

We look for the largest real solutions of the equations 
\begin{equation}\label{eq:first_identity_for_spectral_radius_on_semi-homogeneous_Lp}
|\qFF| =\qmean^{\frac 1{p'}-\frac 1{p}}
\end{equation}
and
\begin{equation}\label{eq:second_identity_for_spectral_radius_on_semi-homogeneous_Lp}
|\qFF| = \qmean^{\frac 1{p}-\frac 1{p'}}.
\end{equation}
 
 By \eqref{eq:rho_p}, the $\ell^p(V)$-spectral radius is the largest among the real positive values of $\gamma$ that satisfy these equations.

As a function of $\newGamma$, the quadratic expression in the radicand of \eqref{eq:F+(gamma)F-(gamma-again)} 
is always positive, since its discriminant is $16\,q_+q_->0$. Therefore $\qFF$ is real for $\newGamma>0$, that is, for real $\gamma$.
Moreover, as in the first equation of \eqref{eq:F+(gamma)F-(gamma)},
this radicand is 
\begin{equation}\label{eq:numerator_of_qFF_simplified}
(q_+ - q_-)^2 -2 (q_+ + q_-)\newGamma +\newGamma^2
= (\newGamma-(q_+ + q_- ))^2 -4q_+q_-.
\end{equation}
Hence $\qFF>0$ for every real $\gamma$, and the modulus in \eqref{eq:first_identity_for_spectral_radius_on_semi-homogeneous_Lp} and \eqref{eq:second_identity_for_spectral_radius_on_semi-homogeneous_Lp} is irrelevant for $\gamma\in\mathR$.
By \eqref{eq:numerator_of_qFF_simplified}, the largest of the two real solutions of  \eqref{eq:first_identity_for_spectral_radius_on_semi-homogeneous_Lp}
is easily found:
\[
\newGamma=\qmean^{\frac2p}+\qmean^\frac2{p'}+q_++q_-.
\]
Since the right-hand side is symmetric by exchanging $p$ and $p'$, the largest solution of \eqref{eq:second_identity_for_spectral_radius_on_semi-homogeneous_Lp}
is the same. The statement now follows because $\newGamma=(q_++1)(q_-+1)\gamma^2$.
\end{proof}

\begin{figure}
\tikzmath{
  \ticklength=.03;
  \qplus =5;
  \qminus=2;
  \qmeanvalue=sqrt (\qplus   * \qminus);
  \spmin     =sqrt((\qplus   + \qminus+2*\qmeanvalue)/
                  ((\qplus+1)*(\qminus+1)));
  \pmax=10.55;
}
{ 
\tikzset{
 every picture/.style={inner sep=3},
 every    node/.style={font=\footnotesize},
 every   label/.style={font=\footnotesize},
}
\begin{tikzpicture}[xscale=1,yscale=5]
 \draw[variable=\p,domain=1:\pmax,samples=100,smooth]
  plot(\p,{sqrt( (\qplus+\qminus
                +(\qplus*\qminus)^(1/\p)
                +(\qplus*\qminus)^(1-1/\p))/
                ((\qplus+1)*(\qminus+1)))});
 \draw[gray,->    ](-  .4,  . )--(\pmax,   .        ) node[black,below right]{$ p$};
 \draw[gray,->    ](   . ,- .1)--(   . ,  1.1       );
 \draw[gray,dashed](\pmax, 1. )--(-  .2,  1.        ) node[black,       left]{$ 1$};
 \draw[     dotted](  2. ,  . )--(  2. , \spmin     )
                               --(   . , \spmin     );
                                                     \node[      below  left]{$ 0$};
 \foreach \x in {1,...,10}
  \draw[gray      ](\x   ,  . )--(\x   ,-\ticklength) node[black,below      ]{$\x$};
\end{tikzpicture}
}
\caption{The spectral radius of $\mu_1$ on $\ell^p(V)$ as a function of $p\in[1,+\infty]$ is strictly increasing in $[2,+\infty]$ to its asymptotic value $1$ (and symmetric with respect to $1/p\leftrightarrow 1-1/p$), for $q_+=5$ and $q_-=2$. The minimum value is the $\ell^2$-spectral radius.}
\label{Fig:spectral_radius_for_p'=4.32}
\end{figure}
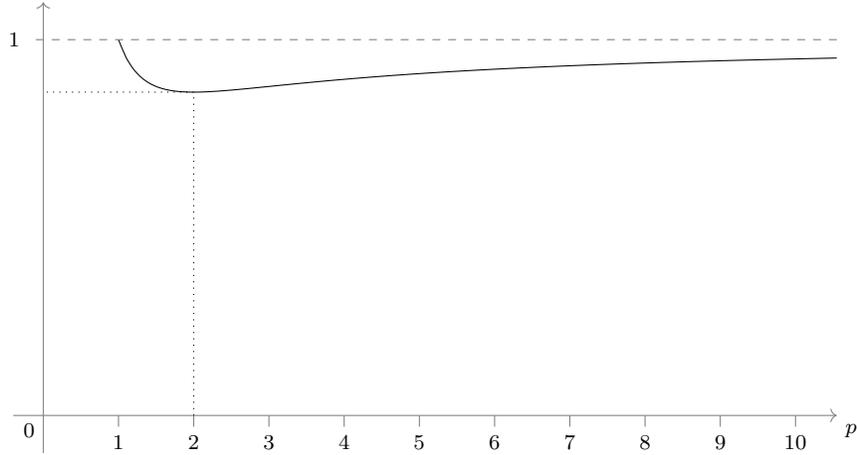

%
\section{Spherical functions on semi-homogeneous trees via difference equations}\label{Sec:spherical_functions_via_difference_equations}

For each eigenvalue $\gamma$ of the Laplace operator $\mu_1$, the corresponding spherical function was computed in Theorem \ref{theo:computation_of_semi-homogeneous_spherical_functions_via_Poisson_kernel} as a boundary integral of the generalized Poisson kernel at $\gamma$, explaining the connection with the Green kernel of $\mu_1$, that is its resolvent operator, and leading to its spectrum. This approach  makes use of analytic continuations of the generalized hitting probabilities $F^\pm(\gamma)$ introduced in Subsection \ref{SubS:first_visit_probabilities}, hence of heavy computations involving Markov chains, complex analysis and spectral theory of self-adjoint operators, carried out in Sections \ref{Sec:Poisson_representation}, \ref{Sec:Semi-homogeneous_generalized_Poisson_kernels}, \ref{Sec:Spherical_functions}, \ref{Sec:Spectrum_of_M2} and \ref{Sec:Spectrum_of_M1}.

In this Section we show how the expression of  the spherical functions proved in Theorem \eqref{theo:computation_of_semi-homogeneous_spherical_functions_via_Poisson_kernel}  can be obtained in a more direct but less enlightening way by means of their difference equation, that is, by convolution methods, without any reference to integral geometry or Poisson representation on the boundary. Indeed, the action of $\mu_1$ on radial functions around $v_0$ is expressed by the difference equation \eqref{eq:semihomogeneous_recurrence} in terms of an integer, the distance from $v_0$. At the end of Section \ref{Sec:Spectrum_of_M2} we determined the spectrum of $\mu_1$ on $\ell^p(V)$ via the spectra of $\mu_1^2$ on $\ell^p(V_+)$ and on $\ell^p(V_-)$, easier to compute because the homogeneity is constant. The operator $\mu_1^2$ is a step-$2$ operator on $V$, corresponding, respectively, to a step-$1$ operator on each of the subsets $V_+$ and $V_-$. We shall iterate the recurrence relation \eqref{eq:semihomogeneous_recurrence} of $\mu_1$ on $V$ to get independent relations on $V_+$ and $V_-$, respectively, that are simpler to solve.

If $f$ is a radial function around $v_0\in V_+$ write $f_n$ for $f(v)$ if $|v|=n$. With $\gamma\in\mathC$, by \eqref{eq:semihomogeneous_recurrence} such $f$ is a $\gamma$-eigenfunction of $\mu_1$, i.e., $\mu_1 f=\gamma f$ holds, if and only if
\begin{equation}\label{eq:recurrence_relations_for_vertex_spherical_functions_on_nonhomogeneous_trees}
 \gamma f_n=\begin{cases}
           f_   1                    &\text{for      $n=0$,}\\[.4cm]
\dfrac{q_- f_{n+1} +f_{n-1}}{q_- +1} &\text{for odd  $n  $,}\\[.4cm]
\dfrac{q_+ f_{n+1} +f_{n-1}}{q_+ +1} &\text{for even $n>0$.}
\end{cases}
\end{equation}
In order to translate this into a unique difference equation valid for both $n$ even and $n$ odd (although at the price of doubling the order) and thus be able to solve it by standard tools, we plug the expressions for $ \gamma f_{n+1}$ and $ \gamma f_{n-1}$ into that of $\gamma\cdot\gamma f_n $ to obtain
\begin{equation*}
\gamma^2 f_n =\dfrac{q_+ q_- f_{n+2} +(q_+ +q_-   )f_n +f_{n-2}}{(q_+ +1)(q_- +1)}
\quad\text{for all $n\geqslant 2$.}
\end{equation*}
This relation, surprisingly, turns out to be the same regardless of parity, in accord with the fact that the volume growth rates of $V_+$ and $V_-$ are the same, namely $q_+q_-$. From the normalization condition $f_0=1$ for spherical functions, we can also work out the explicit values of $f_1,f_2,f_3$ recursively from \eqref{eq:recurrence_relations_for_vertex_spherical_functions_on_nonhomogeneous_trees}.

Consequently, we obtain
\begin{equation}
\label{eq:general_nonhomogeneous_recurrence_relation_of_step_2}
\begin{cases}
f_0=1\\[.2cm]
f_1=\gamma\\[.2cm]
f_2=\bigl((q_- +1)\gamma^2 -1\bigr)/q_-,\\[.2cm]
f_3
   =\biggl(\dfrac{(q_+ +1)(q_- +1)(\gamma^2 -1)}{q_+ q_-} +1\biggr)\gamma,
   \\[.3cm]
f_k-(\newGamma-q_+ -q_- )f_{n+2} +q_+q_- f_{n+4}=0 \qquad\text{for every $n$.}
\end{cases}
\end{equation}
This is a linear problem, given by a recurrence equation of order four and as many independent initial conditions, therefore it admits a unique solution $\{f_n\colon n\geqslant 0\}$. Via the $Z$-transform we obtain the  characteristic polynomial equation 
\begin{equation}
\label{eq:quartic_equation_of_the_Z-transform}
1 -(\newGamma-q_+ -q_-) Z^2 +q_+q_-Z^4=0,
\end{equation}
that, by \eqref{eq:quartic_polynomial_Z}, Definition \ref{def:the_root_R} and \eqref{eq:F+(gamma)F-(gamma)}, is satisfied by
\begin{equation}
\label{eq:solutions_of_biquadratic_equations}
Z^2 =\frac{\qFF}{\qmean}=\frac1{\qmean\tildeqFF} \quad\text { or }\quad \widetilde Z^2 = \frac{\tildeqFF}{\qmean}= \frac1{\qmean\qFF}.
\end{equation}
There are four solutions $Z,-Z,\widetilde Z,-\widetilde Z$, all non-zero.

If $\qFF\neq \tildeqFF$, or equivalently, by \eqref{eq:F+(gamma)F-(gamma)},
if $\calW\neq 0$, then  $Z\neq\pm \widetilde Z$ and these four roots of the polynomial equation are pairwise distinct (non-degenerate case). Then it is well-known that \eqref{eq:general_nonhomogeneous_recurrence_relation_of_step_2} is solved by a linear combination of $Z^n$, $(-Z)^n$, $\widetilde Z^n$, $(-\widetilde Z)^n$ which can be better expressed in terms of suitable constants $c_e, \widetilde c_e,c_o,\widetilde c_o$ as

\begin{equation}\label{eq:the_shape_of_the_nonhomogeneous_vertex_spherical_function}
f_n=
\begin{cases}
c_e Z^ n    +\widetilde c_e \widetilde Z^ n   =c_e (Z^2)^{ n   /2} +\widetilde c_e (\widetilde Z^2)^{ n   /2} &\text{for $n$ even,}\\[.2cm]
c_o Z^{n-1} +\widetilde c_o \widetilde Z^{n-1}=c_o (Z^2)^{(n-1)/2} +\widetilde c_o (\widetilde Z^2)^{(n-1)/2} &\text{for $n$  odd.}
\end{cases}
\end{equation}
Imposing the four initial conditions 
we obtain by \eqref{eq:solutions_of_biquadratic_equations}
\begin{equation*}
\begin{aligned}
c_e &=\dfrac{(q_- +1)\gamma^2-1-q_-\widetilde Z^2}{q_- (           Z^2 -\widetilde Z^2)}
  =\dfrac{\bigl((q_- +1)\gamma^2-1\bigr)\qmean     \qFF-q_-}{q_-\bigl(     \qFF^2 -1 \bigr)}\;,\\[.2cm]
\widetilde c_e &=\dfrac{(q_- +1)\gamma^2-1-q_-           Z^2}{q_- (\widetilde Z^2 -           Z^2)}
  =\dfrac{\bigl((q_- +1)\gamma^2-1\bigr)\qmean\tildeqFF-q_-}{q_-\bigl(\tildeqFF^2 -1 \bigr)}\;,\\[.2cm]
c_o &=\dfrac{\newGamma-q_+ -q_- -1-\qmean^2 \widetilde Z^2}{\qmean^2(           Z^2 -\widetilde Z^2)}\,\gamma
  =\dfrac{\bigl(\newGamma-q_+ -q_- -1\bigr)     \qFF-\qmean}{\qmean\bigl(     \qFF^2-1\bigr)}\,\gamma,\\[.2cm]
\widetilde c_o &=\dfrac{\newGamma-q_+ -q_- -1-\qmean^2            Z^2}{\qmean^2(\widetilde Z^2 -           Z^2)}\,\gamma
  =\dfrac{\bigl(\newGamma-q_+ -q_- -1\bigr)\tildeqFF-\qmean}{\qmean\bigl(\tildeqFF^2-1\bigr)}\,\gamma.
\end{aligned}
\end{equation*}
By \eqref{eq:semi-homogeneous_spherical_function_as_Poisson_transform-2}, \eqref{eq:quadratica_equation_for_qFF} and \eqref{eq:B/qmean_in_terms_of_Fminus},
\begin{equation}
\begin{split}
           c_e&=          \kappa(\gamma,+1),\\[.2cm]
\widetilde c_e&=\widetilde\kappa(\gamma,+1),\\[.2cm]
           c_o&=          \kappa(\gamma,-1),\\[.2cm]
\widetilde c_o&=\widetilde\kappa(\gamma,-1).
\end{split}
\end{equation}
In particular, if $\gamma=0$ then  $Z^2=-1/\max\{q_+,q_-\}$ and $\widetilde Z^2=-1/\min\{q_+,q_-\}$, and it is easy to verify that 
\begin{equation*}
f_n=\begin{cases}
(-q_-)^{-n/2} &\text{for $n$ even,}\\[.2cm]
  0           &\text{for $n$  odd.}
\end{cases}
\end{equation*}
This is the same expression that we obtained in \eqref{eq:the_semi-homogeneous_spherical_function_of_eigenvalue_zero} by directly computing $\phi(\missarg,v_0\barra 0)$  as the solution of  the step-1 recurrence relation \eqref{eq:semihomogeneous_recurrence},
hence, equivalently, also of the step-2 recurrence relation \eqref{eq:general_nonhomogeneous_recurrence_relation_of_step_2} as done here. It also holds for homogeneous trees, where $\gamma=0$ is a degenerate case (see below).

If $\gamma$ is parametrized by $z\in\mathC$ as the two-valued map of \eqref{eq:the_semihomogeneous_eigenvalue_map}, then by \eqref{eq:B(gamma)=+-1_n_terms_of_z}
the values of $z$ that correspond to the non-degenerate case are exactly those such that $\qmean^{2z}\neq \qmean^{2(1-z)}$,
and, by \eqref{eq:gamma(z)=gammma(1-z)},
\eqref{eq:the_shape_of_the_nonhomogeneous_vertex_spherical_function} becomes 
\begin{equation}\label{eq:solution_of_recurrence_in_non-degenerate_case_in_terms_of_z}
f_n=\begin{cases}
 \dfrac{\qmean^{2(1-z)}-1}{\qmean^{2(1-z)}-\qmean^{2   z }}\cdot\dfrac{q_+ +\qmean^{2   z }}{q_+ +1}
\,\qmean^{-nz}
+\dfrac{\qmean^{2   z }-1}{\qmean^{2   z }-\qmean^{2(1-z)}}\cdot\dfrac{q_+ +\qmean^{2(1-z)}}{q_+ +1}
\,\qmean^{-n(1-z)}
&\text{for $n$ even,}\\[.8cm]
\biggl(
\dfrac{\qmean^{2(1-z)}-1}{\qmean^{2(1-z)}-\qmean^{2z}}\,
\,\qmean^{-(n-1)   z }
+\dfrac{\qmean^{2z}-1}{\qmean^{2z}-\qmean^{2(1-z)}}\,
\,\qmean^{-(n-1)(1-z)}\biggr) \gamma(z)
&\text{for $n$  odd.}
\end{cases}
\end{equation}
For each of the two values of $\gamma(z)$, this yields the solution $\{f_n\}$ of the recurrence equation at that eigenvalue. These two eigenfunctions  coincide on even vertices.
 Observe that, by \eqref{eq:gamma(z)=gammma(1-z)}, the expression for either parity is invariant under exchange of $z$ with $1-z$, and, for odd $n$ (unlike for even $n$) also under exchange of $q_+$ with $q_-$.
Finally, if $q_+=q_-=q$ the two expressions on the right-hand side coincide to give
\begin{equation*}
\begin{split}
f_n
&=\dfrac{q^{2(1-z)}-1}{q^{2(1-z)}-q^{2   z }}\cdot\dfrac{q+q^{2   z }}{q +1}
q^{-nz}
+\dfrac{q^{2   z }-1}{q^{2   z }-q^{2(1-z)}}\cdot\dfrac{q+q^{2(1-z)}}{q +1}
q^{-n(1-z)}\\[.2cm]
&=c(z)\, q^{-nz}+c(1-z)\,q^{-n(1-z)}
\quad\text{for any $n$,}
\end{split}
\end{equation*}
as in  \eqref{eq:homogeneous_spherical_function}.

Let us now consider the degenerate case, that is, 
$\qFF = \tildeqFF$.
Then
 the solution of the recurrence equation \eqref{eq:general_nonhomogeneous_recurrence_relation_of_step_2} is a linear combination of $Z^n$, $(-Z)^n$, $nZ^n$, $n(-Z)^n$, that, as before, can be conveniently expressed via appropriate constants $c_e,c_o,c'_e,c'_o$ as
\begin{equation}\label{eq:the_shape_of_the_nonhomogeneous_vertex_spherical_function,degenerate}
f_n=
\begin{cases}
(c_e +nc'_e)Z^n=(c_e +nc'_e) (Z^2)^{ n   /2} &\text{for $n$ even,}\\[.2cm]
(c_o +nc'_o)Z^{n-1}=(c_o +nc'_o)(Z^2)^{(n-1)/2} &\text{for $n$  odd.}
\end{cases}
\end{equation}
The degenerate case occurs
if $\qFF= 1$ or $\qFF=-1$.
If $\qFF=1$ we have $Z^2=\widetilde Z^2=1/\qmean$ and $\gamma=\pm b$ by \eqref{eq:solutions_of_biquadratic_equations} and
  \eqref{eq:B(gamma)=+1};  the corresponding values of $z$ are as in \eqref{eq:B(gamma)=+-1_n_terms_of_z}.
Then the four initial conditions 
yield 
\begin{equation*}
\begin{aligned}
 c_e &=1,
&c_o &=\biggl(\frac12 + \dfrac1{2\qmean}\biggr) (\pm b),\\[.2cm]
 c'_e &=\dfrac12 \biggl( -1 + \sqrt{\frac{q_+}{q_-}}\,\bigl((q_- +1)b^2  -1\bigr)\biggr),
&c'_o &= \biggl(\frac12 - \dfrac1{2\qmean}\biggr) (\pm b),
\end{aligned}
\end{equation*}
whence  \eqref{eq:B(gamma)=+1:spherical_function}.
Instead, if $\qFF=-1$ we have $Z^2=\widetilde Z^2=-1/\qmean$, $\gamma=\pm a$ \eqref{eq:B(gamma)=-1},
and
\begin{equation*}
\begin{aligned}
 c_e &=1,
&c_o &=\biggl(\frac12 - \dfrac1{2\qmean}\biggr) (\pm a),\\[.2cm]
 c'_e &=\dfrac12 \biggl( -1- \sqrt{\frac{q_+}{q_-}}\,\bigl((q_- +1)a^2  -1\bigr)\biggr),
&c'_o &= \biggl(\frac12 + \dfrac1{2\qmean}\biggr) (\pm a),
\end{aligned}
\end{equation*}
that yields
 \eqref{eq:B(gamma)=-1:spherical_function}. Therefore the solutions of the recurrence equation in the degenerate  case  are in agreement with Theorem \ref{theo:computation_of_semi-homogeneous_spherical_functions_via_Poisson_kernel}.

Observe that $\gamma=0$ is a degenerate case if and only if $q_+=q_-$, i.e., if $T$ is homogeneous. Conversely, if $T$ is homogeneous. then, by \eqref{eq:B(gamma)=+1} and \eqref{eq:B(gamma)=-1}, the degenerate case occurs if and only if $\gamma$ vanishes (then 
$\qFF=1$) or it equals $2\sqrt q/(q+1)$ or its opposite (then 
$\qFF=-1$), the two endpoints of the spectrum of $\mu_1$ on $\ell^2(V)$.

\section{List of abbreviations}
{\it 
\begin{itemize}
\item $F^\pm(\gamma)$, \qquad the determinations that vanish at infinity of the analytic continuations of the two one-step generalized transition probabilities, as functions of the eigenvalue $\gamma$,
\qquad \eqref{eq:n.n._semihomogeneous_visit_probabilities_at_eigenvalue_gamma}
\item  $\widetilde F^\pm(\gamma)$, \qquad the other determinations, diverging at infinity,  of the analytic continuations of the solution of the algebraic equation of generalized transition probabilities\qquad \eqref{eq:n.n._semihomogeneous_visit_probabilities_at_eigenvalue_gamma-not_acceptable}
\item
$\qFF =\qmean\,F^+(\gamma)\,F^-(\gamma)$, \qquad \eqref{eq:definition_of_qFF}, see also \eqref{eq:F+(gamma)F-(gamma)}
\item
$\tildeqFF=\qmean\widetilde F^+(\gamma)\widetilde F^-(\gamma)$,
\item
$\newGamma =(q_+ +1)(q_-+1)\gamma^2$,
\qquad \eqref{eq:definition_of_newGamma}
\item
$\calW=\newGamma^2-2(q_+ +q_-)\newGamma+(q_+ - q_-)^2$,\qquad \eqref{eq:quartic_polynomial_Z}
\item
$\theroot=\sqrt{\calW}$, \qquad Definition \ref{def:the_root_R}
\end{itemize}
}


\end{document}